\documentclass[11pt]{amsart}

\usepackage{amsmath, amsthm, amssymb, amsfonts, enumerate}
\usepackage[colorlinks=true,linkcolor=blue,urlcolor=blue]{hyperref}
\usepackage{dsfont}
\usepackage{color}
\usepackage{geometry}
\usepackage{todonotes}

\usepackage{graphicx}
\usepackage{caption}

\usepackage{float}

\usepackage{mathtools}
\usepackage{soul}

\mathtoolsset{showonlyrefs}

\def \cA{\mathcal{A}}
\def \cC{\mathcal{C}}
\def \cD{\mathcal{D}}
\def \cB{\mathcal{B}}

\def \cS{\mathcal{S}}

\def \cF{\mathcal{F}}
\def \cG{\mathcal{G}}
\def \cJ{\mathcal{J}}

\def \P{\mathsf P}

\def \E{\mathsf E}

\def \N{\mathbb{N}}
\def \R{\mathbb{R}}

\def \ud{\mathrm{d}}
\def \e{\mathrm{e}}

\newcommand{\eps}{\varepsilon}

\newtheorem{theorem}{Theorem}[section]
\newtheorem{lemma}[theorem]{Lemma}

\newtheorem{proposition}[theorem]{Proposition}
\newtheorem{definition}[theorem]{Definition}
\newtheorem{remark}[theorem]{Remark}

\title[Nash equilibria for dividend distribution with competition ]{Nash equilibria for dividend distribution with competition}
\author[De Angelis, Gensbittel, Villeneuve]{Tiziano De Angelis, Fabien Gensbittel, St\'ephane Villeneuve}
\subjclass[2010]{91A15, 93E20, 91G50, 60J60, 35R35}
\keywords{Singular controls, nonzero-sum games, Nash Equilibrium, dividend problem, free boundary problems, randomised strategies}
\address{T.\ De Angelis: School of Management and Economics, Dept ESOMAS, University of Torino, Corso Unione Sovietica, 218 bis 10134, Torino, Italy; Collegio Carlo Alberto, Piazza Arbarello 8, 10122, Torino, Italy}
\email{\href{mailto:}{tiziano.deangelis@unito.it}}
\address{F.\ Gensbittel: Toulouse School of Economics, 1 esplanade de l’université 31000 Toulouse, France.}
\email{\href{mailto:}{fabien.gensbittel@tse-fr.eu}}
\address{S.\ Villeneuve: Toulouse School of Economics, 1 esplanade de l’université 31000 Toulouse, France.}
\email{\href{mailto:}{stephane.villeneuve@tse-fr.eu}}
\date{\today}

\numberwithin{equation}{section}

\begin{document}

\begin{abstract}
We construct Nash equilibria in feedback form for a class of two-person stochastic games of singular control with absorption, arising from a stylized model for corporate finance.
More precisely, the paper focusses on a strategic dynamic game in which two financially-constrained firms operate in  
the same market. The firms distribute dividends and are faced with default risk. The strategic interaction arises from the fact that if one firm defaults, the other one becomes a monopolist and increases its profitability. The firms choose their dividend distribution policies from a class of randomised strategies and we identify two types of equilibria, depending on the firms' initial endowments. In both situations the optimal strategies and the equilibrium payoffs are found explicitly. 
\end{abstract}

\maketitle

\section{Introduction}

\subsection{Motivation.} 
In this paper we consider a 2-player nonzero-sum stochastic game that arises from a stylized model of dividend distribution for two competing firms. We build on the mathematical framework of De Finetti’s classical dividend problem \cite{DeFinetti} which was formulated as a stochastic control problem by Jeanblanc and Shiryaev \cite{JS} and Radner and Shepp \cite{RS}. 

In \cite{JS} and \cite{RS} a firm's capital evolution is described by an arithmetic Brownian motion (aBm) which, for the sake of argument, we denote by $(X_t)_{t\ge 0}$. The constant drift represents the firm’s profitability per unit time. The Brownian motion carries the uncertainty. External financing is costly, thus creating a precautionary demand for cash. The firm's manager decides on the distribution of dividends to the shareholders with the goal of maximizing the total discounted amount of dividends paid until the firm's default. The firm lives until its default time so that there is a tradeoff between maximizing dividend payments at each time and keeping the firm alive for as long as possible.

The dividend problem is perhaps the most popular application of stochastic control theory in corporate finance. One of the reasons is that it yields a stock price model which endogenises the firm's valuation---denoting $\hat v$ the value function of the control problem, the stock price dynamics reads $S_t=\hat v(X_t)$. This is not the case, for example, in the celebrated Black and Scholes model, where the stock price dynamics is given exogenously with no real connection to the firm's financial performance.

There exists an extremely vast literature around the dividend problem, which covers, for example, more general jump-diffusive dynamics of $(X_t)_{t\ge 0}$, multi-dimensional models with stochastic interest rates, constraints on the admissible dividend distributions, cash injections, etc. Accounting for these developments is difficult because of the rapidly expanding field and it falls outside the scope of our work. We refer the reader to the review paper \cite{Avanzi} for a detailed overview of the field until 2010 and to the introduction of \cite{BDeA} for an overview of some recent results. 

From our point of view, an obvious limitation of the classical dividend distribution model is that a single firm operates in isolation in a market with no other economic agent. Our paper is 
the first one to study how the presence of a competitor will impact the optimal dividend distribution when two firms interact strategically.

\subsection{A summary of the model and economic insight.}

We consider two identically-efficient firms acting on a single-good market in which the demand for the good is random. Both firms have a capital evolution which is driven by the same aBm (i.e., with the same drift and diffusion coefficients). Each firm's manager chooses how to distribute dividends to the shareholders and dividends are subtracted from the firms' capital. Each firm lives until its default time. 

Firms {\em are not} identical, because they may have different initial endowments and they can choose different dividend policies, say $(L_t)_{t\ge 0}$ and $(D_t)_{t\ge 0}$. In particular, we denote by $(X^L_t)_{t\ge 0}$ and $(Y^D_t)_{t\ge 0}$ the capital net of dividend payments for the first and the second firm, respectively. These are the firms' controlled dynamics. Although all ingredients will be formally introduced in Section \ref{sec:setup}, it is worth mentioning here that we allow for the most general class of dividend distribution policies by considering {\em singular} controls (i.e., non-decreasing, right-continuous processes). 

At time zero, the two firms are in duopoly but if/when one of the two firms defaults, the surviving firm becomes a monopolist with an increase of its profitability. 
We model the duopoly/monopoly transition with a change of drift in the aBm of the surviving firm: the drift changes from an initial value $\mu_0$ to a larger value $\hat \mu$. In this context, each firm's manager must decide how to optimally distribute dividends in order to maximize the total discounted amount of dividends paid until the default time of one of the two firms. The prospect of becoming a monopolist and the presence of a rival exacerbate the standard trade-off between exerting controls (paying dividends) and keeping a high level of cash reserves. 

Despite the perceived simplicity of our model (firms' capitals have the same dynamic evolution in the absence of dividend payments) we observe non-trivial deviations from the optimal dividend policy of the De Finetti’s dividend model. In particular, we find a Nash equilibrium in which at least one of the two firms' dividend distribution policy depends on a {\em moving} threshold. The value of the threshold depends on the relative amount of capital held by one firm as compared to the other. This is in contrast with the classical solution of the single-firm dividend problem, where the threshold is fixed. Moreover, our equilibrium shows that the firm with the largest initial endowment may use this initial asymmetry to its own advantage and induce the poorer firm to forego the option of becoming monopolist. 
In a nutshell, we conclude that cash-rich firms are less likely to pay dividends, because their shareholders have more interest in waiting for competitors to go bankrupt than in receiving early dividends. This observation connects us to an emerging literature concerned with the study of cash-rich firms who may engage in predatory strategies to drive poorer rivals out of the market and thus benefit from monopolistic profits (particularly in the digital economy). In the deep-pocket theory of predation (see \cite{Kobayashi} for a general presentation), predatory behavior may arise when a firm adopts a strategy intended to induce the exit from the market of a financially constrained competitor by depletion of its resources. The Nash equilibrium we find in our model goes in this direction. 

We do not claim that our Nash equilibrium is unique but it holds for all choices of the model's parameters (positive drift and diffusion coefficients of the aBm, initial firms' endowments, discount rate). Moreover, the dividend distribution policies at equilibrium and the firm's equilibrium payoffs are essentially explicit (see Theorem \ref{thm:NE1} and Proposition \ref{prop:SC}).

\subsection{A summary of our mathematical results.}

We construct Nash equilibria for a two-player nonzero-sum game of singular controls with an exogenous absorbing boundary for the state-dynamics of each player (absorption is due to default). The game is played in continuous time and the underlying dynamics is driven by a Brownian motion. We assume that players have complete information about the dynamics of the system (including the initial states and all parameters involved), about the class of admissible controls and about the game's payoffs. We also assume that each player can observe her opponent's actions. Players' admissible strategies are drawn from a class of randomised strategies introduced in Definition \ref{def:random_strategies}. 

In Section \ref{sec:strategies}, we carefully relate randomised (and pure) strategies to randomised (and pure) controls and to the resulting payoffs. Not all pairs of randomised strategies produce a well-defined dynamics of our system and therefore we need to introduce a notion of {\em control-inducing} pairs (cf.\ Definition \ref{def:theta_r}), related to the existence of a suitable fixed point in the space of paths.

We find two types of Nash equilibria, depending on the initial endowment/state of the two firms/players. Both equilibria are constructed explicitly relying upon free boundary methods and stochastic calculus. When the initial endowments of the two firms are different, we solve two interconnected free boundary problems. The free boundary for the ``poorer'' firm is constant whereas the free boundary of the ``richer'' firm moves with the state-variable associated to the other firm's level of capital (the free boundary is monotone and its inverse moves with the difference in the amount of capital held by the two firms). Along the equilibrium trajectory, the poorer firm acts as if it were alone in the single-good market by following the classical optimal policy from \cite{JS} and \cite{RS}. The richer firm instead controls the level of its cash reserves in order to stay ahead of its rival, making sure that the other firm defaults first. It is important to notice that, although we consider deviations from equilibrium in the general class of randomised strategies, at equilibrium the richer firm uses a {\em pure strategy} while the poorer firm uses a {\em pure control} (cf.\ Definitions \ref{def:random_strategies} and \ref{def:random_controls}, respectively). We emphasize that this structure is {\em not} due to some sort of restriction that we impose on players' action sets but rather it is one particular instance of a game in randomised strategies that admits an equilibrium in a smaller class of strategies.   

The use of randomised strategies is instead fundamental in order to construct a symmetric equilibrium when the two firms have the same initial endowment. The firms are completely symmetric at time zero. As soon as one of them distributes dividends, the symmetry is broken and the game is back into the previous asymmetric situation. The firm that makes the first dividend payment effectively accepts to be in a disadvantaged position compared to the other one. However, waiting forever will eventually lead both firms to default and yield zero payoff for both. In this context, we construct an equilibrium in which firms start making dividend payments at a randomised stopping time and we characterise explicitly the optimal intensity of stopping. After the first dividend payment is performed by one of the two players, further actions of each player are obtained following ideas similar to those from the asymmetric setting.

\subsection{Our contribution to the literature.}
 
There exists an abundant literature on single-agent singular control problems, dating back to seminal work by, e.g., Bather and Chernoff \cite{BaCh}, Benes et al.~\cite{Be}, Karatzas \cite{Ka} and many others. As already explained, the two papers \cite{JS} and \cite{RS} set the benchmark case for the analysis of corporate cash management in continuous-time, based on the work of De Finetti's \cite{DeFinetti}. Starting from this framework we investigate the impact of introducing competition in problems of singular control with absorption. 

From a mathematical perspective, the literature on nonzero-sum stochastic games of singular control is still in its infancy. Kwon and Zhang \cite{Kw} find Markov perfect equilibria in a game of competitive market share control, in which each player can make irreversible investment decisions via singular controls as well as deciding to strategically exit the market. 
De Angelis and Ferrari \cite{DeAFe2} and Dammann et al.\ \cite{DRV} obtain Nash equilibria in the class of Skorokhod-reflection policies for a nonzero-sum game where two players control the same one-dimensional state dynamics. In \cite{DeAFe2} an equilibrium is found by establishing a connection between the nonzero-sum game of monotone controls and a nonzero-sum stopping game. In \cite{DRV} an equilibrium is found by solving a free boundary problem. There are two important differences of our work compared to \cite{DeAFe2}, \cite{DRV} and \cite{Kw}. First, the two players in those papers control the same dynamics and no default may occur (in \cite{Kw} players may decide to exit the game and therefore the transition from duopoly to monopoly occurs {\em only} because of optimality considerations). In our paper instead each player controls her own dynamics and default may occur also in the absence of control (actually, controlling will increase the probability of default). Second, each player in \cite{DeAFe2} and \cite{DRV} chooses a point on the real line and exerts control in order to reflect the dynamics at that point (one player pushes the dynamics upwards and the other one pushes it downwards); equilibria in \cite{Kw} are sought in a class of {\em barrier} strategies, which is close in spirit to the classical Skorokhod reflection. Instead, we have 2-dimensional dynamics and we {\em do not} restrict our class of strategies to those triggered by thresholds (i.e., we determine our Nash equilibria by allowing deviations in a general class of randomised strategies).

Another related paper is by Ekstr\"om and Lindensj\"o \cite{EL}. They study an $N$-player competitive game in an extraction problem from a common resource with Brownian dynamics. Ekstr\"om and Lindensj\"o find a Nash equilibrium for a class of Markov strategies of bang-bang type with regular controls. That is, each player extracts at the maximum (bounded) rate when the controlled dynamics is above a certain critical value. Since the game is symmetric, it turns out that all players act simultaneously (i.e., they all choose the same critical value). All players control the same dynamics and therefore they all default at the same time once the resources are depleted. Again, our setup is different because the dynamics in \cite{EL} is 1-dimensional and equilibria are sought in the class of threshold policies.

Finally, Steg \cite{Steg-a} studied $N$-player games of irreversible investment in the context of capital accumulation (without default). He finds equilibria in {\em open-loop} strategies, in the sense that players choose their investment policies based only on information contained in a commonly observed filtration (e.g., that generated by a commonly observed stochastic process). Differently from our notion of strategy (Definition \ref{def:random_strategies}), in \cite{Steg-a} players do not react to possible deviations of their opponents from equilibrium trajectories. They only play what we call {\em pure controls} in Definition \ref{def:random_controls}. Moreover, in his Introduction, Steg observes that ``{\em even to specify sensible feedback strategies poses severe conceptual problems}'' in games with singular controls (this was brought up by Back and Paulsen \cite{BP} with reference to an earlier paper by Grenadier \cite{Gr}). In this sense, our paper (and in particular Section \ref{sec:strategies}) provides a rigorous framework for the study of stochastic games of singular controls with feedback strategies.

\subsection{Structure of the paper.}

The paper is organised as follows. In Section \ref{sec:setup} we set up the problem and recall some useful facts about the classical dividend problem. In Section \ref{sec:strategies} we introduce the class of admissible actions for the two firms and the associated payoffs. We study in detail the relationship between the notions of randomised/pure strategies/controls, control-inducing pairs and Nash equilibria (cf., in particular, Lemma \ref{lem:deviations}). In Section \ref{sec:NEasym} we construct an equilibrium for the game with firms having different initial endowments. In Section \ref{sec:sym} we construct a symmetric equilibrium for firms with the same initial endowment. In Section \ref{sec:conclusions} we briefly discuss limitations and possible extensions of our model. A short Appendix with a small technical result and a summary of frequently used notations completes the paper.

\section{Problem setting.}\label{sec:setup}
We have two firms operating on the same market and whose cash reserves increase at a rate $\mu_0>0$ but are subject to a volatility $\sigma>0$. A firm defaults if its cash reserves drop below zero. In that case, the surviving firm becomes a monopolist, resulting in a higher rate of increase of its cash reserve, i.e., $\hat \mu>\mu_0$. For simplicity we keep the same volatility also for the monopolist but, as it will become clear later, a change of volatility brings no substantial difference in our analysis.

More formally, we consider the space $\Omega=C_0([0,\infty))$ of continuous functions $\varphi:[0,\infty)\to\R$ with $\varphi(0)=0$  endowed with the $\sigma$-algebra $\cF$ generated by all finite-dimensional cylinders (cf., e.g., \cite[Ch. 2.2]{KS}) and the Wiener measure $\P$. The canonical process $(B_t)_{t\ge 0}$ on $\Omega$ is a standard $1$-dimensional Brownian motion, and $(\cF_t)_{t\ge 0}$ with $\cF_t\coloneqq\sigma(B_s, 0\leq s \leq t)$ is the canonical {\em raw} filtration. 
We denote by $(X_t)_{t\ge 0}$ and $(Y_t)_{t\ge 0}$ the cash reserve dynamics of the first and second firm, respectively. Then, for $t\ge 0$ and $x,y\in[0,\infty)$, we have
\begin{equation}\label{eq:XY}
X_t=x+\mu_0 t+\sigma B_t -L_t\quad\text{and}\quad Y_t=y+\mu_0 t+\sigma B_t -D_t,
\end{equation}
where $L_t$ is the cumulative amount of dividends paid by the first firm up to time $t$, and $D_t$ is the analogue for the second firm. We will often use $X^L$ and $Y^D$ to emphasise the dependence of the processes on their controls.

As explained in the Introduction, we are interested in understanding predatory strategic behavior. That behaviour would typically require a firm to be able to react to their competitor's actions. For that reason, we assume that both firms can observe the Brownian process $B$ and the cumulative amount of dividends paid by the other firm up to time $t$. As a result, each firm can compute the cash reserve of the other firm. We also assume that each firm can act strategically and adjust dynamically the dividend payments in reaction to both the Brownian fluctuations and, crucially, to the past dividend payments made by the other firm. Moreover, firms can randomize their dividend payments across different strategies. Before rigorously defining randomised strategies, which are used to determine how the firms select the processes $L$ and $D$, we first define the class from which processes $L$ and $D$ are actually drawn. The formal definition of randomised strategies together with the associated payoffs will be given in Section \ref{sec:strategies}.

\begin{definition}[Admissible controls]\label{def:div}
A pair of processes $(L_t,D_t)_{t\ge 0}$ is called a {\em pair of admissible controls} if
$L$ and $D$ are non-decreasing, adapted to $(\cF_t)_{t\ge 0}$ and right-continuous, with\footnote{Here we formally denote by $L_{0-}$ and $D_{0-}$ the values of the processes before a possible jump at time zero.} $L_{0-}=D_{0-}=0$. Moreover, letting $(x)^+\coloneqq \max\{0,x\}$, it must hold
\begin{equation}\label{eq:def0}
L_t-L_{t-}\le (X^L_{t-})^+\quad\text{ and }\quad D_t-D_{t-}\le (Y^D_{t-})^+,\quad\text{for all $t\ge 0$, and all $\omega \in \Omega$.}
\end{equation}
\end{definition}
\medskip

Condition \eqref{eq:def0} ensures that the firms cannot pay dividends in excess of their cash reserve. 
We denote default times by $\gamma_X$ and $\gamma_Y$, with\footnote{Condition \eqref{eq:def0} could be replaced by the weaker condition $X^L_{\gamma_X}=0$ on $\{ \gamma_X<\infty\}$, or equivalently $L_t-L_{t-}\le X^L_{t-}$ for all $t\leq \gamma_X$, as only the trajectory up to $\gamma_X$ is relevant. Our choice is motivated by Definition \ref{def:random_strategies} of  strategies as functionals on the canonical space, in which for simplicity we avoid to define strategies only up to a stopping time.}
$\gamma_X=\inf\{t\ge 0: X^L_t\le 0\}$ and $\gamma_Y=\inf\{t\ge 0: Y^D_t\le 0\}$.
Finally, we denote $\P_{x,y}(\,\cdot\,)=\P(\,\cdot\,|X_{0-}=x,Y_{0-}=y)$.

Given a pair $(L,D)$ of admissible controls, the expected payoffs $\cJ^1$ and $\cJ^2$, for the first and second firm, respectively, read
\begin{equation}\label{eq:J12}
\begin{aligned}
&\cJ^1_{x,y}(L,D):=\E_{x,y}\Big[\int_{[0,\gamma_X\wedge\gamma_Y]}\e^{-r t}\ud L_t+1_{\{\gamma_Y<\gamma_X\}}\e^{-r\gamma_Y}\hat v(X^L_{\gamma_Y})\Big],\\
&\cJ^2_{x,y}(D,L):=\E_{x,y}\Big[\int_{[0,\gamma_X\wedge\gamma_Y]}\e^{-r t}\ud D_t+1_{\{\gamma_X<\gamma_Y\}}\e^{-r\gamma_X}\hat v(Y^D_{\gamma_X})\Big].
\end{aligned}
\end{equation}
Here, $r>0$ is a discount rate and $\hat v$ is the value function of the classical dividend problem with cash reserves growing at the rate $\hat \mu$. That, is
\begin{equation}\label{eq:JS}
\hat v(x):=\sup_{\xi}\E_{x}\Big[\int_{[0,\gamma_C]}\e^{-rt }\ud \xi_t\Big],
\end{equation}
with underlying dynamics given by 
$C^\xi_t=x+\hat \mu t+\sigma B_t-\xi_t,\quad t\ge 0$,
and with $\gamma_C=\inf\{t\ge 0: C^\xi_t\le 0\}$. The supremum is taken over all admissible controls, according to Definition \ref{def:div}. Here we emphasise the dependence of the problem's structure on the drift of the underlying process, because we will later use results for the dividend problem when the drift is either $\mu_0$ or $\hat \mu$. In particular, we will use the notation $\hat v(x)=w(x;\hat \mu)$ and $C^\xi=C^{\hat \mu;\xi}$, when convenient. An account of useful facts about the classical dividend problem will be provided below in Section \ref{sec:divid}.

The integral term in each one of the two payoffs in \eqref{eq:J12} is the discounted value of the cumulative dividends paid by the firm until both firms are active. At the (random) time $\gamma_X\wedge\gamma_Y$ one or both of the two firms goes bankrupt and the surviving firm becomes a monopolist with a larger cashflow rate $\hat \mu$. 
Notice that if $\gamma_X=\gamma_Y$ then no firm survives and the continuation payoff for both players is zero. This is induces no loss of generality because the monopolist's payoff is equal to zero when the initial cash reserve is zero, i.e., $\hat v(0)=0$.
If instead $\gamma_X<\gamma_Y$, the remaining firm (i.e., firm $2$) at time $\gamma_X$ is faced with the classical dividend problem but with an initial cash reserve $Y^D_{\gamma_X}$. Hence, the payoff at time $\gamma_X$ reads $\hat v(Y^D_{\gamma_X})$. Analogous considerations justify the continuation payoff $\hat v(X^L_{\gamma_Y})$ for firm $1$ in the event $\gamma_Y<\gamma_X$.

In what follows we will unambiguously refer to the first and second firm as first and second player, respectively.
By the symmetry of the set-up it is clear that the player with the largest initial cash reserve has an advantage on her opponent with a lower risk of being in liquidation ex-ante. We will show that this allows a rather explicit construction of a Nash equilibrium. Furthermore, in the completely symmetric situation in which $x=y$, the use of randomised strategies will be key to the construction of a {\em symmetric} equilibrium.

\subsection{Useful facts about the classical dividend problem.}\label{sec:divid}
Here we recall a few well-known results concerning the classical dividend problem for a generic drift $\mu>0$ of the cash reserve. We start by introducing notations which will be used throughout the paper. 
Given a set $A\subset\R^2$ (or $A\subset\R$) we denote its closure by $\overline A$. Given a function $f:\R^2\to\R$ and open set $A\subset\R^2$ (or $f:\R\to\R$ and $A\subset\R$) we write $f\in C^k(A)$ for $k\in\N$ to indicate that $f$ is $k$ times continuously differentiable in $A$. We write $f\in C^k(\overline A)$ to indicate that the function $f$ with all its $k$ derivatives admit a continuous extension to the boundary $\partial A$. Given two open sets $A\subsetneq B$ in $\R^2$ (or in $\R$), letting $E:=B\setminus \overline A$, for $k\in\N$ we use the notation $f\in C^k(\overline A\cup \overline E)$ to indicate that $f\in C^k(\overline A)\cap C^k(\overline E)$ with derivatives which may be discontinuous across the boundary $\partial A$. 

In the notation of \eqref{eq:JS} we consider a generic value function $w(x;\mu)$ when the underlying dynamics $C^{\mu;\xi}$ has drift $\mu>0$ (so that for \eqref{eq:JS} we have $\hat v(x)=w(x;\hat \mu)$). All the results listed here can be found, for instance, in \cite[Ch.\ 2.5.2]{Schmidli}. 

It is well-known that the optimal dividend policy in the classical dividend problem is of the form
\begin{equation}\label{eq:xi*}
\xi^*_t=\xi^*_t(\mu):=\sup_{0\le s\le t} \Big(x-a_*+\mu s+\sigma B_s\Big)^+,\quad \xi^*_{0-}=0,
\end{equation}
where $a_*=a_*(\mu)$ is an optimal boundary and dividends are paid so that the cash reserve process $(C^{\mu;\xi^*}_t)_{t\ge 0}$ is reflected downwards at $a_*$. The solution is generally constructed by showing that the value function $w$ belongs to the class $C^2([0,\infty))$ and that the pair $(w,a_*)$ is the unique solution of the free boundary problem
\begin{equation}\label{eq:fbpdiv}
\begin{aligned}
&\tfrac{\sigma^2}{2}w''(x;\mu)+\mu w'(x;\mu)-rw(x;\mu)=0,\quad x\in(0,a_*(\mu)),\\
&\tfrac{\sigma^2}{2}w''(x;\mu)+\mu w'(x;\mu)-rw(x;\mu)\le 0,\quad x\in[a_*(\mu),\infty),\\
&w'(x;\mu)\ge 1\:\text{for all}\: x\in[0,\infty),\\
&w'(x;\mu)> 1\iff x\in(0,a_*(\mu)),\\
&w(0;\mu)=0.
\end{aligned}
\end{equation}
The value of the optimal boundary $a_*(\mu)$ is determined by the smooth-pasting condition 
\begin{equation}\label{eq:smoothp}
w''(a_*(\mu);\mu)=0
\end{equation} 
and it can be calculated explicitly as $a_*(\mu)=2(\beta_1-\beta_2)^{-1}\log( -\beta_2/\beta_1)$ where $\beta_1=\beta_1(\mu)>0>\beta_2(\mu)=\beta_2$ are the two roots of the equation $\frac{\sigma^2}{2}\beta^2+\mu\beta-r=0$. In order to simplify the notation, we omit the dependence on $\mu$ from $w$ and $a_*$ when no confusion shall arise. 

For $x\in(0,a_*)$ the expression for $w$ reads
\begin{equation}\label{eq:w}
w(x)=C\big(e^{\beta_1x}-e^{\beta_2x}\big),
\end{equation}
where $C=C(\mu)>0$ is a constant that can be determined explicitly. 
Finally, we notice that the conditions $w'(a_*)=1$ and $w''(a_*)=0$ and the first equation in \eqref{eq:fbpdiv} imply $rw(a_*)=\mu$. Since $w$ is non-decreasing, then
\begin{equation}\label{eq:wnondecr}
rw(x)>\mu\iff x\in(a_*,\infty).
\end{equation}
Moreover, the condition $w'(x)=1$ for $x\ge a_*$ leads to $w(x)=(x-a_*)+w(a_*)$ for $x\ge a_*$.
Then, simple algebra yields, for $x\in[0,\infty)$,
\begin{equation}\label{eq:Lw}
\tfrac{\sigma^2}{2}w''(x)+\mu w'(x)-rw(x)=-r[x-a_*]^+=-[rw(x)-\mu]^+,
\end{equation}
where $[p]^+:=\max\{p,0\}$.

There are two particular values of $\mu$ which will crop up in our analysis below, i.e., $\mu=\mu_0$ and $\mu=\hat \mu$, corresponding to the drift for the duopoly and for the monopoly, respectively. Then we denote
\begin{equation}\label{eq:dnot}
\begin{aligned}
\hat v(x):=w(x;\hat \mu),\quad &v_0(x):=w(x;\mu_0),\quad \hat a:=a_*(\hat \mu),\quad a_0:=a_*(\mu_0)\\
&\hat \xi_t:=\xi^*_t(\hat \mu)\quad\text{and}\quad \xi^0_t:=\xi^*_t(\mu_0).
\end{aligned}
\end{equation}

\section{Strategies and equilibria.}\label{sec:strategies}
We introduce randomised strategies and define the associated non-zero sum game. All the equilibria that we construct in the subsequent sections will be Nash equilibria for such a game. 

For a proper definition of strategy, in addition to the class $C_0([0,\infty))$ we need the class $D_0^+([0,\infty))$ of right-continuous non-decreasing functions $\zeta:[0,\infty)\to[0,\infty)$ with $\zeta(0-)=0$. 
We introduce the canonical space $C_0([0,\infty))\times D_0^+([0,\infty))$ equipped with the Borel $\sigma$-algebra. The coordinate mapping on the canonical space is denoted ${\mathbb W}_t(\varphi,\zeta):=(\varphi(t),\zeta(t))$ for any $(\varphi,\zeta)\in C_0([0,\infty))\times D_0^+([0,\infty))$ and $t\in[0,\infty)$. Its raw filtration is denoted $(\cF^{\mathbb W}_t)_{t\ge 0}$.
 
Notice that the space $(\Omega,\cF)$ introduced in the previous section is a subspace of the canonical space and $\varphi=B$ is the first component of the coordinate mapping. 
An element $(\varphi,\zeta)$ of the canonical space has to be interpreted as the pair formed by a realized trajectory of the process $B$ and a realized trajectory of a control, say $D$ for Player 2 (resp.\ $L$ for Player 1). 
That pair is the information available to Player 1 (resp.\ Player 2) during the game and a strategy will be a map from the canonical space to the set of control trajectories. There is no natural reference probability on $D_0^+([0,\infty))$, because the trajectory of the control of a player can be chosen freely. As a consequence there is no natural reference probability on the canonical space. For that reason, our analysis of strategies will be performed pathwise and the connection with the probabilities will appear only when defining the payoffs of the game. 

We say that a mapping
$\chi:[0,\infty)\times C_0([0,\infty))\times D_0^+([0,\infty))\to \R$
is {\em non-anticipative} if $\chi(t,\varphi,\zeta)$ is $\cF^{\mathbb W}_t$-measurable for all $t\ge 0$. In order to avoid further notation, in what follows we treat mappings $\Phi:[0,\infty)\times C_0([0,\infty))\to \R$ as mappings defined on the canonical space but with no dependence on the coordinate process $\zeta$.

The next definition of randomised strategy follows an idea proposed by Aumann \cite{Aumann}: a randomised strategy is a family of strategies depending on an auxiliary randomisation variable $u$. Since we consider firms with different initial endowments, we formally introduce the class of admissible randomised strategies for the problem starting from an arbitrary initial point $x$.
\begin{definition}[Randomised Strategy] \label{def:random_strategies}
A measurable mapping $(u,t,\varphi,\zeta)\mapsto \Xi(u,t,\varphi,\zeta)$ with
$\Xi: [0,1]\times[0,\infty)\times C_0([0,\infty))\times D_0^+([0,\infty))\to [0,\infty)$
is an admissible {\em randomised strategy} with initial condition $x$ if:
\begin{itemize}
\item[(i)] $\Xi(u,\cdot,\cdot,\cdot)$ is non-anticipative for each $u\in[0,1]$, 
\item[(ii)] $t\mapsto \Xi(u,t,\varphi,\zeta)$ is right-continuous, non-decreasing for $(u,\varphi,\zeta)\in[0,1]\times C_0([0,\infty))\times D_0^+([0,\infty))$, 
\item[(iii)] For all $(u,\varphi,\zeta) \in [0,1]\times C_0([0,\infty))\times D_0^+([0,\infty))$ and $t\geq 0$. 
\begin{equation*}
\Xi(u,t,\varphi,\zeta)- \Xi(u,t-,\varphi,\zeta)  \leq \left( x+\mu_0 t + \sigma \varphi(t)- \Xi(u,t-,\varphi,\zeta) \right)^+,
\end{equation*}
with the convention $\Xi(u,0-,\varphi,\zeta)=0$.
\end{itemize}
The set of admissible {\em randomised strategies} with initial condition of the state dynamics equal to $x$ is denoted $\Sigma_R(x)$. The subset of $\Sigma_R(x)$
of mappings $\Xi(u,t,\varphi,\zeta)=\Psi(t,\varphi,\zeta)$
which {\em do not depend} on the randomisation parameter $u$ is denoted $\Sigma(x)$ and its elements are called {\em pure} admissible strategies.
\end{definition}
\medskip

The two players use mutually independent randomisation devices. To emphasise that, we let the $i$-th player's randomised strategy depend on a variable $u_i \in [0,1]$, $i=1,2$. The next definition will cast rigorously the following heuristics: when a player plays a randomised strategy, for a fixed value of $u$ and a fixed $\zeta$, the realised trajectory $(t,\varphi)\mapsto\Xi(u,t,\varphi,\zeta)$ is an admissible control according to Definition \ref{def:div}; letting the variable $u$ vary on $[0,1]$ randomises such control.
We now introduce a family of randomised controls which depends on {\em two} auxiliary variables $(u_1,u_2)\in [0,1]^2$. Elements from this family represent realised controls when {\em both players} play randomised strategies.
\begin{definition}[Randomised Control] \label{def:random_controls}
A measurable mapping $(u_1,u_2,t,\varphi)\mapsto \Upsilon(u_1,u_2,t,\varphi)$, 
with $\Upsilon: [0,1]^2\times[0,\infty)\times C_0([0,\infty))\to [0,\infty)$
is an admissible {\em randomised control} with initial condition $x$ if:
\begin{itemize}
\item[(i)] $\Upsilon(u_1,u_2,\cdot,\cdot)$ is non-anticipative for each $(u_1,u_2)\in[0,1]^2$, 
\item[(ii)] $t\mapsto \Upsilon(u_1,u_2,t,\varphi)$ is right-continuous and non-decreasing for any $(u_1,u_2,\varphi)\in[0,1]^2\times C_0([0,\infty))$, 
\item[(iii)] For all $(u_1,u_2,\varphi) \in [0,1]^2\times C_0([0,\infty))$ and $t\geq 0$. 
\begin{equation*}
\Upsilon(u_1,u_2,t,\varphi)- \Upsilon(u_1,u_2,t-,\varphi)  \leq \left( x+\mu_0 t + \sigma \varphi(t)- \Upsilon(u_1,u_2,t-,\varphi) \right)^+,
\end{equation*}
with the convention $\Upsilon(u_1,u_2,0-,\varphi)=0$.
\end{itemize}
The set of admissible {\em randomised controls} with initial condition of the state dynamics equal to $x$ is denoted $\cD_R(x)$.
The subset of $\cD_R(x)$ of mappings $\Upsilon(u_1,u_2,t,\varphi)=\Phi(t,\varphi)$ which {\em do not depend} on the randomisation parameters $(u_1,u_2)$ is denoted $\cD(x)$ and its elements are called {\em pure} admissible controls. 
\end{definition}
\medskip

\begin{remark}
Notice that the set $\cD(x)$ is exactly the set of controls introduced in Definition \ref{def:div}. Indeed, in the notation of Definition \ref{def:div} the coordinate mapping $\varphi$ corresponds to $B$ and a $(\cF_t)$-adapted, c\`adl\`ag process $L$ can be expressed as a non-anticipative mapping $(t,B)\mapsto\Phi(t,B)$. 
\end{remark}
\medskip

Every pair $(\Phi,\Psi)\in \cD(x)\times \Sigma(y)$ induces a unique pair of (pure) admissible controls: 
$(t,\varphi)\rightarrow (\Phi(t,\varphi), \Psi(t,\varphi,\Phi(t,\varphi))$.
However, it is well-known  that not every pair of pure strategies, or more generally any pair of randomised strategies, induces a pair of admissible controls\footnote{For example, let us consider the pure strategies $\Psi^\dagger(t,\varphi,\zeta)=1_{\{\zeta(0)>0\}}$ (``I move only if you move") and $\Psi^\sharp(t,\varphi,\zeta)=1_{\{\zeta(0)=0\}}$ (``I move only if you don't move"), which do not depend on $\varphi$. 
It is easy to verify that there exists no pair $(L,D)$ induced by $(\Psi^\dagger,\Psi^\sharp)$. Indeed, both strategies $\Psi^\dagger$ and $\Psi^\sharp$ take their values in $\{0,1\}$ (i.e., constant processes), but none of the pairs $(L,D) \in \{(0,0),(0,1),(1,0),(1,1)\}$ satisfies $L=\Psi^\dagger(\cdot,\cdot,D)$ and $D=\Psi^\sharp(\cdot,\cdot,L)$. 
More general discussions on technical and conceptual issues related to strategies in continuous-time games can be found in \cite{Neyman} and \cite{PTZ}.}. That leads us to consider the subset of $\Sigma_R(x)\times \Sigma_R(y)$ defined below:
\begin{definition}[Control-inducing pairs]\label{def:theta_r}
We let $\Theta_R(x,y)$ be the collection of all pairs $(\Xi_1,\Xi_2)\in \Sigma_R(x)\!\times\! \Sigma_R(y)$ for which there exists $(L,D) \in \cD_R(x)\!\times\! \cD_R(y)$ such that, for all $(u_1,u_2) \in [0,1]^2$, $t\geq 0$ and $\varphi \in C_0([0,\infty))$, the map $[0,t] \ni s \rightarrow (L(u_1,u_2,s,\varphi),D(u_1,u_2,s,\varphi))$ is the unique solution pair of the system:
\begin{equation*}
\left\{ 
\begin{matrix} 
L(u_1,u_2,s,\varphi)=\Xi_1(u_1,s,\varphi,D(u_1,u_2,s,\varphi)), \\ 
D(u_1,u_2,s,\varphi)=\Xi_2(u_2,s,\varphi,L(u_1,u_2,s,\varphi)), 
\end{matrix} 
\right. 
\quad\forall s\in [0,t].
\end{equation*}
\end{definition}
\medskip

Notice that uniqueness of the solution of the fixed point problem appearing in the definition of $\Theta_R(x,y)$ is required {\em for any time interval} $[0,t]$. This is needed because a situation may arise in which there are several fixed points on $[0,t]$ and all but one of them will later lead to non-existence of the fixed point (say at some $t'>t$). In that situation, global uniqueness holds but local uniqueness fails. Problems of existence and uniqueness of such fixed points in continuous-time games have been discussed extensively in the literature. We point the interested reader to \cite{PTZ} and references therein for a survey of different approaches. 
 
Player 1's payoff associated to a pair $(\Xi_1,\Xi_2) \in \Theta_R(x,y)$ is defined as the expected payoff associated to the induced control pair $(L,D)\in\cD_R(x)\times\cD_R(y)$, evaluated along the trajectory of the underlying Brownian motion (and analogously for Player 2). 
Compared to the formulation of the two players' payoffs given in Equation \eqref{eq:J12}, we need to expand the probability space in order to account for the randomisation of the strategies. At the payoffs' level, that corresponds to integrating the variables $(u_1,u_2)$ over the square $[0,1]^2$. 

More precisely, we consider the enlarged probability space 
\begin{equation*}
(\bar \Omega, \bar \cF, \bar \P):=(\Omega\times [0,1]\times [0,1], \cF\otimes \cB([0,1])\otimes \cB([0,1]), \P\otimes \lambda \otimes \lambda),
\end{equation*}
with canonical element $\bar \omega=(\omega,u_1,u_2)$, where $\cB([0,1])$ denotes the Borel $\sigma$-field and $\lambda$ the Lebesgue measure. We denote $\bar \P_{x,y}(\,\cdot\,)=\bar \P(\,\cdot\,|X_{0-}=x,Y_{0-}=y)$ and $\bar \E_{x,y}$ the corresponding expectation operator.
The two random variables $U_i(\bar \omega)=u_i$ for $i=1,2$  defined on $(\bar \Omega,\bar \cF, \bar \P)$, are uniformly distributed on $[0,1]$, mutually independent and independent of the Brownian motion $B=B(\omega)$. 

On this space, a pair of randomised controls $(L,D)\in\cD_R(x)\times\cD_R(y)$ is identified with random processes on $\bar \Omega$ through the relations 
$L_t =L(U_1,U_2,t,B)$ and $D_t=D(U_1,U_2,t,B)$, $\forall t\geq 0$.
Moreover, for every pair $(u_1,u_2)$ the maps $L^{u_1,u_2}=L(u_1,u_2,\cdot,\cdot)$ and $D^{u_1,u_2}=D(u_1,u_2,\cdot,\cdot)$ are identified with admissible controls in the sense of Definition \ref{def:div}.

\begin{definition}[Players' payoffs] \label{def:payoffs}
For $(L,D)\in\cD_R(x)\times\cD_R(y)$, with a slight abuse of notation we set 
\begin{equation}\label{eq:payoff0}
\begin{aligned}
\cJ^1_{x,y}(L,D)&\coloneqq \bar \E_{x,y}\Big[\int_{[0,\gamma_X\wedge\gamma_Y]}\e^{-r t}\ud L_t+1_{\{\gamma_Y<\gamma_X\}}\e^{-r\gamma_Y}\hat v(X^L_{\gamma_Y})\Big]\\
&=\int_0^1 \int_0^1 \cJ^1_{x,y}(L^{u_1,u_2},D^{u_1,u_2}) \ud u_1 \ud u_2,
\end{aligned}
\end{equation}
where $\cJ^1_{x,y}(L^{u_1,u_2},D^{u_1,u_2})$ has been introduced in \eqref{eq:J12} 
and the second equality in \eqref{eq:payoff0} follows from Fubini's theorem.
Player $2$'s payoff associated to a pair $(L,D)$ is defined analogously.

Given a pair $(\Xi_1,\Xi_2) \in \Theta_R(x,y)$ of randomised strategies with associated $(L,D)\in\cD_R(x)\times\cD_R(y)$, we identify (again with a small abuse of notation)
\begin{equation}\label{eq:payoff}
\begin{aligned}
\cJ^1_{x,y}(\Xi_1,\Xi_2)\coloneqq\cJ^1_{x,y}(L,D)\quad\text{and}\quad\cJ^2_{x,y}(\Xi_2,\Xi_1)\coloneqq\cJ^2_{x,y}(D,L).
\end{aligned}
\end{equation}
The payoffs associated to a pair $(\Xi_1,\Xi_2) \notin \Theta_R(x,y)$ are defined as
$\cJ^1_{x,y}(\Xi_1,\Xi_2)=\cJ^2_{x,y}(\Xi_2,\Xi_1)\coloneqq-\infty$ (in the same spirit as in \cite{PTZ}).
\end{definition}
\medskip

Next we introduce the general notion of equilibrium used in this paper.
\begin{definition}[Nash Equilibrium in randomised strategies]\label{def:NE3}
Given $(x,y)\in [0,\infty)^2$, a pair $(\Xi^*_1,\Xi^*_2) \in \Sigma_R(x)\times \Sigma_R(y)$ is a Nash equilibrium if and only if for all $(\Xi_1,\Xi_2) \in \Sigma_R(x)\times \Sigma_R(y)$
\begin{equation*}
\cJ^1_{x,y}(\Xi_1,\Xi^*_2)\le\cJ^1_{x,y}(\Xi^*_1,\Xi^*_2)\quad\text{and}\quad \cJ^2_{x,y}(\Xi_2,\Xi^*_1)\le\cJ^2_{x,y}(\Xi^*_2,\Xi^*_1). 
\end{equation*}
\end{definition}
\medskip

We notice that a different definition of admissible strategy profile and equilibrium can be found in, e.g., \cite{Neyman}. A player's strategy in \cite{Neyman} is admissible if it induces a well-defined dynamics for {\em any} choice of the opponent's strategy. In our context, that definition is too restrictive. Indeed the equilibrium pairs $(\Xi^*_1,\Xi^*_2)$ that we find in this paper are not admissible in the sense of Neyman's but they are control-inducing in the sense of our Definition \ref{def:theta_r}. From now on we often use the notation for the game payoffs introduced in \eqref{eq:payoff}.

Randomised and pure strategies/controls are linked and this link leads to some useful considerations about Nash equilibria. First of all we notice that $\cD(x)\subset \Sigma(x) \subset \Sigma_R(x)$. The first inclusion is obtained by identifying a control with a pure strategy that does not depend on the variable $\zeta$. The second inclusion holds because a pure strategy is a randomised strategy that does not depend on $u$. It follows that $\Theta_R(x,y)\neq\varnothing$, because for all $(x,y)\in [0,\infty)^2$, we have
\begin{equation}\label{eq:inclusion_control_theta}
\cD(x)\times   \Sigma_R(y) \subseteq \Theta_R(x,y) \quad\text{and}\quad \Sigma_R(x)\times \cD(y) \subseteq \Theta_R(x,y).
\end{equation}
Indeed, any pair $(\Xi_1,\Phi_2) \in\Sigma_R(x)\times   \cD(y)$ induces a pair $(L,D)\in\cD_R(x)\times\cD_R(y)$ that does not depend on $u_2$ defined by
\begin{equation}\label{eq:pair}
L(u_1,u_2,t,\varphi)=\Xi_1(u_1,t,\varphi,\Phi_2(\,\cdot\, ,\varphi)), \quad D(u_1,u_2,t,\varphi)=\Phi_2(t,\varphi),
\end{equation}
for all $u_1,u_2\in[0,1]^2$, $t\geq 0$ and $\varphi \in C_0([0,\infty)$. 
The associated payoffs read 
\begin{equation}
\cJ^1_{x,y}(\Xi_1,\Phi_2)=\int_0^1  \cJ^1_{x,y}(L^{u_1},D)  \ud u_1\quad\text{and}\quad \cJ^2_{x,y}(\Phi_2,\Xi_1)=\int_0^1  \cJ^2_{x,y}(D,L^{u_1})  \ud u_1,
\end{equation}
where for every $u_1\in [0,1]$, $L^{u_1}:=L(u_1,\cdot,\cdot,\cdot)$ is an admissible control. 
We use a similar notation for payoffs associated to a pair $(\Phi_1,\Xi_2) \in \cD(x) \times \Sigma_R(y)$.

Thanks to the observations above, we obtain useful results about Nash equilibria. In particular, in Definition \ref{def:NE3} each player can restrict deviations from the equilibrium pair to pure strategies. Moreover, if there is a Nash equilibrium in which both players use a pure strategy, each player can restrict deviations to controls. 
 
\begin{lemma}\label{lem:deviations}
Given $(x,y)\in [0,\infty)^2$, the following properties hold: 
\begin{enumerate}
\item[(i)] If $(\Xi^*_1,\Xi^*_2)\in \Sigma_R(x)\times \Sigma_R(y)$ is a Nash equilibrium, then $(\Xi^*_1,\Xi^*_2)\in \Theta_R(x,y)$. 
\item[(ii)] A pair $(\Xi^*_1,\Xi^*_2) \in \Sigma_R(x)\times \Sigma_R(y)$ is a Nash equilibrium if and only if,\ $\forall(\Psi_1,\Psi_2) \in \Sigma(x)\times \Sigma(y)$
\begin{equation}
\cJ^1_{x,y}(\Psi_1,\Xi^*_2)\le\cJ^1_{x,y}(\Xi^*_1,\Xi^*_2)\quad\text{and}\quad \cJ^2_{x,y}(\Psi_2,\Xi^*_1)\le\cJ^2_{x,y}(\Xi^*_2,\Xi^*_1). 
\end{equation}
\item[(iii)] A pair $(\Psi^*_1,\Psi^*_2) \in \Sigma(x)\times \Sigma(y)$ is a Nash equilibrium if and only if,\ $\forall(\Phi_1,\Phi_2) \in \cD(x)\times \cD(y)$
\begin{equation}
\cJ^1_{x,y}(\Phi_1,\Psi^*_2)\le\cJ^1_{x,y}(\Psi^*_1,\Psi^*_2)\quad\text{and}\quad \cJ^2_{x,y}(\Phi_2,\Psi^*_1)\le\cJ^2_{x,y}(\Psi^*_2,\Psi^*_1). 
\end{equation}
\end{enumerate}
\end{lemma}
\begin{proof}
The first point follows directly from the fact that the payoff is non-negative on $\Theta_R(x,y)$, that both players can deviate by choosing an arbitrary control and using \eqref{eq:inclusion_control_theta}.

Let us prove (ii). For the {\em only if} part, it is clear that a Nash equilibrium satisfies the inequalities stated in (ii), because $\Sigma(x)\subset \Sigma_R(x)$ and $\Sigma(y)\subset \Sigma_R(y)$. For the {\em  if} part, let us assume that those inequalities hold. Then, as in the proof of (i) it must be $(\Xi^*_1,\Xi^*_2)\in\Theta_R(x,y)$. Take $\Xi_1 \in \Sigma_R(x)$. Clearly, if $(\Xi_1,\Xi^*_2)\notin \Theta_R(x,y)$ then $\cJ^1_{x,y}(\Xi^*_1,\Xi^*_2)\ge \cJ^1_{x,y}(\Xi_1,\Xi^*_2)=-\infty$. So, with no loss of generality we assume $(\Xi_1,\Xi^*_2)\in \Theta_R(x,y)$. 
Let  $(L,D) \in \cD_R(x)\times \cD_R(y)$ denote the controls induced  by the pair  $(\Xi_1,\Xi^*_2)$ as per Definition \ref{def:theta_r}. From Definition \ref{def:random_strategies}, for every $u_1$, we have $\Xi_1^{u_1}:=\Xi_1(u_1,\cdot,\cdot,\cdot)\in\Sigma(x)$. Uniqueness of the solution of the fixed point problem appearing in the definition of $\Theta_R(x,y)$ implies that $(\Xi_1^{u_1},\Xi^*_2) \in \Theta_R(x,y)$. More precisely, for fixed $u_1$, the unique control pair induced  by the pair  $(\Xi_1^{u_1},\Xi^*_2)$ is the pair of maps depending only on $(u_2,t,\varphi)$  given by $(L^{u_1},D^{u_1})$ where $L^{u_1}:=L(u_1,\cdot,\cdot,\cdot)$ and $D^{u_1}:=D(u_1,\cdot,\cdot,\cdot)$.
Recalling the notation $L^{u_1,u_2}, D^{u_1,u_2}$ for $i=1,2$ and Player 1's payoff in Definition \ref{def:payoffs}, we have 
$\cJ^1_{x,y}(\Xi_1^{u_1},\Xi^*_2)=\cJ^1_{x,y}(L^{u_1},D^{u_1})=\int_0^1 \cJ^1_{x,y}(L^{u_1,u_2},D^{u_1,u_2}) \ud u_2$.
With this in mind, Player 1's payoff reads 
\begin{equation*}
\begin{aligned}
\cJ^1_{x,y}(\Xi_1,\Xi^*_2)&=\int_0^1 \int_0^1 \cJ^1_{x,y}(L^{u_1,u_2},D^{u_1,u_2}) \ud u_2 \ud u_1 \\
&= \int_0^1 \cJ^1_{x,y}(\Xi_1^{u_1},\Xi^*_2) \ud u_1 \leq \int_0^1 \cJ^1_{x,y}(\Xi^*_1,\Xi^*_2) \ud u_1 =\cJ^1_{x,y}(\Xi^*_1,\Xi^*_2),
\end{aligned}
\end{equation*} 
where the inequality holds by the equations in (ii). A similar inequality holds for Player $2$, thus concluding the proof. 

Let us prove point (iii). For the {\em only if} part, it is clear that a Nash equilibrium satisfies the inequalities stated in (iii) because $\cD(x)\subset \Sigma_R(x)$ and $\cD(y)\subset \Sigma_R(y)$. 
For the {\em if} part, let us assume that those inequalities hold for all $(\Phi_1,\Phi_2) \in \cD(x)\times \cD(y)$. Then, as in the proof of (i) it must be $(\Psi^*_1,\Psi^*_2)\in\Theta_R(x,y)$. As in the proof of (ii), we may assume that $\Xi_1 \in \Sigma_R(x)$ is such that $(\Xi_1,\Psi_2^*)\in \Theta_R(x,y)$.
Let  $(L,D) \in \cD_R(x)\times \cD_R(y)$ denote the randomised controls induced  by the pair  $(\Xi_1,\Psi^*_2)$ as per Definition \ref{def:theta_r}, and note that these controls depend only on the variables $(u_1,t,\varphi)$. 
Recalling the definition of Player 1's payoff, we have 
$\cJ^1_{x,y}(\Xi_1,\Psi^*_2)=\int_0^1 \cJ^1_{x,y}(L^{u_1},D^{u_1}) \ud u_1$,
where $L^{u_1}:=L(u_1,\cdot,\cdot,\cdot)$ and $D^{u_1}:=D(u_1,\cdot,\cdot,\cdot)$.
The fact that $(L,D)$ is solution of the fixed point problem appearing in the definition of $\Theta_R(x,y)$ associated to $(\Xi_1,\Psi^*_2)$ implies that for each $u_1$ the pair $(L^{u_1},\Psi^*_2)$ induces the unique control pair $(L^{u_1},D^{u_1})$. It follows that
$\cJ^1_{x,y}(L^{u_1},\Psi^*_2)=\cJ^1_{x,y}(L^{u_1},D^{u_1})$.
We conclude that
\begin{equation*}
\begin{aligned}
\cJ^1_{x,y}(\Xi_1,\Psi^*_2)&=\int_0^1 \cJ^1_{x,y}(L^{u_1},D^{u_1}) \ud u_1 = \int_0^1 \cJ^1_{x,y}(L^{u_1},\Psi^*_2) \ud u_1 \\
&\leq \int_0^1 \cJ^1_{x,y}(\Psi^*_1,\Psi^*_2) \ud u_1 =\cJ^1_{x,y}(\Psi^*_1,\Psi^*_2),
\end{aligned}
\end{equation*} 
where the inequality holds by the {\em if} assumption in (iii). A similar inequality holds for player $2$, thus concluding the proof. 
\end{proof}

This section provided a rigorous formulation of the game faced by the two firms. The challenge is now to characterise equilibrium strategies. A general characterization of {\em all} possible Nash equilibria (according to Definition \ref{def:NE3}) seems unfeasible. In the next two sections we find one equilibrium for the case of firms with different initial endowment and one {\em symmetric} equilibrium for firms with the same initial endowment. However, we do not claim uniqueness of those equilibria and indeed, in the symmetric case, we show that there are at least three different equilibria, two of which are not symmetric across players (cf.\ Remark \ref{MultipleNash}). We also leave for future study the question of subgame perfection of our equilibria, because of subtle technical difficulties that arise in our continuous-time setup. Nevertheless, we notice that the construction of our equilibria is based on dynamic programming ideas and it holds for any initial point of the underlying dynamics. That seems a natural starting point for a suitable notion of subgame perfection.

\section{Nash equilibrium with asymmetric initial endowment.}\label{sec:NEasym}

In this section we consider firms with different initial endowments. Specifically, we study the case $y>x$ and construct a Nash equilibrium. It turns out that, in our setting, it is enough for both players to use pure strategies to obtain a Nash equilibrium. Even more, the poorer firm's strategy is independent of the $\zeta$-variable and it is therefore a pure control. Next we state the theorem. Its proof will be distilled in a series of intermediate results which will be illustrated in the rest of the section.

\begin{theorem}[NE with asymmetric endowment]\label{thm:NE1}
Let $y>x$ and recall $\hat a$ and $a_0$ as in \eqref{eq:dnot}. There exists a function $b:[0,\infty)\to [0,\infty)$ and a constant $\alpha>0$ with the following properties:
\begin{itemize}
\item[(i)] $b(0)=\hat a$ (we extend $b$ to $(-\infty,0)$ as $b(z)=+\infty$ for $z<0$),  
\item[(ii)] $b\in C([0,\infty))$ and $b\in C^1([0,a_0])$, 
\item[(iii)] $b$ strictly decreasing on $[0,a_0]$, 
\item[(iv)] $b(x)=\alpha>0$ for $x\ge a_0$, 
\end{itemize}
such that the pair $(\Phi^*(t,\varphi),\Psi^*(t,\varphi,\Phi^*(t,\varphi)))\in \cD(x)\times\Sigma(y)$ is a Nash equilibrium for the game starting at $(x,y)$, with
\begin{equation}\label{eq:maps}
\begin{aligned}
\Phi^*(t,\varphi)&:=\sup_{0\le s\le t} \Big(x-a_0+\mu_0 s+\sigma \varphi(s)\Big)^+,\\
\Psi^*(t,\varphi,\zeta)&:=\sup_{0\le s\le t}\Big(y-x+\zeta(s)-b\big(x+\mu_0 s +\sigma \varphi(s)-\zeta(s)\big)\Big)^+,
\end{aligned}
\end{equation} 
for $t\geq 0$ and all $(\varphi,\zeta)\in C_0([0,\infty))\times D_0^+([0,\infty))$.
\end{theorem}
In keeping with the notation from \eqref{eq:XY}, the controls induced by the pair of strategies $(\Phi^*,\Psi^*)$ (cf.\ Definition \ref{def:theta_r}) will be denoted
\begin{equation}\label{eq:phipsi*}
L^*_t=\Phi^*(t,B)\quad\text{ and }\quad D^*_t=\Psi^*(t,B,L^*),
\end{equation}
upon recalling that $B$ is the canonical process on $\Omega$ under the Wiener measure $\P$ (recall also that $(\cF_t)_{t\ge 0}$ is the canonical {\em raw} filtration generated by $B$). 
The associated controlled processes are denoted $X^*\!=\!X^{L^*}$ and $Y^*\!=\!Y^{D^*}$. It follows that $L^*\!=\!\xi^0$ as in \eqref{eq:dnot} and 
\begin{equation}\label{eq:D*}
D^*_t=\sup_{0\le s\le t}\Big(y-x+L^*_s-b(X^*_s)\Big)^+,\quad D^*_{0-}=0.
\end{equation}
Then, Player 1's equilibrium payoff reads $v_1(x,y)\coloneqq\cJ^1_{x,y}(L^*,D^*)=v_0(x)$ with $v_0$ as in \eqref{eq:dnot}. The function $x\mapsto b(x)$ is actually constructed explicitly thanks to \eqref{eq:defb}. Moreover, it will later be clear that Player 2's equilibrium payoff is a $C^1$ function $v_2(x,y)\coloneqq\cJ^2_{x,y}(D^*,L^*)$ solving a suitable free boundary problem with $b(x)$ the free boundary (cf.\ Proposition \ref{prop:SC}).
\begin{remark}[An intuitive interpretation]
An intuitive interpretation of the equilibrium obtained in the theorem is as follows: first of all, given the initial cash advantage, Player 2 can afford to act so as to make sure that the rival cannot become monopolist; this effectively ``forces'' Player 1 to solve the same optimization problem as in the single-agent dividend problem (with drift $\mu_0$); hence, Player 1 adopts the control $L^*_t=\Phi^*(t,B)$ distributing dividends when the cash-reserve exceeds/equals the threshold $a_0$ (optimal in the classical dividend problem); now, Player 2 knows what Player 1 is going to do, and she needs to select a best response; it turns that the best response consists of paying dividends when the cash reserve $Y^*_t$ is at the moving threshold $X^*_t+b(X^*_t)$ where $X^*=X^{L^*}$ (compare to \eqref{eq:X*Y*} later on); in other words Player 2 distributes dividends 
only when the difference of capital between the two firms is sufficiently large (i.e., $Y^*_t-X^*_t\ge b(X^*_t)\ge \alpha>0$); this happens because, due to discounting, Player 2 does not want to be idle while waiting for Player 1 to default; at the same time the condition $Y^*_t-X^*_t\ge b(X^*_t)\ge \alpha>0$ creates a safety buffer that guarantees Player 2 can never be driven out of the market before Player 1. The level $\alpha$ is determined endogenously in Player 2's optimization, balancing the interplay between the effect of discounting and the value of the continuation payoff as a monopolist. The value of $\alpha$ corresponds to the cash holding that makes Player 2 indifferent between paying dividends or not when $X^*$ is equal to $a_0$. 
Finally, we observe that when Player 2 becomes monopolist (i.e., $X^*_t=0$) she pays dividends at the threshold $b(0)=\hat a$, which is optimal for the monopolist's problem. 
\end{remark}
\medskip

\begin{remark}[Measurable dependence on initial endowment]\label{rem:phipsi*}
Notice for later use that the strategies $\Phi^*$ and $\Psi^*$  depend in a measurable way on the initial positions $x$ and $(x,y)$, respectively. When necessary, in Section \ref{sec:sym} we will denote them $\Phi^*(x,t,\varphi)$ and $\Psi^*(x,y,t,\varphi,\zeta)$ to emphasise that dependence. Throughout the current section $x$ and $y$ are fixed, with $y>x$, and we use notations as in \eqref{eq:maps}.
\end{remark}
\medskip

\begin{figure}
\includegraphics[scale=0.25]{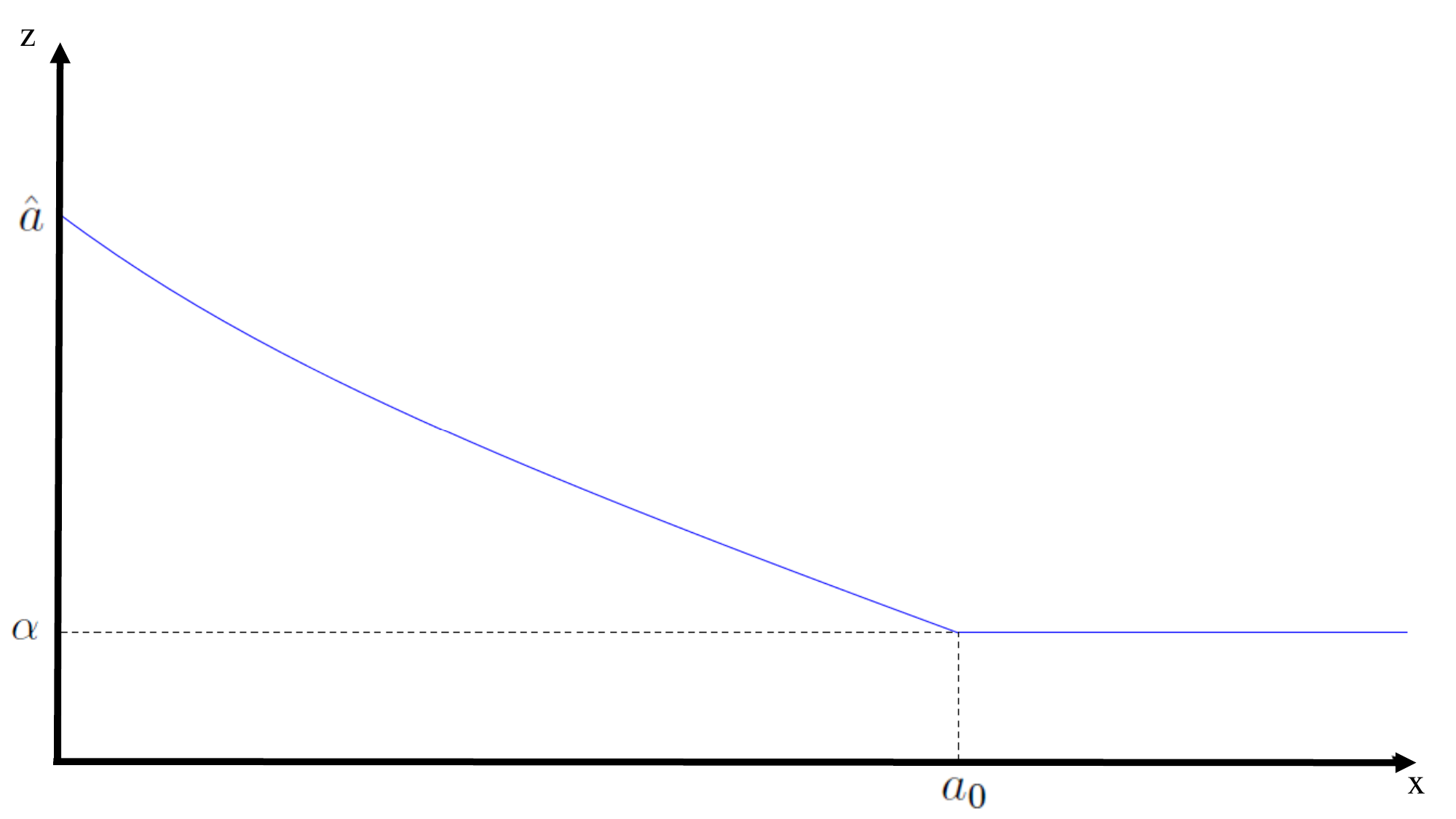}
\caption{An illustration of the boundary $x\mapsto b(x)$ (in blue) in the positive quadrant, obtained by solving \eqref{eq:defb} with parameter values:  $\hat \mu = 1.8$, $\mu_0 = 0.8$, $r = 0.8$, $\sigma = 0.4$, so that  $a_0 \simeq 0.419$, $\alpha \simeq 0.079$ and $\hat a \simeq 0.339$. In the plot we see that the boundary is decreasing in $[0,a_0]$ and it is extended to a constant $b(x)=\alpha$ for $x\ge a_0$.} 
\end{figure}

Our analysis in the rest of the section is organised as follows: first, we construct the solution of a suitable free boundary problem (Proposition \ref{prop:SC}); second, we illustrate properties of the control $D^*$ (Lemma \ref{lem:D*}) and show that it is a best response against the control $L^*$ (Proposition \ref{prop:D*}); third, we show that $L^*$ is best response against the strategy $\Psi^*$ (Lemma \ref{lem:Phi}); finally, combining these results we obtain the proof of Theorem \ref{thm:NE1} thanks to Lemma \ref{lem:deviations}.

It is convenient to change our reference system and consider the state variables $(x,z)$ with $z=y-x$. Letting $Z^{L,D}_t:=Y^D_t-X^L_t$ it is immediate to check that 
$Z^{L,D}_t=y-x+L_t-D_t$, for $t\ge 0$. Moreover, for the default times we have 
$\gamma_Y=\inf\{t\ge 0\,:\,Z^{L,D}_t\le -X^L_t\}=:\gamma_Z$ and $Y_{\gamma_X}=Z_{\gamma_X}$, $\P$-a.s. It should be clear that $\gamma_X=\gamma_X(L)$ and $\gamma_Z=\gamma_Z(D)$, therefore we omit the dependence on the controls, for ease of notation.

In the first part of our analysis we focus on constructing Player 2's best response to Player 1 using the control $L^*=\xi^0$.
It is well-known from the classical dividend problem (and it can be easily verified from \eqref{eq:xi*}) that $L^*_{0-}=0$ and $L^*$ is continuous except for a possible jump at time zero of size $L^*_0=(x-a_0)^+$. Moreover, the Skorokhod condition $\ud L^*_t=1_{\{X^*_t=a_0\}}\ud L^*_t$ holds for all $t>0$, $\P$-a.s. Finally $X^*_t\le a_0$ for all $t\ge 0$, $\P$-a.s. In summary, $L^*$ reflects the dynamics of $X^*$ at the boundary $a_0$ and it is well-known to be optimal for the dividend problem with drift $\mu_0$. 

Since $\cJ^2_{x,y}(\Psi^*,\Phi^*)=\cJ^2_{x,y}(D^*,L^*)$, it follows from  Lemma \ref{lem:deviations}-(iii) that $\Psi^*$ is a best reply to $\Phi^*$ if the control $D^*$ is optimal for Player 2 in the singular control problem with value 
\begin{equation}\label{eq:v2}
v_2(x,y;L^*)=\sup_{D \in \cD(y)}\cJ^2_{x,y}(D,L^*).
\end{equation}

Changing variables we define $u_2(x,z):=v_2(x,z+x;L^*)$ and we have
\begin{equation}\label{eq:u2}
u_2(x,z)=\sup_{D\in \cD(x+z)}\E_{x,z}\Big[\int_{[0,\gamma_{X^*}\wedge\gamma_Z]}\e^{-r t}\ud D_t+1_{\{\gamma_{X^*}<\gamma_Z\}}\e^{-r\gamma_{X^*}}\hat v(Z^{L^*,D}_{\gamma_{X^*}})\Big].
\end{equation}
It suffices to define the function $u_2$ on the set
$H=\{ (x,z) \in [0,a_0]\times \R  \,|\, z \geq -x\}$.
Indeed, if $x>a_0$, the initial jump of the control process $L^*$ shifts the $x$ coordinate to the value $X^*_0=a_0$. Then, we can simply extend the definition of $u_2$ as 
\begin{equation}\label{eq:extu2}
u_2(x,z):=u_2(a_0,z),\quad\text{for $x>a_0$.}
\end{equation} 
We will characterise properties of the function $u_2$, along with an optimal control, using a verification approach. 
In the next proposition we use three closed sets defined as follows: given a decreasing function $b \in C([0,a_0])$ with $\alpha:=b(a_0)>0$,
\begin{equation}
\begin{aligned}
H_\le := \{(x,z)&\in H \,|\, z\le 0\},\quad H_{[0,\alpha]} :=\{(x,z)\in H\,|\,z\in[0,\alpha]\},\\ 
&H_{[\alpha,b]} :=\{ (x,z) \in H \,|\, \alpha \leq z \leq b(x) \}. 
\end{aligned}
\end{equation}

\begin{proposition}\label{prop:SC}
Recall $\hat v$, $v_0$, $\hat a$ and $a_0$ from \eqref{eq:dnot}. There is a unique pair $(u, b)$ such that:
\begin{itemize}
\item[(i)] The function $b$ is continuous on $[0,\infty)$, strictly decreasing on $[0,a_0]$ with $b\in C^1([0,a_0])$, $b(0)=\hat a$ and $b(a_0)=\alpha>0$. Here $\alpha$ is the unique solution of $\hat v'(\alpha)=v_0'(0)$. Moreover, $b(\cdot)$ satisfies
\begin{equation}\label{eq:defb} 
b(x)=[(\hat v')^{-1}\circ v_0'](a_0-x),\quad \forall x \in [0,a_0].
\end{equation}
\item[(ii)] For $\cC:=\{(x,z)\in H : z< b(x)\}$ and $\cS=H\setminus \cC$,
it holds 
$u\in C^1(H)\cap C^2(H_\le\cup H_{[0,\alpha]}\cup H_{[\alpha, b]}\cup\cS)$,
with $u_{xz}$ continuous across the boundary $x\mapsto b(x)$.
\item[(iii)] The function $u$ solves the variational system
\begin{equation}\label{eq:fbp}
\left\{
\begin{array}{ll}
\big(\tfrac{\sigma^2}{2}\partial_{xx} u+\mu_0\partial_x u-r u\big)(x,z)=0,& \text{for $(x,z)\in \overline \cC$},\\
\big(\tfrac{\sigma^2}{2}\partial_{xx} u+\mu_0\partial_x u-r  u\big)(x,z)\le0,& \text{for $(x,z)\in \cS$},\\
\partial_z  u(x,z)> 1, &\text{for $(x,z)\in \cC $},\\
\partial_z  u(x,z)=1, &\text{for $(x,z)\in \cS$},\\
\partial_{zx}  u(x,b(x))=0, &\text{for $x\in [0,a_0]$},\\
u(0,z)=\hat v(z), &\text{for $z\in[0,\infty)$},\\
u(x,-x)=0, & \text{for $x\in (0,a_0]$}, \\ 
(\partial_z u-\partial_x u)(a_0,z)=0, &\text{for $z\in [-a_0,+\infty)$}.
\end{array}
\right.
\end{equation}
\end{itemize}
\end{proposition}

\begin{proof}
We proceed with the construction of the function $u$ and of the boundary $b$ in four steps, considering first the region $H_{[\alpha,b]}$, then the region $H_{[0,\alpha]}$ and finally the regions $H_\le$ and $\cS$. We emphasise that in steps 1--3 we will produce {\em candidates} for the function $u(x,z)$ (or its derivative $\partial_z u$) and the boundary $b(x)$. In Step 4 we will verify that the candidate pair $(u,b)$ solves the free boundary problem \eqref{eq:fbp}. With a slight abuse of notation, we use $u$ and $b$ to denote those candidates.

We begin with some basic preliminaries. Let $\beta_2<0<\beta_1$ be the solutions of the quadratic characteristic equation
$\tfrac{\sigma^2}{2} \beta^2+ \mu_0 \beta - r=0$.
Since $b$ must be strictly decreasing on $[0,a_0]$, we can equivalently consider its inverse $c(z):=b^{-1}(z)$, which must be continuous, decreasing on some interval $[b(a_0),b(0)]$, with $c(b(a_0))$ to be determined and $c(b(0))=0$.
From the first equation in \eqref{eq:fbp} we must have for $(x,z)\in\cC$ 
\begin{equation}\label{eq:SC0}
u(x,z)=A(z)\e^{\beta_1 x}+ B(z) \e^{\beta_2 x},
\end{equation}
for some maps $A(z)$ and $B(z)$ which we find it convenient to define (with a slight abuse of notation) as $A,B:\overline \cC\to\R$. In particular, in order to guarantee that $u$ satisfies the regularity required in (ii), it must be 
$A,B\in C^1(\overline \cC)\cap C^2(H_\le\cup H_{[0,\alpha]}\cup H_{[\alpha,b]})$.
For $z\geq 0$, the boundary condition at $x=0$ (sixth equation in \eqref{eq:fbp}) leads to $A(z)+B(z)=\hat v(z)$, and by differentiation 
\begin{equation}\label{eq:SC1}
A'(z)+B'(z)=\hat v'(z),\quad z\in[0,b(0)].
\end{equation}
We are now ready to construct the solution to the free-boundary problem.
\smallskip

\noindent
\textbf{Step 1}. Here we take $(x,z)\in H_{[\alpha,b]}$. The fourth equation in \eqref{eq:fbp} yields $\partial_z u(x,b(x))=1$, which can be written equivalently as $\partial_z u(c(z),z)=1$. That is, 
\begin{equation}\label{eq:SC2}
A'(z)\e^{\beta_1 c(z)}+B'(z)\e^{\beta_2 c(z)}=1,\quad z \in [b(a_0),b(0)].
\end{equation} 
We deduce from \eqref{eq:SC1} and \eqref{eq:SC2} that:
\begin{equation}\label{eq:SC3}
A'(z) (\e^{\beta_1 c(z)}- \e^{\beta_2 c(z)})=1- \hat v'(z)\e^{\beta_2 c(z)}.
\end{equation} 

The fifth equation in \eqref{eq:fbp} (so-called smooth-pasting condition) can be written equivalently as $\partial_{zx} u (c(z),z) =0$. That is, \begin{equation}\label{eq:SC4}
A'(z)\beta_1 \e^{\beta_1 c(z)}+B'(z) \beta _2 \e^{\beta_2 c(z)}=0,\quad z \in [b(a_0),b(0)].
\end{equation} 
Combining \eqref{eq:SC4} with \eqref{eq:SC1} yields
\begin{equation}\label{eq:SC5}
A'(z) (\beta_1 \e^{\beta_1 c(z)}- \beta_2 \e^{\beta_2 c(z)})=- \hat v'(z) \beta_2 \e^{\beta_2 c(z)}\quad z \in [b(a_0),b(0)].
\end{equation}

Solving \eqref{eq:SC5} and \eqref{eq:SC3} for $A'(z)$ and equating the two expressions, we find
\begin{equation}\label{eq:SC5.1} 
\hat v'(z)= \frac{\beta_1 \e^{\beta_1 c(z)}- \beta_2 \e^{\beta_2 c(z)}}{(\beta_1-\beta_2)\e^{(\beta_1+\beta_2)c(z)}}.
\end{equation}
Next we show that for each $z\in[b(a_0),b(0)]$ there is a unique $c(z)$ that solves \eqref{eq:SC5.1}. More precisely, we notice that \eqref{eq:SC5.1} was derived only for $z\ge 0$ (because we used \eqref{eq:SC1}). Therefore, as part of the proof we must show that $b(a_0)\ge 0$ (and even $b(a_0)>0$).

Setting
\begin{equation} \label{eq:SC6}
\phi(\ell):=\frac{\beta_1 \e^{\beta_1 \ell}- \beta_2 \e^{\beta_2 \ell}}{(\beta_1-\beta_2)\e^{(\beta_1+\beta_2)\ell}},\quad \ell\in[0,\infty),
\end{equation}
it is immediate to see that $\phi(0)=1$ and simple algebra allows to check that $\phi'> 0$ on $(0,\infty)$ and that $\lim_{+\infty}\phi=+\infty$. Since $\hat v'(z)>1$ for $z\in[0,\hat a)$ and $\hat v'(\hat a)=1$, we then obtain $c(\hat a)=0$. Letting $z$ decrease, starting from $z=\hat a$, the function $c(z)$ increases. That is, $z\mapsto c(z)$ is strictly decreasing in a left-neighbourhood of $\hat a$. Moreover, in such neighbourhood we have 
\begin{equation}\label{eq:SC7}
c(z)=\phi^{-1}\big(\hat v'(z)\big).
\end{equation}
For the inverse function $b(x)=c^{-1}(x)$ we have $b(0)=\hat a$ and $x\mapsto b(x)$ (strictly) decreasing on a right-neighbourhood of $0$. Now we want to show that $b(x)>0$ for $x\in[0,a_0]$.

Recall from \eqref{eq:w} that $v'_0(x)=C_0 (\beta_1\e^{\beta_1 x} - \beta_2 \e^{\beta_2 x})$, that $v_0'(a_0)=1$ and $v_0''(a_0)=0$. From those expressions we deduce 
\begin{equation}\label{eq:C0}
C_0= \frac{1}{\beta_1 \e^{\beta_1 a_0} - \beta_2 \e^{\beta_2 a_0}} \quad\text{ and }\quad \beta_1^2 \e^{\beta_1 a_0}= \beta_2^2 \e^{\beta_2 a_0}.
\end{equation} 
In particular, plugging the second expression into the first one yields $C_0=\beta_2/[\beta_1(\beta_2-\beta_1)]\e^{-\beta_1 a_0}$.
We can thus rewrite \eqref{eq:SC6}, for $\ell \in [0,a_0]$, as
\begin{equation*}
\begin{aligned}
\phi(\ell)=\frac{\left(\beta_1 \e^{\beta_1 \ell}- \beta_2 \e^{\beta_2 \ell}\right)\e^{-(\beta_1+\beta_2)\ell}}{(\beta_1-\beta_2)}=\frac{\beta_2}{\beta_1(\beta_2-\beta_1)}\e^{- \beta_1 a_0}[\beta_1 e^{\beta_1 (a_0-\ell)}-\beta_2 \e^{\beta_2 (a_0-\ell)}]=v_0'(a_0-\ell).
\end{aligned}
\end{equation*}
The equation \eqref{eq:SC5.1} is therefore equivalent to
\begin{equation}\label{eq:SC5.2}
\hat v'(z)=v_0'(a_0-c(z))\iff \hat v'\big(b(x)\big)=v_0'\big(a_0-x\big).
\end{equation}

We notice that $\hat v \geq v_0$ by definition of the problem: any admissible control for the problem with drift $\mu_0$ is admissible for the problem with drift $\hat \mu$ and it gives a weakly larger payoff in the latter because $\hat \mu>\mu_0$.  Therefore $\hat v'(0+)\geq v_0'(0+)$ since $\hat v(0)=v_0(0)=0$.

Arguing by contradiction, let us assume $x_0:=\inf\{x>0:b(x)=0\}\in[0,a_0]$. Then $\hat v'(0+)=v_0'(a_0-x_0)\le v_0'(0+)$, where the final inequality is by concavity of $v_0$. Combining with the inequality from the paragraph above it must be $\hat v'(0+)=v_0'(0+)>0$.
Recalling the ODE solved by the value function of the optimal dividend problem (cf.\ \eqref{eq:fbpdiv}) and using $\hat v(0)=v_0(0)=0$ and $\hat v'(0+)=v_0'(0+)>0$, we deduce
$\hat v''(0+)- v_0''(0+) = 2\sigma^{-2}(\mu_0 - \hat \mu) v_0'(0+) < 0$.
The latter implies $\hat v'(x)-v'_0(x)<0$ for $x\in(0,\eps)$ and sufficiently small $\eps>0$. Thus $\hat v(x)<v_0(x)$ for $x\in(0,\eps)$, which contradicts $\hat v\ge v_0$. Then we conclude $x_0\notin [0,a_0]$, as needed. The latter also implies that $\alpha=b(a_0)$ is the unique solution of $\hat v'(\alpha)=\phi(a_0)=v_0'(0)$.

By construction (cf.\ \eqref{eq:SC7}), $b(x)=[(\hat v')^{-1} \circ \phi](x)$ is continuous for $x\in[0,a_0]$. Moreover, using that $\hat v\in C^2([0,\infty))$ with $\hat v''(z)<0$ for $z\in[0,\hat a)$ we have
\begin{equation*}
b'(x)=\frac{\phi'(x)}{(\hat v''\circ(\hat v')^{-1}\circ \phi)(x)}=\frac{\phi'(x)}{(\hat v''\circ b)(x)}\in (-\infty,0),\quad \text{for $x\in(0,a_0]$}.
\end{equation*}
Letting $x\downarrow 0$ we have $b(0)=\hat a$ and $\hat v''(\hat a)=0$ in the denominator of the equation above. However, also the numerator vanishes. Then, using De L'Hopital's rule, in the limit as $x\downarrow 0$ we have
\begin{equation}\label{eq:DH}
b'(0+)\approx\frac{\phi''(0)}{\hat v'''\big(b(0)\big)b'(0+)}\implies \big[b'(0+)\big]^2=\frac{\phi''(0)}{\hat v'''(\hat a)}.
\end{equation}
Simple algebra yields $\phi'''(0)=-\beta_1\beta_2$. Differentiating once the ODE for $\hat v$ and imposing the boundary conditions $\hat v'(\hat a)=1$ and $\hat v''(\hat a)=0$ yields $\hat v'''(\hat a)=2r/\sigma^2$. Hence, from \eqref{eq:DH} we conclude $b'(0+)=\sigma\sqrt{|\beta_1\beta_2|/(2r)}$ and $b\in C^1([0,a_0])$ as claimed.

Note that \eqref{eq:SC5} (or \eqref{eq:SC3}) together with the explicit expression of $c$ in \eqref{eq:SC7} determine $A'$, on the interval $[\alpha,\hat a]$. They also determine $B'$ on $[\alpha,\hat a]$ by \eqref{eq:SC1} and, finally, $\partial_z u$ on $H_{[\alpha,b]}$. More precisely, \eqref{eq:SC5} and \eqref{eq:SC4} yield
\begin{equation}\label{eq:AB1} 
\begin{aligned}
&A'(z)= \frac{-\beta_2 \e^{\beta_2c(z)}}{\beta_1 \e^{\beta_1 c(z)}- \beta_2 \e^{\beta_2 c(z)}}\hat v'(z)  > 0,\\ 
&B'(z)=(-\beta_1 /\beta_2) \e^{(\beta_1-\beta_2) c(z)}A'(z)= \frac{\beta_1 \e^{\beta_1 c(z)}}{\beta_1 \e^{\beta_1 c(z)}- \beta_2 \e^{\beta_2 c(z)}}\hat v'(z) > 0.
\end{aligned}
\end{equation}
For future reference notice that 
\begin{equation}\label{eq:Aalpha}
A'(\alpha)=-\frac{\beta_2}{\beta_1-\beta_2}\e^{-\beta_1 a_0}.
\end{equation}

Combining the above we have an explicit expression for $\partial_z u(x,z)=A'(z)e^{\beta_1 x}+B'(z)e^{\beta_2 x}$ for $(x,z)\in H_{[\alpha,b]}$. We use that to compute $\partial_{xz} u (x,z)$ for $(x,z)\in H_{[\alpha,b]}$. In particular, for $x<c(z)$, $z\in[\alpha,\hat a]$,  
\begin{equation*} 
\partial_{xz} u (x,z)=  A'(z)\beta_1 \e^{\beta_1 x}+ B'(z)\beta_2 \e^{\beta_2 x}< A'(z)\beta_1 \e^{\beta_1 c(z)}+ B'(z)\beta_2 \e^{\beta_2 c(z)}=0,
\end{equation*}
where the final equality holds due to \eqref{eq:SC4}. Since $\partial_z u(c(z),z) = 1$ (c.f.\ \eqref{eq:SC2}), then $\partial_z u(x,z) > 1$ for $x<c(z)$, $z\in[\alpha,\hat a]$. This verifies the third condition in \eqref{eq:fbp} in the set $\cC\cap H_{[\alpha, b]}$.
Finally, it is a matter of simple algebra to check that $\partial_{xx}u$ and $\partial_{zz}u$ are also continuous on $H_{[\alpha,b]}$.

In this step we have obtained formulae for the coefficients $A'(z)$, $B'(z)$ and the boundary $b(x)$ (or its inverse $c(z)$). Then we have a {\em candidate} expression for $\partial_z u(x,z)$ and for $b(x)$. To emphasise that these are just candidates for now, we adopt the notation
\begin{equation}\label{eq:sum1}
Q_1(x,z):=A'(z)\e^{\beta_1 x}+B'(z)\e^{\beta_2 x}=\partial_z u(x,z),\quad (x,z)\in H_{[\alpha,b]}.
\end{equation}

\noindent\textbf{Step 2}. Here we take $(x,z)\in H_{[0,\alpha]}$.
For $z\in [0,\alpha)$, recall the reflection condition at $(a_0,z)$ as given in the final equation in \eqref{eq:fbp}. Differentiating that condition with respect to $z$ yields $\partial_{zz}u(a_0,z)=\partial_{xz}u(a_0,z)$. Using \eqref{eq:SC0}, the latter reads 
\begin{equation}\label{eq:uzzuxz}
A''(z)\e^{\beta_1 a_0}+B''(z)\e^{\beta_2 a_0} = A'(z)\beta_1 \e^{\beta_1 a_0}+ B'(z)\beta _2 \e^{\beta_2 a_0}. 
\end{equation}
Differentiating also \eqref{eq:SC1} we obtain $B''(z)=\hat v ''(z) - A''(z)$, which we plug into the equation above to obtain
\begin{equation*}
A''(z)\e^{\beta_1 a_0}+ (\hat v ''(z) - A''(z))\e^{\beta_2 a_0} = A'(z)\beta_1 \e^{\beta_1 a_0}+ (\hat v '(z) - A'(z))\beta _2 \e^{\beta_2 a_0}. 
\end{equation*}
Thus, $A'$ satisfies a linear non-homogeneous ODE on the interval $[0,\alpha]$ given by
\begin{equation} \label{eq:SC8}
A''(z)(\e^{\beta_1 a_0}-\e^{\beta_2 a_0})   = A'(z)(\beta_1 \e^{\beta_1 a_0} -\beta _2 \e^{\beta_2 a_0})  + \hat v '(z)\beta _2 \e^{\beta_2 a_0}-\hat v ''(z)\e^{\beta_2 a_0}.
\end{equation}
We impose the boundary condition at $z=\alpha$ given by \eqref{eq:Aalpha}. That ensures that 
\begin{equation}\label{eq:A'cnt}
\text{$A'$ is continuous in $[0,b(x)]$ for every $x\in[0,a_0]$ (hence in $H_{[0,\alpha]}\cup H_{[\alpha,b]}$). }
\end{equation}
Knowledge of $A'$ yields also the function $B'$ thanks to \eqref{eq:SC1}. Therefore, by continuity of $A'$ we deduce that $A''$ is also continuous by \eqref{eq:SC8} and, thanks to \eqref{eq:uzzuxz} and \eqref{eq:SC0},
\begin{equation}\label{eq:contdz}
\text{functions $B'$, $B''$, $\partial_z u$, $\partial_{zz}u$, $\partial_{zx}u$ and $\partial_{zxx}u$ are continuous in $H_{[0,\alpha]}\cup H_{[\alpha,b]}$.}
\end{equation}

Now we analyse properties of $A'$ in more detail. Plugging the expression for $A'(\alpha)$ from \eqref{eq:Aalpha} into \eqref{eq:SC8} and recalling that $c(\alpha)=a_0$, we have  
\begin{equation}\label{eq:A''alpha}
A''(\alpha-)=-\frac{\hat v ''(\alpha)\e^{\beta_2 a_0}}{\e^{\beta_1 a_0}-\e^{\beta_2 a_0}} >0.
\end{equation}

We want to prove that $A''(z)>0$ on $[0,\alpha]$. Arguing by contradiction we assume that $A''$ vanishes on $[0,\alpha)$ and we let $\tilde z$ denote the largest zero of $A''$. Clearly $\tilde z<\alpha$ by \eqref{eq:A''alpha}. By differentiating \eqref{eq:SC8} and taking $z=\tilde z$, we obtain
\begin{equation}\label{eq:A'''} 
A'''(\tilde z)(\e^{\beta_1 a_0}-\e^{\beta_2 a_0}) =g'(z),
\end{equation}
where 
$g(z):= \hat v '(z)\beta _2 \e^{\beta_2 a_0}-\hat v ''(z)\e^{\beta_2 a_0}$. 
We claim (and we prove it later) that $g'(z)<0$ for $z\in[0,\alpha]$. Thus $A'''(\tilde z)<0$ because $\e^{\beta_1 a_0}>\e^{\beta_2 a_0}$. Since $A''(\tilde z)=0$ and $A''$ does not change its sign on $(\tilde z,\alpha)$ (by definition of $\tilde z$), it then follows that $A''(z)<0$ for $z\in(\tilde z,\alpha)$. That contradicts \eqref{eq:A''alpha}.

We want to show that $g$ is strictly decreasing for $z\in[0,\alpha]$. With the notation from \eqref{eq:dnot}, recalling the expression for the value function of the classical dividend problem in \eqref{eq:w} and setting $\hat \beta_j=\beta_j(\hat \mu)$, $j=1,2$ and $\hat C=C(\hat \mu)$, we can express $g$ as
\begin{equation*}
\begin{aligned} 
g(z) &= \hat C \e^{\beta_2 a_0} \left( \beta_2(\hat \beta_1 \e^{\hat \beta_1 z} - \hat \beta_2 \e^{\hat \beta_2 z}) - \hat \beta_1^2 \e^{\hat \beta_1 z} + \hat \beta_2^2 \e^{\hat \beta_2 z}  \right)  \\
&= \hat C \e^{\beta_2 a_0} \left((\beta_2-\hat \beta_1)\hat \beta_1 \e^{\hat \beta_1 z} + \hat \beta_2(\hat \beta_2-\beta_2) \e^{\hat \beta_2 z}  \right),
\end{aligned}
\end{equation*}
on the interval $z\in [0,\hat a)\supset [0,\alpha]$. Differentiating that expression yields
\begin{equation*}
g'(z) = \hat C \e^{\beta_2 a_0} \left((\beta_2-\hat \beta_1)\hat \beta_1^2 \e^{\hat \beta_1 z} + \hat \beta_2^2(\hat \beta_2-\beta_2) \e^{\hat \beta_2 z}  \right),
\end{equation*}
Observing that $\hat \mu > \mu_0\implies\hat \beta_2 < \beta_2$ we deduce that $g'(z)<0$ from the expression above (recall that $\hat C>0$). 

Next, we want to check the third condition in \eqref{eq:fbp}, i.e., $\partial_z u(x,z)>1$ for $(x,z)\in H_{[0,\alpha]}$. By direct calculation we have 
\begin{equation*}
\begin{aligned}
&\partial_{zz}u (x,z) = A''(z)\e^{\beta_1 x}+ (\hat v ''(z) - A''(z))\e^{\beta_2 x},\\
&\partial_{zzx}u (x,z) = A''(z)\beta_1 \e^{\beta_1 x}+ (\hat v ''(z) - A''(z))\beta_2 \e^{\beta_2 x}.
\end{aligned}
\end{equation*}
As shown earlier, for $z\in[0,\alpha]$ we have $A''(z)>0$. Moreover, $\hat v$ is concave and therefore $\hat v''(z)-A''(z)<0$. We deduce for $(x,z)\in H_{[0,\alpha]}$ 
\begin{equation}\label{eq:uz}
\begin{aligned}
\partial_{zz}u (x,z) &= A''(z)\e^{\beta_1 x}+ (\hat v ''(z) - A''(z))\e^{\beta_2 x} \\
&\leq A''(z)\e^{\beta_1 a_0}+ (\hat v ''(z) - A''(z))\e^{\beta_2 a_0}=\partial_{zz}u (a_0,z)=\partial_{zx}u (a_0,z)\le 0,
\end{aligned}
\end{equation}
where in the third equality we use the reflection condition from the final equation in \eqref{eq:fbp} and it remains to prove the final inequality.
For that, it is sufficient to observe that $\beta_2(\hat v''(z)-A''(z))>0$ implies 
\begin{equation*}
\partial_{zzx} u (a_0,z)= A''(z)\beta_1 \e^{\beta_1 a_0}+ (\hat v ''(z) - A''(z))\beta _2 \e^{\beta_2 a_0}>0.
\end{equation*}
Then, thanks to the fact that $A'$ is continuous at $\alpha$, choosing $z=\alpha$ in \eqref{eq:SC5} so that $c(\alpha)=a_0$ (recall $B'(\alpha)=\hat v'(\alpha)-A'(\alpha)$) we have
\begin{equation*}
\partial_{zx}u (a_0,\alpha)=A'(\alpha)\beta_1 \e^{\beta_1 a_0}+ (\hat v '(\alpha) - A'(\alpha))\beta _2 \e^{\beta_2 a_0} =0.
\end{equation*}
Combining these last two expressions yields the final inequality in \eqref{eq:uz}.

Finally, by construction $\partial_z u$ is continuous at $(x,\alpha)$ for all $x\in[0,a_0]$ and the value of $\partial_z u(x,\alpha)>1$ was 	determined in Step 1 of this proof. Therefore, \eqref{eq:uz} implies that $\partial_z u(x,z)>1$ for all $(x,z)\in H_{[0,\alpha]}$ as needed.

As in Step 1, also in this step we have obtained expressions for the coefficients $A'(z)$ an $B'(z)$. We use them to construct a candidate expression for $\partial_z u(x,z)$. To emphasise this fact we set
\begin{equation}\label{eq:sum2}
Q_2(x,z):=A'(z)\e^{\beta_1 x}+B'(z)\e^{\beta_2 x}=\partial_z u(x,z),\quad (x,z)\in H_{[0,\alpha]}.
\end{equation}

\noindent
\textbf{Step 3}. The set $(x,z)\in H_\le$ is analysed in this step. 
The boundary condition $u(x,-x)=0$ reads, for $z=-x$, 
\begin{equation}\label{eq:SC9}
A(z)\e^{-\beta_1 z}+ B(z) \e^{-\beta_2 z}=0,
\end{equation}
and by differentiation
\begin{equation}\label{eq:SC10}
B'(z)=[(\beta_1-\beta_2)A(z)-A'(z) ]\e^{(\beta_2-\beta_1) z}.
\end{equation}
Note that in particular $B(0)=-A(0)$.
The reflection condition $(\partial_z u-\partial_x u)(a_0,z)=0$ (last equation in \eqref{eq:fbp}) implies
\begin{equation}\label{eq:SC11}
A'(z)\e^{\beta_1 a_0}+B'(z)\e^{\beta_2 a_0} = A(z)\beta_1 \e^{\beta_1 a_0}+ B(z)\beta _2 \e^{\beta_2 a_0}.
\end{equation}
Combining \eqref{eq:SC10} and \eqref{eq:SC11}, we obtain
\[
A'(z)\e^{\beta_1 a_0}+[(\beta_1-\beta_2)A(z)-A'(z) ]\e^{(\beta_2-\beta_1) z}\e^{\beta_2 a_0} = A(z)\beta_1 \e^{\beta_1 a_0}-A(z)\e^{(\beta_2-\beta_1) z}\beta _2 \e^{\beta_2 a_0}, 
\]
which leads to
$A'(z)=\beta_1 A(z)$ and thus 
$A(z)= A(0) \e^{\beta_1 z}$ and $B(z)=-A(0)\e^{\beta_2 z}$ for $z\in [-a_0,0)$.
Plugging this expression back into \eqref{eq:SC9} yields the form of $A(z)$ and $B(z)$, and then
\begin{equation}\label{eq:SC12}
u(x,z)= A(0)(\e^{\beta_1(x+z)}- \e^{\beta_2(x+z)}),\quad (x,z)\in H_\le\,.
\end{equation}

In order to determine $A(0)$ we can, for example, impose $\partial_z u(a_0,0+)=\partial_z u(a_0,0-)$, noticing that $\partial_z u(a_0,0+)=Q_2(a_0,0+)>1$ is obtained from Step 2 (see \eqref{eq:sum2}). More precisely, we have
\begin{equation}\label{eq:A0_C0} 
A(0)=\frac{Q_2 (a_0,0+)}{\beta_1\e^{\beta_1 a_0} - \beta_2 \e^{\beta_2 a_0}} >\frac{1}{\beta_1\e^{\beta_1 a_0} - \beta_2 \e^{\beta_2 a_0}}=C_0, 
\end{equation}
where the last equality is \eqref{eq:C0}. Therefore, for $z<0$,
\begin{equation}\label{eq:uzQ2} 
\partial_z u (x,z)= \frac{Q_2 (a_0,0+)}{\beta_1\e^{\beta_1 a_0} - \beta_2 \e^{\beta_2 a_0}} (\beta_1 \e^{\beta_1(x+z)}- \beta_2 \e^{\beta_2(x+z)}). 
\end{equation}
Taking one more derivative and observing that $\beta_1(x+z)\le \beta_1 a_0$, $\beta_2 z\ge 0$ and $\beta_2 x\ge \beta_2 a_0$ for $(x,z)\in H_\le$ we obtain
\begin{equation*}
\begin{aligned}
 \partial_{zz} u (x,z) &= \frac{\partial_z u (a_0,0+)}{\beta_1\e^{\beta_1 a_0} - \beta_2 \e^{\beta_2 a_0}} (\beta_1^2 \e^{\beta_1(x+z)}- \beta_2^2 \e^{\beta_2(x+z)}) \\
 &\leq \frac{\partial_z u (a_0,0+)}{\beta_1\e^{\beta_1 a_0} - \beta_2 \e^{\beta_2 a_0}} (\beta_1^2 \e^{\beta_1 a_0}- \beta_2^2 \e^{\beta_2 a_0})=\frac{\partial_z u (a_0,0+)}{\beta_1\e^{\beta_1 a_0} - \beta_2 \e^{\beta_2 a_0}} \frac{v_0''(a_0)}{C(\mu_0)}=0,
 \end{aligned}
 \end{equation*}
where the penultimate equality follows from \eqref{eq:w}, upon recalling that we are working with $\beta_1=\beta_1(\mu_0)$ and $\beta_2=\beta_2(\mu_0)$, and the final equality is by the smooth pasting condition \eqref{eq:smoothp}. We then conclude that $\partial_z u (x,z)>1$ for $(x,z)\in H_\le$.

It is important to notice that, differently from Steps 1 and 2, here we have obtained an expression for the candidate value function $u(x,z)$, $(x,z)\in H_\le$, instead of a candidate for its $z$-derivative. We are going to use this fact in the fourth and final step of the proof. 
\medskip

\noindent
\textbf{Step 4}. In this step we piece together the functions obtained in the previous steps and confirm that our candidate pair $(u,b)$ solves the free boundary problem \eqref{eq:fbp} as claimed.

We define 
$Q(x,z):=Q_1 1_{H_{[\alpha,b]}}(x,z)+Q_2 1_{H_{[0,\alpha]}}(x,z)+1_\cS(x,z)$ on $[0,a_0]\times [0,\infty)$, with $Q_1$ and $Q_2$ from \eqref{eq:sum1} and \eqref{eq:sum2}, 
By construction and thanks to the smooth pasting condition $\partial_{zx} u(x,b(x))=0$, the functions $Q$ and $\partial_x Q(x,z)$ are continuous on $[0,a_0]\times[0,\infty)$ (see \eqref{eq:A'cnt} and \eqref{eq:contdz}). Its derivative $\partial_{xx} Q$ is continuous on $H_{[0,\alpha]}\cup H_{[\alpha,b]}$ and it belongs to $L^\infty([0,a_0]\times[0,\infty))$. Next
we define 
\begin{equation}\label{eq:SC13}
u(x,z)= u(x,0-)+ \int_0^z Q(x,\zeta)\ud\zeta,\quad\text{for $(x,z)\in [0,a_0]\times [0,\infty)$},
\end{equation}
where $u(x,0-)$ is given by \eqref{eq:SC12} in Step 3. On the set $H_\le$ the function $u(x,z)$ is defined by \eqref{eq:SC12}.

It is then clear that $u\in C(H)$. By Step 3 we have $\partial_x u,\partial_{xx}u,\partial_z u\in C(H_\le)$. By dominated convergence in $[0,a_0]\times[0,\infty)$ we have 
\begin{equation*}
\begin{aligned}
\partial_x u(x,z)=\partial_x u(x,0-)+\int_0^z\partial_x Q(x,\zeta)\ud \zeta\quad\text{and}\quad\partial_{xx} u(x,z)=\partial_{xx} u(x,0-)+\int_0^z\partial_{xx} Q(x,\zeta)\ud \zeta.
\end{aligned}
\end{equation*}
Moreover, it is clear that $\partial_z u(x,z)=Q(x,z)$ on $[0,a_0]\times[0,\infty)$. Then, thanks to the regularity of $Q$ stated above, these derivatives paste continuously with the corresponding derivatives in $H_\le$. That is, $u\in C^1(H)$ and $\partial_{xx}u(x,z)\in C(\overline \cC\cup \cS)$ as required in the free boundary problem.

By construction $u$ solves the free boundary problem in $H_\le$. Next, for $(x,z)\in H_{[0,\alpha]}\cup H_{[\alpha,b]}\cup \cS$ we can compute directly
\begin{equation*}
\begin{aligned}
\big(\tfrac{\sigma^2}{2}\partial_{xx}u\!+\!\mu_0\partial_xu\!-\!r u\big)(x,z)=\big(\tfrac{\sigma^2}{2}\partial_{xx}u\!+\!\mu_0\partial_xu\!-\!r u\big)(x,0-)\!+\!\int_0^z\!\!\big(\tfrac{\sigma^2}{2}\partial_{xx}Q\!+\!\mu_0\partial_xQ\!-\!r Q\big)(x,\zeta)\ud\zeta.
\end{aligned}
\end{equation*}
In particular, for $(x,z)\in H_{[0,\alpha]}\cup H_{[\alpha,b]}$ the right hand side of the equation above vanishes 
using the explicit expressions \eqref{eq:SC12}, \eqref{eq:sum2} and \eqref{eq:sum1}. Instead, for $(x,z)\in\cS$ we can compute
\begin{equation*}
\begin{aligned}
&\big(\tfrac{\sigma^2}{2}\partial_{xx}u+\mu_0\partial_xu-r u\big)(x,z)\\
&=\big(\tfrac{\sigma^2}{2}\partial_{xx}u+\mu_0\partial_xu-r u\big)(x,0-)+\int_0^{b(x)}\big(\tfrac{\sigma^2}{2}\partial_{xx}Q+\mu_0\partial_xQ-r Q\big)(x,\zeta)\ud\zeta-r(z-b(x))\\
&=-r(z-b(x))\le0.
\end{aligned}
\end{equation*}
This shows validity of the first two equations in \eqref{eq:fbp}. The third, fourth, fifth and seventh condition in \eqref{eq:fbp} hold by construction and by properties of \eqref{eq:SC12}, \eqref{eq:sum2} and \eqref{eq:sum1} illustrated in Steps 1--3. For the sixth condition in \eqref{eq:fbp} we observe that $u(0,0)=0$ due to \eqref{eq:SC12} and therefore
 \begin{equation*}
 u(0,z)= u(0,0-)+ \int_0^z Q(0,\zeta)\ud\zeta=\int_0^{\hat a\wedge z} \big(A'(\zeta)+B'(\zeta)\big)\ud \zeta + \int_{\hat a\wedge z}^z 1 \ud \zeta=\hat v(z),
 \end{equation*}
where the final equality is due to \eqref{eq:SC1}, and the facts that $\hat v'=1$ on $[\hat a, \infty)$ and $\hat v(0)=0$.
The final condition in \eqref{eq:fbp} can be checked as follows: for $z\in[-a_0,0]$,
$\big(\partial_z u-\partial_x u\big)(a_0,z)=0$
by \eqref{eq:SC11}; for $z\in[0,\alpha]$
\begin{equation*}
\begin{aligned}
\big(\partial_z u-\partial_x u\big)(a_0,z)&=Q_2(a_0,z)-\partial_x u(a_0,0-)-\int_0^z\partial_x Q_2(a_0,\zeta)\ud \zeta\\
&=Q_2(a_0,z)-\partial_x u(a_0,0-)-\int_0^z\partial_z Q_2(a_0,\zeta)\ud \zeta=Q_2(a_0,0)-\partial_x u(a_0,0-)=0
\end{aligned}
\end{equation*}
where the second equality is by \eqref{eq:uzzuxz} and the final one by \eqref{eq:uzQ2}, upon recalling that $\partial_x u(a_0,0-)=\partial_z u(a_0,0-)$ by construction; finally, for $z\in[\alpha,\infty)$, we have
\begin{equation*}
\begin{aligned}
&\big(\partial_z u\!-\!\partial_x u\big)(a_0,z)=Q(a_0,z)\!-\!\partial_x u(a_0,0-)\!-\!\int_0^z\!\!\!\partial_x Q(a_0,\zeta)\ud \zeta\\
&=1\!-\!\partial_x u(a_0,0-)\!-\!\int_0^\alpha\!\!\partial_z Q_2(a_0,\zeta)\ud \zeta=Q_2(a_0,0)\!-\!\partial_x u(a_0,0-)\!+\!1\!-\!Q_2(a_0,\alpha)=0,
\end{aligned}
\end{equation*}
where in the second equality we use that $Q(a_0,z)=1$ for $z\in[\alpha,\infty)$ and in the final one we use $Q_2(a_0,\alpha)=1$ (thanks to \eqref{eq:A'cnt} and \eqref{eq:contdz}) and $Q_2(a_0,0)=\partial_x u(a_0,0-)$ as above.

\end{proof}
\medskip

The pair $(u,b)$ constructed in the previous proposition is going to be used in a verification result to determine $u_2$ (cf.\ \eqref{eq:u2}) and to prove optimality of $D^*$ for Player 2 (cf.\ \eqref{eq:D*}). Let us start by noticing that for 
\begin{equation*}
D^*_t=\sup_{0\le s\le t}\Big(Y^0_s-X^*_s-b(X^*_s)\Big)^+,\quad\text{for $t\ge 0$},
\end{equation*}
where $Y^0_s=y+\mu_0s+\sigma B_s$. As a result, the pair $(X^*,Y^*)$ introduced in Theorem \ref{thm:NE1} can be expressed as 
\begin{equation}\label{eq:X*Y*}
X^*_t=X^0_t-\sup_{0\le s\le t}\Big(X^0_s-a_0\Big)^+\quad\text{and}\quad
Y^*_t=Y^0_t-\sup_{0\le s\le t}\Big(Y^0_s-X^*_s-b(X^*_s)\Big)^+,
\end{equation}
with $X^0_s=x+\mu_0s+\sigma B_s$.
It is useful to determine some properties of $(L^*,D^*)$  relatively to the pair $(X^*,Z^*)$, where $Z^*:=Y^*-X^*$.  Let us set $\gamma_*\coloneqq\gamma_{X^*}\wedge\gamma_{Z^*}=\inf\{t\ge 0: Z^*_t\le -X^*_t\,\text{or}\,X^*_t\le 0\}$.
\begin{lemma}\label{lem:D*} 
For $b$ constructed in Proposition \ref{prop:SC}
and $(L^*,D^*)$ defined as in Theorem \ref{thm:NE1}, it holds: 
\begin{itemize}
\item[(i)]  $D^*$ is continuous except for a possible jump at time zero of size 
\begin{equation*}
D^*_0=\big(y-x+L^*_0-b(X^*_0)\big)^+=\big(y-x\wedge a_0-b(x\wedge a_0)\big)^+.
\end{equation*}
Therefore, $Z^*$ is continuous except for a possible jump at time zero of size 
\begin{equation*}
Z^*_0-Z^*_{0-}=L^*_0-D^*_0=(x-a_0)^+-\big(y-x\wedge a_0-b(x\wedge a_0)\big)^+.
\end{equation*}
\item[(ii)] $\P(Z^*_{t\wedge\gamma_*}\le b(X^*_{t\wedge\gamma_*}),\:\forall t\ge 0)=1$.
\item[(iii)] For any $S>0$, we have
$\int_{(0,S\wedge\gamma_*]}1_{\{Z^*_t<b(X^*_t)\}}\ud D^*_t=0$, $\P$-a.s. 
\end{itemize}
\end{lemma}
\begin{proof}
Continuity of the mapping $t\mapsto D^*_{t\wedge\gamma_*}$ on $(0,\infty)$ is clear because of continuity of the mapping $t\mapsto L^*_{t\wedge\gamma_*}-b(X^*_{t\wedge\gamma_*})$ on $(0,\infty)$. The expression for $D^*_0$ is also clear by the explicit formula \eqref{eq:D*} and thus the one for $Z^*_0-Z^*_{0-}$ follows immediately. 

The condition in item $(ii)$ can be verified easily upon noticing that $D^*_{t\wedge\gamma_*}\ge z+L^*_{t\wedge\gamma_*}-b(X^*_{t\wedge\gamma_*})$ for every $t\ge 0$. It remains to check the condition in item $(iii)$. 
The argument is carried out path-wise. Fix $\omega\in\Omega$ and assume $Z^*_t(\omega)<b(X^*_t(\omega))$ (with no loss of generality it suffices to consider $t<\gamma_*(\omega)$). Then
\begin{equation*}
Z^*_t(\omega)=z+L^*_t(\omega)-D^*_t(\omega)<b\big(X^*_t(\omega)\big)\implies D^*_t(\omega)>z+L^*_t(\omega)-b\big(X^*_t(\omega)\big).
\end{equation*}
Therefore, by continuity there is $\delta_\omega>0$ such that 
$D^*_t(\omega)>\sup_{t\le s\le t+\delta_\omega}\left(z+L^*_s(\omega)-b\big(X^*_s(\omega)\big)\right)^+$.
The latter implies $D^*_s(\omega)-D^*_t(\omega)=0$ for $s\in[t,t+\delta_\omega)$ which proves the claim.
\end{proof}

Next we state our verification result.
\begin{proposition}\label{prop:D*}
Let $(u,b)$ be the pair constructed in Proposition \ref{prop:SC} and recall $u_2$ from \eqref{eq:u2} and \eqref{eq:extu2}. Set $u(x,z)=u(a_0,z)$ for $x>a_0$ and $z\ge -x$. Then, 
$u(x,z)=u_2(x,z)$, for $x\in[0,\infty)$ and $z\ge -x$.
Moreover, $D^*$ as in \eqref{eq:D*} is an optimal control, i.e., $u_2(x,z+x)=v_2(x,y;L^*)=\cJ^2_{x,y}(D^*,L^*)$ for all $(x,y)\in[0,\infty)^2$.
\end{proposition}
\begin{proof} 
In Proposition \ref{prop:SC} we showed that $u$ solves the free boundary problem \eqref{eq:fbp}. 
Take an arbitrary control $D \in \cD(x+z)$ and take the admissible pair $(D,L^*)$. Denote $Z^D=Z^{L^*,D}$, $\gamma=\gamma_Z\wedge\gamma_{X^*}$ and $\cA:=\frac{\sigma^2}{2}\partial_{xx}+\mu_0\partial_x-r$, for simplicity. Fix $(x,z)\in H$. We wish to apply Dynkin's formula to calculate $\E_{x,z}[\e^{-r(t\wedge\gamma)}u(X^*_{t\wedge\gamma},Z^D_{t\wedge\gamma})]$. This can be done, for example, invoking \cite[Thm.\ 2.1]{Cai}. We must notice that 
$(X^*_{s\wedge\gamma},Z^D_{s\wedge\gamma})\in H$, for all $s\in [0,\infty)$ and, for any $S>0$, 
\begin{equation}\label{eq:BV}
\begin{aligned}
&\int_0^S\P_{x,z}\big( Z^D_s= b(X^*_s),Z^D_s> \alpha\big)\ud s=\int_0^S\P_{x,z}\big(X^*_s=c(Z^D_s),Z^D_s>\alpha\big)\ud s\\
&=\int_0^S\P_{x,z}\big(X^0_s=c(Z^D_s)+L^*_s,Z^D_s>\alpha\big)\ud s=0,
\end{aligned}
\end{equation}
where the final equality holds due to Lemma \ref{lem:BV}. Equation \eqref{eq:BV} guarantees Assumption A.1 in \cite[Thm.\ 2.1]{Cai}, while the other assumptions are satisfied in our set-up thanks to the regularity of $u$ and $b$.

Applying Dynkin's formula, up to a standard localisation argument that makes the stochastic integral a martingale, we obtain
\begin{equation}\label{eq:D1}
\begin{aligned}
&\E_{x,z}\Big[\e^{-r(t\wedge\gamma)}u(X^*_{t\wedge\gamma},Z^D_{t\wedge\gamma})\Big]\\
&=u(x,z)+\E_{x,z}\Big[\int_0^{t\wedge\gamma}\e^{-rs}(\cA u)(X^*_s,Z^D_s)\ud s\Big]\\
&\:\:+\E_{x,z}\Big[\int_0^{t\wedge\gamma}\e^{-rs}\big(\partial_z u-\partial_x u\big)(X^*_s,Z^D_s)\ud L^{*,c}_s-\int_0^{t\wedge\gamma}\e^{-rs}\partial_z u(X^*_s,Z^D_s)\ud D^c_s\Big]\\
&\:\:+\E_{x,z}\Big[\sum_{s\in[0,t\wedge\gamma]}\e^{-rs}\big(u(X^*_s,Z^D_s)-u(X^*_{s-},Z^D_{s-})\big)\Big],
\end{aligned}
\end{equation}  
where $(L^{*,c},D^c)$ is the continuous part of the pair $(L^*,D)$. Thanks to \eqref{eq:BV} we have for a.e.\ $s\in [0,\gamma(\omega)]$
\begin{equation}\label{eq:Au}
(\cA u)\big(X^*_s(\omega),Z^D_s(\omega)\big)=(\cA u)\big(X^*_s(\omega),Z^D_s(\omega)\big)1_{\{(X^*_s(\omega),Z^D_s(\omega))\in\cS\}}\le 0,
\end{equation}
where the first equality is by the first equation in \eqref{eq:fbp} and the inequality by the second equation therein.

Jumps of $D$ do not affect the dynamics of $X^*$ and therefore, by definition of $L^*$ and the fact that $X^*_{0-}\le a_0$ we have $L^*_t=L^{*,c}_t$ for $t\ge 0$. The sum of jumps then reads
\begin{equation*}
\begin{aligned}
&\sum_{s\in[0,t\wedge\gamma]}\e^{-rs}\big(u(X^*_s,Z^D_s)-u(X^*_{s-},Z^D_{s-})\big)
=\sum_{s\in[0,t\wedge\gamma]}\e^{-rs}\big(u(X^*_s,Z^D_s)-u(X^*_{s},Z^D_{s-})\big)
\\
&=-\sum_{s\in[0,t\wedge\gamma]}\e^{-rs}\int_{0}^{|\Delta Z^D_s|}\partial_z u(X^*_s,Z^D_{s-}-\zeta)\ud \zeta\le -\sum_{s\in[0,t\wedge\gamma]}\e^{-rs}\Delta D_s,
\end{aligned}
\end{equation*}
where we use that $\Delta Z^D_s=-\Delta D_s\le 0$ and the third and fourth equation in \eqref{eq:fbp}. By the same two conditions we also obtain
\[
\int_0^{t\wedge\gamma}\e^{-rs}\partial_z u(X^*_s,Z^D_s)\ud D^c_s\ge \int_0^{t\wedge\gamma}\e^{-rs}\ud D^c_s.
\]
Combining these terms we obtain 
\begin{equation}\label{eq:jumps}
\begin{aligned}
&\sum_{s\in[0,t\wedge\gamma]}\e^{-rs}\big(u(X^*_s,Z^D_s)-u(X^*_{s-},Z^D_{s-})\big)-\int_0^{t\wedge\gamma}\e^{-rs}\partial_z u(X^*_s,Z^D_s)\ud D^c_s\\
&\le -\sum_{s\in[0,t\wedge\gamma]}\e^{-rs}\Delta D_s-\int_0^{t\wedge\gamma}\e^{-rs}\ud D^c_s=-\int_{[0,t\wedge\gamma]}\e^{-rs}\ud D_s.
\end{aligned}
\end{equation}
Finally, the last equation in \eqref{eq:fbp} gives
\begin{equation}\label{eq:refl}
\begin{aligned}
&\int_0^{t\wedge\gamma}\e^{-rs}\big(\partial_z u-\partial_x u\big)(X^*_s,Z^D_s)\ud L^{*,c}_s=\int_0^{t\wedge\gamma}\e^{-rs}\big(\partial_z u-\partial_x u\big)(a_0,Z^D_s)\ud L^{*}_s=0,
\end{aligned}
\end{equation}
where we used that $\ud L^{*,c}_s=\ud L^{*}_s=1_{\{X^*_s=a_0\}}\ud L^{*}_s$ for $s>0$ in the first equality.

Combining \eqref{eq:D1}, \eqref{eq:Au}, \eqref{eq:jumps} and \eqref{eq:refl} we obtain
\begin{equation*}
\begin{aligned}
u(x,z)\ge \E_{x,z}\Big[\int_{[0,t\wedge\gamma]}\e^{-rs}\ud D_s+e^{-r(t\wedge\gamma)}u(X^*_{t\wedge\gamma},Z^D_{t\wedge\gamma})\Big].
\end{aligned}
\end{equation*}
On the event $\{\gamma\le t,\gamma_Z>\gamma_X\}$ we have $u(X^*_{t\wedge\gamma},Z^D_{t\wedge\gamma})=u(0,Z^D_{\gamma_X})=\hat v (Z^D_{\gamma_X})$ by the sixth equation in \eqref{eq:fbp}. On the event $\{\gamma\le t,\gamma_Z\le \gamma_X\}$ we have $u(X^*_{t\wedge\gamma},Z^D_{t\wedge\gamma})=u(X^*_{\gamma_Z},-X^*_{\gamma_Z})=0$ by the seventh equation in \eqref{eq:fbp}. Therefore
\begin{equation*}
\begin{aligned}
&\e^{-r(t\wedge\gamma)}u(X^*_{t\wedge\gamma},Z^D_{t\wedge\gamma})=1_{\{\gamma\le t,\gamma_Z>\gamma_X\}}\e^{-r\gamma_X}\hat v (Z^D_{\gamma_X})+1_{\{\gamma>t\}}\e^{-rt}u(X^*_{t},Z^D_{t})\\
&\ge 1_{\{\gamma\le t,\gamma_Z>\gamma_X\}}\e^{-r\gamma_X}\hat v (Z^D_{\gamma_X}),
\end{aligned}
\end{equation*}
where the final inequality uses $u\ge 0$ which is due to $u(x,-x)=0$ and $\partial_z u\ge 1$. Then  
\begin{equation*}
\begin{aligned}
u(x,z)\ge \E_{x,z}\Big[1_{\{t\ge \gamma\}}\Big(\int_{[0,\gamma_X\wedge\gamma_Z]}\e^{-rs}\ud D_s+1_{\{\gamma_Z>\gamma_X\}}\e^{-r\gamma_X}\hat v (Z^D_{\gamma_X})\Big)\Big].
\end{aligned}
\end{equation*}
Letting $t\to \infty$, using Fatou's lemma and noticing that $\P_{x,z}(\gamma_X\wedge\gamma_Z<\infty)=1$ yields
\begin{equation}\label{eq:ver0}
\begin{aligned}
u(x,z)\ge \E_{x,z}\Big[\int_{[0,\gamma_X\wedge\gamma_Z]}\e^{-rs}\ud D_s+1_{\{\gamma_Z>\gamma_X\}}\e^{-r\gamma_X}\hat v (Z^D_{\gamma_X})\Big].
\end{aligned}
\end{equation}

Let us now consider the pair $(L^*,D^*)$ with $(D^*_t)_{t\ge 0}$ as in \eqref{eq:D*}. Denote the controlled dynamics $(X^*,Z^*)=(X^{L^*},Z^{L^*,D^*})$ and the associated stopping time $\gamma_*:=\gamma_{Z^*}\wedge\gamma_{X^*}$. The controlled process is bound to evolve in $\overline \cC$ in the random time-interval $(0,\gamma_*]$, thanks to Lemma \ref{lem:D*}. Then, we can repeat the arguments from above based on Dynkin's formula but now \eqref{eq:Au} reads $(\cA u)(X^*_s(\omega),Z^*_s(\omega))=0$ for $s\in (0,\gamma_*(\omega)]$ and the inequality in \eqref{eq:jumps} becomes an equality, due to Lemma \ref{lem:D*}-(iii) and $\partial_z u(x,b(x))=1$ (see the fourth equation in \eqref{eq:fbp}). Thus we obtain
\begin{equation*}
\begin{aligned}
u(x,z)&= \E_{x,z}\Big[\int_{[0,t\wedge\gamma_*]}\e^{-rs}\ud D^*_s+1_{\{\gamma_*\le t,\gamma_{Z^*}>\gamma_{X^*}\}}\e^{-r\gamma_{X^*}}\hat v (Z^D_{\gamma_{X^*}})+1_{\{\gamma_*>t\}}\e^{-rt}u(X^*_{t},Z^*_{t})\Big].
\end{aligned}
\end{equation*}
Now we let $t\to \infty$. Since $(X^*_t,Z^*_t)\in \overline\cC$ for all $t\in (0,\gamma_*]$ and $\overline \cC$ is compact, then $u$ is bounded on $\overline \cC$ and clearly
\begin{equation*}
\lim_{t\to\infty}\E_{x,z}\Big[1_{\{\gamma_*>t\}}\e^{-rt}u(X^*_{t},Z^*_{t})\Big]=0.
\end{equation*}
Monotone convergence also implies
\begin{equation}\label{eq:ver1}
\begin{aligned}
u(x,z)&=\lim_{t\to\infty}\E_{x,z}\Big[\int_{[0,t\wedge\gamma_*]}\e^{-rs}\ud D^*_s+1_{\{\gamma_*\le t,\gamma_{Z^*}>\gamma_{X^*}\}}\e^{-r\gamma_{X^*}}\hat v (Z^D_{\gamma_{X^*}})\Big]\\
&=\E_{x,z}\Big[\int_{[0,\gamma_{X^*}\wedge\gamma_{Z^*}]}\e^{-rs}\ud D^*_s+1_{\{\gamma_{Z^*}>\gamma_{X^*}\}}\e^{-r\gamma_{X^*}}\hat v (Z^D_{\gamma_{X^*}})\Big],
\end{aligned}
\end{equation}
where, in particular, we used that for each $\omega\in\Omega$ there is $T_\omega>0$ sufficiently large that $\gamma_*(\omega)<t$ for all $t\ge T_\omega$ and therefore 
\begin{equation*}
\int_{[0,t\wedge\gamma_*]}\e^{-rs}\ud D^*_s(\omega)=\int_{[0,\gamma_*]}\e^{-rs}\ud D^*_s(\omega),\quad \text{for all $t\ge T_\omega$}.
\end{equation*}

Combining \eqref{eq:ver0} and \eqref{eq:ver1} we obtain that $u=u_2$ on $H$ and optimality of $D^*$. Notice that the equivalence $u(x,z)=u_2(x,z)$ extends to all $x\in[0,\infty)$ and $z\ge -x$ because of \eqref{eq:extu2}.
\end{proof}

\begin{remark}\label{rem:ineq_v2_v0}
Recalling that $v_0(a_0)=C_0(\e^{\beta_1 a_0} - \e^{\beta_2 a_0})$, it follows from \eqref{eq:SC12} and \eqref{eq:A0_C0}, that 
$v_2(a_0,a_0)=u_2(a_0,0)>v_0(a_0)$.
Using that $\partial_z u_2(a_0,z)\geq 1$ for $z\geq 0$ and $v_0'(y)=1$ for $y\geq a_0$, we deduce that 
$v_2(a_0,y)=u_2(a_0,y-a_0)>v_0(y)$,for all $y\geq a_0$.
\end{remark}
\medskip

Recalling the discussion after Theorem \ref{thm:NE1}, the proposition above has established that when Player 1 uses the strategy $\Phi^*$ from \eqref{eq:phipsi*}, Player 2's best response is the strategy $\Psi^*$. In the next lemma we show the viceversa, i.e., when Player 2 uses the strategy $\Psi^*$, then Player 1's best response is to use $\Phi^*$. That will establish that $(\Phi^*,\Psi^*)$ is a Nash equilibrium in the game with asymmetric endowment, thanks to Lemma \ref{lem:deviations}-(iii).

\begin{lemma}\label{lem:Phi}
Assume Player 2 uses the strategy $\Psi^*$. Then for every $y>x\ge 0$, Player 1's best-response is the strategy $\Phi^*$, i.e., for $L^*=\Phi^*(\,\cdot\,,B)$ we have
\begin{equation}
v_1(x,y;\Psi^*)=\sup_{L \in \cD(x)}\cJ^1_{x,y}(L,\Psi^*(\,\cdot\,,B,L))=\cJ^1_{x,y}(L^*,\Psi^*(\,\cdot\,,B,L^*)).
\end{equation}
\end{lemma}
\begin{proof}
Recalling the expression of $\Psi^*$ in \eqref{eq:maps} and using that $b(x)\ge \alpha>0$, for all $x\ge 0$, then $\Psi^*(t,B,L)\le (y-x +L_t-\alpha)^+$ for all $t\ge 0$ and any control $L$. That implies for Player 2's dynamics:
\begin{equation}\label{eq:last}
\begin{aligned}
Y^*_t=y\!+\!\mu_0 t\!+\!\sigma B_t\!-\! \Psi^*(t,B,L)=X^0_t\!+\!(y\!-vx)\!-\!\Psi^*(t,B,L)\ge X^L_t\!+\!(y\!-\!x\!+\!L_t)\wedge \alpha > X^L_t,
\end{aligned}
\end{equation}
for all $t\ge 0$, any realisation $B(\omega)=(B_s(\omega))_{s\ge 0}$ of the Brownian path and any choice of Player 1's control $L=(L_s)_{s\ge 0}$. Therefore, 
$\gamma_{Y^*}>\gamma_{X^L}$, $\P_{x,y}$-a.s.\ and Player 1's expected payoff reduces to
\begin{equation*}
\cJ^1_{x,y}(L,\Psi^*)=\E_{x,y}\Big[\int_{[0,\gamma_X]}\e^{-r t}\ud L_t\Big].
\end{equation*}
Thus, Player 1 is faced with the classical dividend problem and $L^*$ is optimal.
\end{proof}

Combining the results above we have a simple proof of Theorem \ref{thm:NE1}.
\begin{proof}[Proof of Theorem \ref{thm:NE1}.]
Fix $(x,y) \in [0,\infty)^2$ with $y>x$. Recall that the pair $(\Psi^*,\Phi^*)$ induces the pair of controls $(L^*,D^*)$ with $L^*=\Phi^*(\,\cdot\,,B)$ and $D^*=\Psi^*(\, \cdot,B,L^*)$. From Proposition \ref{prop:D*}, $D^*$ is optimal in the control problem with value $v_2(x,y;L^*)$, and  we have
$\cJ^2_{x,y}(\Psi^*,\Phi^*)=\cJ^2_{x,y}(D^*,L^*)\ge\cJ^2_{x,y}(D,L^*)$, $\forall D \in \cD(y)$.
From Lemma \ref{lem:Phi} we have 
$\cJ^1_{x,y}(\Phi^*,\Psi^*)=\cJ^1_{x,y}(L^*,D^*)\ge\cJ^1_{x,y}(L,\Psi^*)$, $\forall L \in \cD(x)$.
We conclude by Lemma \ref{lem:deviations}-(iii) that $(\Phi^*,\Psi^*)$ is a Nash equilibrium of the game starting at $(x,y)$. 
\end{proof}

\begin{remark}
The equilibrium constructed in this section turns out to be of the form {\em control vs.\ strategy}, in the sense that Player 1 uses a pure control and Player 2 uses a pure strategy (cf.\ Definitions \ref{def:random_strategies} and \ref{def:random_controls}). However, one must notice that it is a Nash equilibrium of the game introduced in Section \ref{sec:strategies} where both players have access to the larger class of randomised strategies, according to Definition \ref{def:NE3}. Thanks to Lemma \ref{lem:deviations}-(iii), it was sufficient for us to prove that there is no profitable unilateral deviation from the pair $(\Phi^*,\Psi^*)$ in the class of controls in order to deduce that there is no such deviation in the class of randomised strategies. That allowed us to rely on singular control techniques to complete the argument.

We do not claim that the equilibrium we found is unique but it is an equilibrium for all parameter specifications ($\mu_0$, $\hat \mu$, $\sigma$, $r$) and all values of the initial states $y>x\ge 0$. The question as to whether an equilibrium exists in which Player 1 uses a strategy and Player 2 uses a control (both either pure or randomised) is open when $y>x$, especially if we seek an equilibrium valid for all parameter choices and initial states. The difficulty in this case arises from the fact that Player 2 (who has larger initial endowment) can always become monopolist by waiting until Player 1 defaults (for example by paying less than $y-x>0$ in dividends until $\gamma_X$). Therefore, it would seem that a mapping $\Phi(t,B_\cdot(\omega))$, based solely on the observation of the Brownian sample paths, should be insufficient for Player 2 to achieve a best response. We leave the question open for future research.
\end{remark}
\medskip

In the remainder of the paper we simplify our notation and adopt
\begin{equation}\label{eq:values}
v_1(x,y)=v_1(x,y;\Psi^*)\quad\text{and}\quad v_2(x,y)=v_2(x,y;\Phi^*),
\end{equation}
for the equilibrium payoffs when $0\le x < y$. 
\begin{remark}\label{MultipleNash}
Notice that by continuity of the mappings 
$(x,y)\mapsto (v_1(x,y),v_2(x,y))$, we can actually extend the formulae in \eqref{eq:values} to points on the diagonal $\{(x,y)\in[0,\infty)^2:x=y\}$, and the previous equilibrium analysis remains valid. For $x=y$ by exchanging the roles of the players, we obtain two different equilibria with asymmetric payoffs (i.e., $(\Phi^*,\Psi^*)$ and $(\Psi^*,\Phi^*)$ are two equilibria with payoffs $[v_1(x,x;\Psi^*),v_2(x,x;\Phi^*)]=[v_1(x,x),v_2(x,x)]$ and $[v_1(x,x;\Phi^*),v_2(x,x;\Psi^*)]=[v_2(x,x),v_1(x,x)]$, respectively). In these equilibria, despite the players' initial position being symmetric, one of the two players plays a predatory strategy with payoff $v_2(x,x)$ and has a positive probability to obtain the monopolist payoff $\hat v$, while the other player plays an optimal control for the standard dividend problem with drift $\mu_0$, obtains the smaller payoff $v_1(x,x)=v_0(x)$ and has zero probability to obtain the monopolist payoff.

These equilibria may appear unrealistic as none of the players would agree to be in the dominated position when starting from a symmetric position. Situations of this kind arise in classical ``war of attrition'' models (see \cite{HWW} for a complete study of the deterministic model in continuous-time), where typically there exists also a symmetric equilibrium (i.e., with players using the same strategy) in randomised strategies. Symmetric equilibria of this type are constructed in \cite{Steg} in a continuous-time stochastic model and an extension to a model with incomplete information is given in \cite{PK22}. A complete characterization of the equilibria in randomised strategies in the continuous-time stochastic framework with complete information for one dimensional diffusions is given in \cite{DGM}. 

In the next section we construct a symmetric equilibrium with randomised strategies for our game with symmetric initial endowment.
\end{remark}

\section{Nash equilibrium with symmetric initial endowment.}\label{sec:sym}

When the two players have the same initial endowment, i.e., $x=y$, the game is fully symmetric. As soon as one of the players pays an arbitrarily small amount of dividends the symmetry is broken and the game falls back into the situation analysed in the previous section. From a game-theoretic point of view there is a second mover advantage and it is not clear whether a symmetric equilibrium can be found only using pure strategies. We construct an equilibrium in randomised strategies, where the players use a randomised stopping time to decide the time of their first move. By symmetry, we only need to consider one function that describes the `intensity of stopping' (in equilibrium) for both players. This function will be specified later.

In preparation for the proof of the main result of this section we need to introduce some new objects.
For $\zeta\in D_0^+([0,\infty))$, set 
$T(\zeta)=\inf \{ t \geq 0 : \zeta_t >0\}$.
Then, $T(\zeta)$ is an $(\cF^{\mathbb W}_t)$-optional time because it is an entry time to an open set for a c\`adl\`ag process (recall that $(\cF^{\mathbb W}_t)_{t\ge 0}$ is the raw filtration of the canonical process $\mathbb{W}_t(\varphi,\zeta)=(\varphi(t),\zeta(t))$).

We also extend the definition of $\Phi^*$ to account for an activation at an arbitrary time $s\in[0,\infty)$:
\[ 
\Phi^*_s(t,\varphi):=1_{\{t \geq s\}}\sup_{s\le \rho\le t} \Big(x-a_0+\mu_0 \rho+\sigma \varphi(\rho)\Big)^+,\quad\text{for $\varphi\in C_0([0,\infty))$}.
\]
In order to construct our equilibrium in randomised strategies, we let $\ell: [0,\infty)\to[0,\infty)$ be a measurable function to be determined at equilibrium 
and we define, for $\varphi\in C_0([0,\infty))$
\[
\bar \Gamma^\ell_t(\varphi):=\int_0^t\e^{-\int_0^s\ell \big(x+\mu_0 \rho+\sigma\varphi(\rho)\big)\ud \rho}\ell\big(x+\mu_0 s+\sigma\varphi(s)\big)\ud s.
\]
Then, for $u\in[0,1]$ we also introduce
$\bar \gamma^\ell(\varphi,u):=\inf\{t\ge 0:\bar \Gamma^\ell_t(\varphi)\geq u\}$.
Notice that $\bar\gamma^\ell(\varphi,u)$ is an $(\cF^{\mathbb W}_t)$-stopping time as an entry time of a continuous process to a closed set.
In particular, for each $\omega\in\Omega$ we denote
\begin{equation}\label{eq:Gamma}
\Gamma^\ell_t(\omega):=\bar \Gamma^\ell_t(B(\omega))=\int_0^t\e^{-\int_0^s\ell (X^0_\rho(\omega))\ud \rho}\ell\big(X^0_s(\omega)\big)\ud s=1-\e^{-\int_0^t\ell (X^0_s(\omega))\ud s},
\end{equation}
where $X^0_t=x+\mu_0t+\sigma B_t$. 
Recall that  we work on the enlarged probability space 
\[ 
(\bar \Omega, \bar \cF, \bar \P):=(\Omega\times [0,1]\times [0,1], \cF\otimes \cB([0,1])\otimes \cB([0,1]), \P\otimes \lambda \otimes \lambda),
\]
with canonical element $\bar \omega=(\omega,u_1,u_2)$, where $\cB([0,1])$ denotes the Borel $\sigma$-field and $\lambda$ the Lebesgue measure, and that the two random variables $U_i(\bar \omega)=u_i$ for $i=1,2$  defined on $(\bar \Omega,\bar \cF, \bar \P)$, are uniformly distributed on $[0,1]$, mutually independent and independent of the Brownian motion $B=B(\omega)$. 
We define the randomised stopping times for the raw Brownian filtration
\begin{equation}\label{eq:rst}
\gamma^\ell_i:=\bar \gamma^\ell(B,U_i)=\inf\{t\ge 0: \Gamma^\ell_t \geq U_i\}, \quad\text{for $i=1,2$}.
\end{equation}

In order to find an equilibrium the two players need to find a function $\ell^*$ that generates a pair of optimal randomised stopping times $(\gamma^*_1,\gamma^*_2)$. At equilibrium, on the event $\{\gamma^*_1<\gamma^*_2\}$ Player 1 makes the first move and gives her opponent the second mover advantage. After the first move, the game can be analysed with the arguments from Section \ref{sec:NEasym}. Indeed, we will show that Player 1 is going to adopt the strategy $\Phi^*$ as in Theorem \ref{thm:NE1}, while Player 2 is going to use the strategy $\Psi^*$. On the event $\{\gamma^*_1>\gamma^*_2\}$ the first move is made by Player 2 and the situation is analogous but symmetric.

Let 
$\ell^*(x):=[rv_0(x)-\mu_0]^+/\left(v_2(a_0,x)-v_0(x)\right)$, for $x\ge 0$,
with $v_0$ from \eqref{eq:dnot}. Notice that $\ell^*(x)>0\iff x\in(a_0,\infty)$ by \eqref{eq:Lw} and since $v_2(a_0,x)>v_0(x)$ for $x\in [a_0,\infty)$ by Remark \ref{rem:ineq_v2_v0}. Given $(u,\varphi,\zeta)\in [0,1]\times C_0([0,\infty))\times D_0^+([0,\infty))$, set $\bar \gamma^*(\varphi,u):=\bar \gamma^{\ell^*}(\varphi,u)$ and simplify the notation to $\bar \gamma^*=\bar \gamma^*(\varphi,u)$, $T=T(\zeta)$.
An important role will be played by the randomised strategy $\Xi^*$ is defined as
\begin{equation}\label{eq:xistar}
\begin{aligned}
\Xi^*(u, t,\varphi,\zeta):=1_{\{t \ge \bar \gamma^*\wedge T\}} \Big[ 1_{\{T \geq \bar\gamma^*\}}\Phi_{\bar \gamma^*}^*(t,\varphi)+1_{\{T < \bar\gamma^*\}}\Psi^*(t,\varphi,\zeta)\Big],
\end{aligned}
\end{equation}
where $\Psi^*$ is defined as in \eqref{eq:maps} with the initial condition $(x,x)$.
The fact that $\Xi^*$ satisfies the non-anticipative property in Definition \ref{def:random_strategies} is not straightforward, because $T$ is only an $(\cF^{\mathbb W}_t)$-optional time of the canonical filtration. Checking non-anticipativity of the map will be part of the proof of Theorem \ref{thm:NE2}.

The next theorem provides a symmetric equilibrium in randomised strategies in the symmetric set-up and it is the main result of this section. 
\begin{theorem}[NE with symmetric endowment]\label{thm:NE2}
Set $\gamma^*_i:=\gamma^{\ell^*}_i$, $i=1,2$, as in \eqref{eq:rst}.
The pair $(\Xi^*_1,\Xi^*_2)$, with $\Xi^*_i=\Xi^*$, $i=1,2$, is a Nash equilibrium in randomised strategies 
and the equilibrium payoffs for the two players read
$\cJ^1_{x,x}( \Xi^*_1, \Xi^*_2)=\cJ^2_{x,x}( \Xi^*_2, \Xi^*_1)=v_0(x)$, $x\in[0,\infty)$,
with $v_0$ as in \eqref{eq:dnot}.
\end{theorem}
A few remarks are in order before we proceed with the proof of the theorem. Due to the symmetry of the set-up, all the considerations that we make for one player's strategy also hold for the other player's strategy.
\medskip

\begin{remark}[An intuitive interpretation]
Once one of the players ---say Player $1$--- activates her control (i.e., she pays dividends), we reach a position $X^L<Y^D$. According to \eqref{eq:xistar}, at that point both players play 
the asymmetric equilibrium obtained in Theorem \ref{thm:NE1}, with Player $1$ in a dominated position. Thus, a player's activation of her control can be thought of as ``conceding'', in the sense that the first mover accepts to be in a dominated position. 
In principle, neither player wants to be the first to activate her control. However, waiting is costly for both players because of discounting and the risk of default. Then, there is a trade-off ---say for Player $2$--- between the potential gains at the time Player $1$ activates her control (i.e., $v_2(X,Y)-v_0(Y)>0$) and the effect of discounting and default risk. The strategy $\Xi^*$ for Player 2 can be described as follows: Player $2$ will wait for a random amount of time $\gamma_2^*$ and one of two mutually exclusive outcomes is possible: 
\begin{itemize}
\item[(a)] Player 1 activates her control strictly before $\gamma^*_2$ and concedes (that corresponds to the event $\{T<\bar\gamma^*\}$ in \eqref{eq:xistar}); Then the players start playing the continuation equilibrium from Theorem \ref{thm:NE1} in which Player $1$ is in a dominated position; 
\item[(b)] Player 1 does not activate her control before $\gamma_2^*$ (that corresponds to the event $\{T\ge\bar\gamma^*\}$ in \eqref{eq:xistar}) and Player 2 concedes; Then at $\gamma_2^*$ players play the continuation equilibrium obtained from Theorem \ref{thm:NE1} by exchanging the roles of the players, in which Player $2$ is in a dominated position. 
\end{itemize}

The function $\ell^*$ is constructed in order to make each player {\em indifferent} between conceding and waiting at any moment of time when $X^0_t> a_0$ (the so-called {\em indifference principle}). That guarantees that waiting a random time $\gamma^*$ before conceding is a best reply. As this is true for both players, it is then a Nash equilibrium. In particular, we notice that $\gamma^*_1$ is the first jump time of a Poisson process with a stochastic intensity $\ell^*(X^0_t)$, which is positive if and only if $X^0_t>a_0$ (analogously for $\gamma^*_2$).
Note that our construction shares some important features with symmetric equilibria in stopping games with a second mover advantage constructed in \cite{Steg}.
\end{remark}

\begin{remark}[Observable quantities for the two players]
The random time $\gamma^*_2$ depends on the realization of a private randomisation device $U_2$ used by Player $2$. As such, it is not directly observable by Player 1. Then, the definition of Player 1's strategy cannot depend {\em explicitly} on $\gamma^*_2$ and/or on $U_2$. Instead, Player 1's strategy is allowed to depend on $\gamma^*_1$ and on the trajectories of $B$ and of Player $2$'s control $D$, which are {\em observable quantities} for Player 1. 

The realized controls $L$ and $D$ induced by a pair of randomised strategies $(\Xi_1,\Xi_2)$ depend in general on $(U_1,U_2)$ (see Definition \ref{def:theta_r}). Indeed, the control chosen by Player $2$ may depend on $U_2$, as her strategy depends on $U_2$, whereas Player 1's strategy depends {\em indirectly} on $U_2$ through $D$. The same is true by swapping the roles of the two players. 
\end{remark}
\medskip

\begin{proof}[Proof of Theorem \ref{thm:NE2}.]
The proof is divided into five main steps. In the first step, we show that $\Xi^*$ satisfies the conditions of Definition \ref{def:random_strategies}. In the second step we show that $(\Xi^*_1,\Xi^*_2)\in\Theta_R(x,x)$ and, in particular, that there exists a unique pair of randomised controls $(\bar L^*, \bar D^*)$ defined on $\bar \Omega$ such that 
\begin{equation}\label{eq:wellposed}
\bar L^*_t =\Xi^*(U_1,t,B,\bar D^*)\quad\text{and}\quad\bar D^*_t =\Xi^*(U_2,t,B,\bar L^*).
\end{equation}
In the third step we calculate the players' payoffs associated to the pair $(\bar L^*,\bar D^*)$. Then we show that the  pair $(\Xi^*_1,\Xi^*_2)$ is a Nash equilibrium in two subsequent steps. 
\medskip

\noindent{\bf Step 1}. Notice first that $(u,t,\varphi,\zeta)\mapsto\Xi^*(u,t,\varphi,\zeta)$ is jointly measurable. For every $u\in [0,1]$, the trajectory 
\[ 
t \rightarrow \Xi^*(u,t,\varphi,\zeta)=1_{\{t \ge \bar \gamma^*\wedge T\}} \Big[ 1_{\{T \geq \bar\gamma^*\}}\Phi_{\bar \gamma^*}^*(t,\varphi)+1_{\{T < \bar\gamma^*\}}\Psi^*(t,\varphi,\zeta)\Big],
\]
is non-decreasing, right-continuous and it satisfies the admissibility condition in Definition \ref{def:random_strategies}(iii) thanks to analogous properties of $\Phi^*$ and $\Psi^*$.
It only remains to check that $t \rightarrow \Xi^*(u,t,\varphi,\zeta)$ is non-anticipative.
Using that $b(x)\geq \alpha>0$ for $x\in [0,\infty)$, we deduce that for all $(t,\varphi,\zeta)\in [0,\infty)\times C_0([0,\infty))\times D_0^+([0,\infty))$
\begin{equation*}
\Psi^*(t,\varphi,\zeta)=\sup_{0\le s\le t}\Big(\zeta(s)-b\big(x+\mu_0 s +\sigma \varphi(s)-\zeta(s)\big)\Big)^+=1_{\{t \geq \tau_\alpha(\zeta)\}}\Psi^*(t,\varphi,\zeta),
\end{equation*} 
where we set 
$\tau_\alpha(\zeta)=\inf \{ t \geq 0 : \zeta_t \geq \alpha \}$ for $\zeta\in D_0^+([0,\infty))$.
It follows that, denoting $\tau_\alpha=\tau_\alpha(\zeta)$ 
\begin{equation*}
\begin{aligned}
\Xi^* (u,t,\varphi,\zeta) 
&=1_{\{t \ge \bar \gamma^*\wedge T\}} \Big[ 1_{\{T \geq \bar\gamma^*\}}\Phi_{\bar \gamma^*}^*(t,\varphi)+1_{\{T < \bar\gamma^*\}}1_{\{t \geq \tau_{\alpha}\}}\Psi^*(t,\varphi,\zeta)\Big] \\
&=1_{\{t \ge \bar \gamma^*\wedge T\}} \Big[ 1_{\{\zeta_{\bar \gamma^*-}=0\}}\Phi_{\bar \gamma^*}^*(t,\varphi)+1_{ \{\zeta_{\bar \gamma^*-}>0\}}1_{\{t \geq \tau_{\alpha}\}}\Psi^*(t,\varphi,\zeta)\Big] \\
&=1_{\{t \ge \bar \gamma^*\wedge \tau_\alpha\}} \Big[ 1_{\{\zeta_{\bar \gamma^*-}=0\}}\Phi_{\bar \gamma^*}^*(t,\varphi)+ 1_{ \{\zeta_{\bar \gamma^*-}>0\}}\Psi^*(t,\varphi,\zeta)\Big],
\end{aligned}
\end{equation*}
where we used that $\tau_\alpha \geq T$, and that $\{\zeta_{\bar \gamma^*-}=0\}=\{ T \geq \bar \gamma^*\}$. 
The final expression guarantees the non-anticipativity property because $\bar \gamma^*(u,\cdot)$ and $\tau_\alpha$ are $(\cF^{\mathbb W}_t)$-stopping times. 
\medskip

\noindent{\bf Step 2}. Here we show that there is a unique solution of \eqref{eq:wellposed}. 
Fix a treble $(\omega,u_1,u_2)\in \bar \Omega=\Omega\times[0,1]^2$ so that $(U_1,U_2)=(u_1,u_2)$ and the trajectory of the Brownian motion $t\mapsto B_t(\omega)$ is fixed (so is the trajectory $t\mapsto X^0_t(\omega)$). Then, the random times $\gamma^*_i$, $i=1,2$, from \eqref{eq:rst} are uniquely determined. Here we should use the notation $\gamma^*_i(u_i,\omega)$, for $i=1,2$, $\bar L^*(\omega,u_1,u_2)$ and $\bar D^*(\omega,u_1,u_2)$ to stress the dependence of these quantities on $(\omega,u_1,u_2)$. This is rather cumbersome, so we drop the explicit dependence on the treble $(\omega,u_1,u_2)$ but we emphasise that the rest of the construction in this step is performed {\em pathwise}. 

First we show that \eqref{eq:wellposed} admits at most one solution. Let us assume that $(\bar L^*,\bar D^*)$ is a solution pair for \eqref{eq:wellposed}. Then, we show that $t<\gamma^*_1\wedge \gamma^*_2 \implies \bar L^*_t=\bar D^*_t=0$ (actually $\bar L^*_s=\bar D^*_s=0$ for $s\in[0,t]$, by monotonicity). Arguing by contradiction, assume $\bar L^*_t>0$ and $t<\gamma^*_1\wedge \gamma^*_2$. Then from the definition of $\bar L^*$ and \eqref{eq:wellposed}, it must be $T(\bar D^*) \leq t < \gamma^*_1$ and $\bar L^*_t=\Psi^*(t,B,\bar D^*)$. Moreover, using the definition of $\Psi^*$ and recalling that $b(x)\geq \alpha$ for $x\in[0,\infty)$, yields $\bar L^*_t=\Psi^*(t,B,\bar D^*)>0 \implies \bar D^*_t > \alpha$. Then,
\begin{equation}\label{eq:Xa}
X^{\bar L^*}_t=X^0_t-\Psi^*(t,B,\bar D^*)\ge X^0_t-(\bar D^*_t-\alpha)^+=X^0_t-\bar D^*_t+\alpha=Y^{\bar D^*}_t+\alpha. 
\end{equation}
However, since $t<\gamma^*_1\wedge\gamma^*_2$ and $\bar L^*_t>0$ imply $\bar D^*_t > \alpha$, then we should also have $t<\gamma^*_1\wedge\gamma^*_2$ and $\bar D^*_t>0$. Therefore, we can repeat the argument above swapping the roles of the two players, and obtain
$Y^{\bar D^*}_t\ge X^{\bar L^*}_t+\alpha$. 
Combining the latter and \eqref{eq:Xa} leads to a contradiction and it must be $t<\gamma^*_1\wedge \gamma^*_2 \implies \bar L^*_t=\bar D^*_t=0$, as claimed. It follows that $T(\bar D^*) \geq \gamma^*_1\wedge \gamma^*_2$ and $T(\bar L^*) \geq \gamma^*_1\wedge \gamma^*_2$. 

The pair $(\gamma^*_1,\gamma^*_2)$ is exogenously determined by the realisation of the pair $(U_1,U_2)$ and the trajectory of $B$. Moreover, the definition of $\ell^*$ and
\eqref{eq:wnondecr} imply $\ell^*(x)>0 \iff x\in(a_0,\infty)$. Therefore 
\begin{equation}\label{eq:meas}
\begin{aligned}
\text{$t\mapsto \Gamma^{\ell^*}_t(\omega)$ defines a measure on $[0,\infty)$ supported by the set
$\{t\ge 0:X^0_t(\omega)\in [a_0,\infty)\}$.}
\end{aligned}
\end{equation}
That implies that $\gamma^*_1$ and $\gamma^*_2$ can only occur during excursions of the process $X^0$ into the half-line $[a_0,\infty)$, and thus that $X^0_{\gamma^*_i}\in [a_0,\infty)$ for all $(u_i,\omega)$, $i=1,2$. First, we notice that $\gamma^*_1\leq \gamma^*_2$ implies $T(\bar D^*) \geq \gamma^*_1$ and therefore $\bar L^*_t=\Phi^*_{\gamma^*_1}(t,B)$. Second, if $\gamma^*_1<\gamma^*_2$, it must be $u_2>u_1$ and  $\sup_{t\in[\gamma^*_1,\gamma^*_2]}X^0_t >a_0$. This implies $T(\bar L^*)=T(\Phi^*_{\gamma^*_1}(\cdot,B))<\gamma^*_2$, and thus $\bar D^*_t=\Psi^*(t,B,\bar L^*)1_{\{t\ge T(\bar L^*)\}}$.
Similarly, $\gamma^*_2\leq \gamma^*_1\implies \bar D^*_t=\Phi^*_{\gamma^*_2}(t,B)$ and  $\gamma^*_2 < \gamma^*_1$ implies that $T(\bar D^*)=T(\Phi^*_{\gamma^*_2}(\cdot,B))<\gamma^*_1$ and that  $\bar L^*_t=\Psi^*(t,B,\bar D^*)1_{\{t\ge T(\bar D^*)\}}$.

Then, if a solution of \eqref{eq:wellposed} exists, it is uniquely determined by the properties above and 
\begin{equation}\label{eq:bLbD_pathwise}
\begin{aligned}
(\bar L^*_t,\bar D^*_t)=\left\{
\begin{array}{ll}
\Big(\Phi^*_{\gamma^*_1}(t,B),\Psi^*\big(t,B,\Phi^*_{\gamma^*_1}(\cdot,B)\big)1_{\{t \geq T(\Phi^*_{\gamma^*_1}(\cdot,B))\}} \Big)& \text{on $\{\gamma^*_1<\gamma^*_2\}$},\\[+4pt]
\Big(\Psi^*\big(t,B,\Phi^*_{\gamma^*_2}(\cdot,B)\big)1_{\{t \geq T(\Phi^*_{\gamma^*_2}(\cdot,B))\}},\Phi^*_{\gamma^*_2}(t,B) \Big)& \text{on $\{\gamma^*_2<\gamma^*_1\}$}, \\
\Big(\Phi^*_{\gamma^*_1}(t,B),\Phi^*_{\gamma^*_2}(t,B) \Big)& \text{on $\{\gamma^*_2=\gamma^*_1\}$}.
\end{array}
\right.
\end{aligned}
\end{equation}
We check easily that the pair $(\bar L^*,\bar D^*)$ as defined in \eqref{eq:bLbD_pathwise} is a solution to \eqref{eq:wellposed} and that its restriction on $[0,t]$ is the unique solution of \eqref{eq:wellposed} on $[0,t]$, which proves that $(\Xi^*_1,\Xi^*_2) \in \Theta_R(x,x)$.

The characterisation of the pair $(\Xi^*_1,\Xi^*_2)$ is further simplified if we remove  suitable $\bar \P$-null sets.
Indeed, $\bar \P$-almost surely $\Phi^*_{\gamma^*_i}(t,B)>0$ for all $t \geq \gamma^*_i$, $i=1,2$. Then, $\gamma^*_1<\gamma^*_2$ implies $T(\bar L^*)=\gamma^*_1$, $\bar\P$-a.s. Vice-versa, $\gamma^*_1>\gamma^*_2$ implies $T(\bar D^*)=\gamma^*_2$, $\bar \P$-a.s.
Then, the solution of \eqref{eq:wellposed} satisfies
\begin{equation}\label{eq:bLbD}
\begin{aligned}
(\bar L^*_t,\bar D^*_t)=\left\{
\begin{array}{ll}
1_{\{t \geq \gamma^*_1\}}\Big(\Phi^*_{\gamma^*_1}(t,B),\Psi^*\big(t,B,\Phi^*_{\gamma^*_1}(\cdot,B)\big) \Big)& \text{on $\{\gamma^*_1<\gamma^*_2\}$},\\[+4pt]
1_{\{t \geq \gamma^*_2\}}\Big(\Psi^*\big(t,B,\Phi^*_{\gamma^*_2}(\cdot,B)\big),\Phi^*_{\gamma^*_2}(t,B) \Big)& \text{on $\{\gamma^*_2<\gamma^*_1\}$}.
\end{array}
\right.
\end{aligned}
\end{equation}
Since $\gamma^*_1(u_1,\omega)\neq\gamma^*_2(u_2,\omega)\iff u_1\neq u_2$ and $(\lambda\otimes\lambda)(U_1=U_2)=0$, then \eqref{eq:bLbD} characterises the pair $(\bar L^*,\bar D^*)$ up to a $\bar \P$-negligible set.
\medskip

\noindent{\bf Step 3}.
Here we calculate the players' payoffs under the strategy pair $(\Xi^*_1,\Xi^*_2)$. The induced pair of randomised controls is $(\bar L^*,\bar D^*)$ constructed in step 2. We denote $X^*=X^{\bar L^*}$, $Y^*=Y^{\bar D^*}$ with the associated default times $\gamma_{X^*}$ and $\gamma_{Y^*}$. We also denote $\gamma_0=\inf\{t\ge 0 : X^0_t=0\}$. In order to keep track of randomisation, for any realisation $(U_1,U_2)=(u_1,u_2)$ we use the notation 
$(X^*,Y^*)=(X^{*;u_1u_2},Y^{*;u_1u_2})$ and $(\bar D^*,\bar L^*)=(\bar D^{*;u_1,u_2},\bar L^{*;u_1,u_2})$.
These maps are measurable in $(u_1,u_2)$ by construction.   

Player 2's payoff reads 	
\begin{equation}\label{eq:Jrand}
\cJ^2_{x,x}(\bar D^*,\bar L^*)=\int_0^1\int_0^1\cJ^2_{x,x}(\bar D^{*;u_1,u_2},\bar L^{*;u_1,u_2})\ud u_1\ud u_2.
\end{equation}
For a fixed pair $(u_1,u_2)$ we have
\begin{equation}\label{eq:Juu}
\begin{aligned}
&\cJ^2_{x,x}(\bar D^{*;u_1,u_2},\bar L^{*;u_1,u_2})\\
&=\E_{x,x}\Big[\int_{[0,\gamma_{X^{*;u_1,u_2}}\wedge\gamma_{Y^{*;u_1,u_2}}]}\e^{-r t}\ud \bar D^{*;u_1,u_2}_t+1_{\{\gamma_{X^{*;u_1,u_2}}<\gamma_{Y^{*;u_1,u_2}}\}}\e^{-r\gamma_{X^{*;u_1,u_2}}}\hat v(Y^{*;u_1,u_2}_{\gamma_{X^{*;u_1,u_2}}})\Big].
\end{aligned}
\end{equation}
The expression under expectation is zero on the event $\{\gamma^*_1(u_1)\wedge\gamma^*_2(u_2)\ge\gamma_0\}$, because default for both firms occurs before they actually start paying any dividends. 

On the complementary event, we consider separately the cases $\gamma^*_1(u_1)<\gamma^*_2(u_2)$ and $\gamma^*_1(u_1)>\gamma^*_2(u_2)$. We recall the notation from Remark \ref{rem:phipsi*}, i.e., $\Phi^*(x,t,\varphi)$ and $\Psi^*(x,y,t,\varphi,\zeta)$, in order to keep track of the position of the process $X^0$ at the time when the strategies of the two players are activated. We also introduce a shift for the trajectories in the canonical space, defined as $\theta_t (\varphi(\cdot))=\varphi(t+\cdot)-\varphi(t)$ for $\varphi\in C_0([0,\infty))$.  Finally, from the definition of $\Phi^*_s$ and $\Psi^*$ it is not hard to see that for $t\ge\gamma^*_i(u_i)$, $i=1,2$,
\begin{equation*}
\begin{aligned}
\Phi^*_{\gamma^*_i(u_i)}(t,B)&=\Phi^*\big(X^0_{\gamma^*_i(u_i)},t-\gamma^*_i(u_i), \theta_{\gamma^*_i(u_i)}(B_{\cdot})\big)=:\tilde\Phi^*\big(t,\theta_{\gamma^*_i(u_i)}(B_{\cdot})\big)\\
\Psi^*\big(t,B,\Phi^*_{\gamma^*_i(u_i)}\big)&=\Psi^*\Big(X^0_{\gamma^*_i(u_i)},X^0_{\gamma^*_i(u_i)},t-\gamma^*_1(u_i),\theta_{\gamma^*_i(u_i)}(B_\cdot),\tilde\Phi^*\big(t,\theta_{\gamma^*_i(u_i)}(B_{\cdot})\big)\Big)\\
&\eqqcolon\tilde\Psi^*\Big(t,\theta_{\gamma^*_i(u_i)}(B_{\cdot}),\tilde\Phi^*\Big).
\end{aligned}
\end{equation*}
Notice that, for example, on the event $\{\gamma^*_1(u_1)<\gamma^*_2(u_2)\}$ we have $\bar L^*_t=\tilde\Phi^*(t,\theta_{\gamma^*_1(u_1)}(B_\cdot))$ for $t\ge \gamma^*_1(u_1)$, which is the optimal control in the classical dividend problem starting at $X^0_{\gamma^*_1(u_1)}$. Then $\bar D^*_t=\tilde \Psi^*(t,\theta_{\gamma^*_1(u_1)}(B_\cdot),\tilde\Phi^*)$ for $t\ge \gamma^*_1(u_1)$ is Player 2's response in the game starting at $X^0_{\gamma^*_1(u_1)}$, when Player 1's concedes. From this dicusssion it follows that 
$\{\gamma^*_1(u_1)<\gamma^*_2(u_2)\}\subset\{\gamma_{X^{*;u_1,u_2}}<\gamma_{Y^{*;u_1,u_2}}\}$.
A symmetric situation occurs on the event $\{\gamma^*_1(u_1)>\gamma^*_2(u_2)\}$.

Continuing from \eqref{eq:Juu}, on the event $\{\gamma^*_1(u_1)<\gamma^*_2(u_2)\}$ we have
\begin{equation*}
\begin{aligned}
&\E_{x,x}\Big[1_{\{\gamma^*_1(u_1)<\gamma_0\}}1_{\{\gamma^*_1(u_1)<\gamma^*_2(u_2)\}}\Big(\int_{[\gamma^*_1(u_1),\gamma_{X^{*;u_1,u_2}}]}\e^{-r t}\ud \bar D^{*;u_1,u_2}_t+\e^{-r\gamma_{X^{*;u_1,u_2}}}\hat v(Y^{*;u_1,u_2}_{\gamma_{X^{*;u_1,u_2}}})\Big)\Big]\\
&=\E_{x,x}\Big[1_{\{\gamma^*_1(u_1)<\gamma_0\}}1_{\{\gamma^*_1(u_1)<\gamma^*_2(u_2)\}}\E_{x,x}\Big[\int_{[\gamma^*_1(u_1),\gamma_{X^{*;u_1,u_2}}]}\e^{-r t}\ud \bar D^{*;u_1,u_2}_t\\
&\qquad\qquad+\e^{-r\gamma_{X^{*;u_1,u_2}}}\hat v(Y^{*;u_1,u_2}_{\gamma_{X^{*;u_1,u_2}}})\Big|\cF_{\gamma^*_1(u_1)}\Big]\Big].
\end{aligned}
\end{equation*}
Using the strong Markov property, on the event ${\{\gamma^*_1(u_1)<\gamma_0\}}\cap{\{\gamma^*_1(u_1)<\gamma^*_2(u_2)\}}$ we can write 
\begin{equation*}
\E_{x,x}\Big[\int_{[\gamma^*_1(u_1),\gamma_{X^{*;u_1,u_2}}]}\e^{-r t}\ud \bar D^{*;u_1,u_2}_t\!+\!\e^{-r\gamma_{X^{*;u_1,u_2}}}\hat v(Y^{*;u_1,u_2}_{\gamma_{X^{*;u_1,u_2}}})\Big|\cF_{\gamma^*_1(u_1)}\Big]=\cJ^2_{X^0_{\gamma^*_1(u_1)},X^0_{\gamma^*_1(u_1)}}(\tilde \Psi^*,\tilde \Phi^*),
\end{equation*}
where the final expression is Player 2's expected payoff when the game starts from $(X^0_{\gamma^*_1(u_1)},X^0_{\gamma^*_1(u_1)})$, Player 1 uses $\tilde\Phi^*$ and Player 2 uses $\tilde \Psi^*$.  
From the analysis in Section \ref{sec:NEasym} we know that 
\[
\cJ^2_{X^0_{\gamma^*_1(u_1)},X^0_{\gamma^*_1(u_1)}}(\tilde \Psi^*,\tilde \Phi^*)=v_2\big(X^0_{\gamma^*_1(u_1)},X^0_{\gamma^*_1(u_1)}\big),
\]
with $v_2$ as in \eqref{eq:values}. Since $X^0_{\gamma^*_1(u_1)}\in(a_0,\infty)$ by \eqref{eq:meas} and $v_2(x,y)=v_2(a_0,y)$ for $x\ge a_0$ by construction, then
\[
\cJ^2_{X^0_{\gamma^*_1(u_1)},X^0_{\gamma^*_1(u_1)}}(\tilde \Psi^*,\tilde \Phi^*)=v_2\big(a_0,X^0_{\gamma^*_1(u_1)}\big).
\]
That yields
\begin{equation}\label{eq:po1}
\begin{aligned}
&\E_{x,x}\Big[1_{\{\gamma^*_1(u_1)<\gamma_0\}}1_{\{\gamma^*_1(u_1)<\gamma^*_2(u_2)\}}\Big(\int_{[\gamma^*_1(u_1),\gamma_{X^{*;u_1,u_2}}]}\!\!\!\e^{-r t}\ud \bar D^{*;u_1,u_2}_t\!+\!\e^{-r\gamma_{X^{*;u_1,u_2}}}\hat v(Y^{*;u_1,u_2}_{\gamma_{X^{*;u_1,u_2}}})\Big)\Big]\\
&=\E_{x,x}\Big[1_{\{\gamma^*_1(u_1)<\gamma_0\}}1_{\{\gamma^*_1(u_1)<\gamma^*_2(u_2)\}}\e^{-r\gamma^*_1(u_1)}v_2\big(a_0,X^0_{\gamma^*_1(u_1)}\big)\Big].
\end{aligned}
\end{equation}

On the event $\{\gamma^*_2(u_2)<\gamma^*_1(u_1)\}$ the roles of the two players are reversed, in the sense that Player 2 adopts the strategy $\Phi^*_{\gamma^*_2(u_2)}$ and Player 1 uses $\Psi^*$. Continuing from \eqref{eq:Juu}, arguments analogous to the ones that yield \eqref{eq:po1} allow us to deduce
\begin{equation}\label{eq:po2}
\begin{aligned}
&\E_{x,x}\Big[1_{\{\gamma^*_2(u_2)<\gamma_0\}}1_{\{\gamma^*_1(u_1)>\gamma^*_2(u_2)\}}\int_{[\gamma^*_2(u_2),\gamma_{Y^{*;u_1,u_2}}]}\e^{-r t}\ud \bar D^{*;u_1,u_2}_t\Big]\\
&=\E_{x,x}\Big[1_{\{\gamma^*_2(u_2)<\gamma_0\}}1_{\{\gamma^*_1(u_1)>\gamma^*_2(u_2)\}}\e^{-r\gamma^*_2(u_2)}v_0\big(X^0_{\gamma^*_2(u_2)}\big)\Big],
\end{aligned}
\end{equation}
where we also use that $\gamma_{X^{*;u_1,u_2}}>\gamma_{Y^{*;u_1,u_2}}$ on the event $\{\gamma^*_1(u_1)>\gamma^*_2(u_2)\}$.

Combining \eqref{eq:po1} and \eqref{eq:po2} with \eqref{eq:Juu} and \eqref{eq:Jrand} yields
\begin{equation} \label{eq:equilibrium_payoff}
 \cJ^2_{x,x}(\bar D^*,\bar L^*) =\bar \E_{x,x}\big[1_{\{\gamma^*_2<\gamma^*_1 \wedge \gamma_0\}}\e^{-r \gamma^*_2}v_0(X^0_{\gamma^*_2})+1_{\{\gamma^*_1<\gamma^*_2 \wedge \gamma_0\}}\e^{-r \gamma^*_1}v_2(a_0, X^0_{\gamma^*_1})  \big]. 
\end{equation} 
By the same arguments we obtain an analogous expression for $\cJ^1_{x,x}(\bar L^*,\bar D^*)$. Therefore the two players' payoffs are well-defined under the strategy pair $(\Xi^*_1,\Xi^*_2)$.
\medskip 

\noindent{\bf Step 4}. In this step and in the next one we show that $\Xi^*_2$ is Player 2's best response to Player 1's playing $\Xi^*_1$ (recall that $\Xi^*_1=\Xi^*$). In particular, in this step we are going to show that 
\begin{equation}\label{eq:Vineq} 
\cJ^2_{x,x}(\Psi,\Xi^*) \leq V(x),\quad\text{for any } \Psi\in\Sigma(x),
\end{equation}
where 
\begin{equation}\label{eq:V}
V(x)=\sup_{\tau}\E_x\Big[\e^{-r\tau}v_0(X^0_\tau)(1-\Gamma^*_\tau)1_{\{\tau<\gamma_0\}}+\int_0^{\tau\wedge\gamma_0}\e^{-rt}v_2(a_0,X^0_t)\ud \Gamma^*_t\Big],
\end{equation}
with $\Gamma^*=\Gamma^{\ell^*}$ and the supremum is taken over $(\cF_{t+})$-stopping times (recall that $(\cF_t)$ is the raw Brownian filtration which is not right-continuous).
By independence of $U_1$ from $\cF_\infty$ the expected payoff in $V(x)$ can be rewritten as
\begin{equation}\label{eq:VG}
\begin{aligned}
&\E_{x,x}\Big[\e^{-r\tau}v_0(X^0_\tau)(1-\Gamma^*_\tau)1_{\{\tau<\gamma_0\}}+\int_0^{\tau\wedge\gamma_0}\e^{-rt}v_2(a_0,X^0_t)\ud \Gamma^*_t\Big]\\
&=\bar \E_{x,x}\Big[\e^{-r\tau}v_0(X^0_\tau)1_{\{\tau<\gamma^*_1\}}1_{\{\tau<\gamma_0\}}+1_{\{\gamma^*_1\le \tau\wedge\gamma_0\}}\e^{-r\gamma^*_1}v_2(a_0,X^0_{\gamma^*_1})\Big],
\end{aligned}
\end{equation}
which coincides with the right-hand side of \eqref{eq:equilibrium_payoff} with $\tau$ instead of $\gamma^*_2$. 
It is a well-known fact that the value function in \eqref{eq:V} does not change if we allow $\tau$ to be chosen from the class of randomised stopping times\footnote{This can be seen as follows: recall that a randomised stopping time $\gamma$ is a measurable function $[0,1]\ni u\mapsto\gamma(u)$ with $\gamma(u)$ a pure stopping time for each $u\in[0,1]$; for a generic (measurable) payoff process $(P_t)_{t\ge 0}$ and any randomised stopping time $\gamma$ we have $\E[P_\gamma]=\int_0^1\E[P_{\gamma(u)}]\ud u\le \sup_{\tau}\E[P_\tau]$, where the supremum is taken over pure stopping times; then $\sup_\gamma\E[P_\gamma]\le \sup_\tau\E[P_\tau]$ and since the reverse inequality is trival, we conclude.}. 
Then, a priori it must be $\cJ^2_{x,x}(\bar D^*,\bar L^*)\le V(x)$.
In order to prove \eqref{eq:Vineq}, we will work  with the filtration $\cG_t$ generated by $\cF_{t+}$ and the random variable $U_1$, so that $\gamma^*_1$ is a $(\cG_t)$-stopping time. It is well-known and it is not hard to prove that, thanks to independence of $B$ and $U_1$, $B$ is also a $(\cG_t)$-Brownian motion. 

We now argue similarly to \eqref{eq:pair} and consider an arbitrary pair $(\Xi^*_1,\Psi) \in \Theta_R(x,x)$ with $\Psi\in\Sigma(x)$. Then, there exists a pair $(\bar L^*,D)\in\cD_R(x)\times\cD_R(y)$ which does not depend on $U_2$ and such that for every $u_1 \in [0,1]$, every $t\geq 0$, every $\omega$,  the restriction on the time interval $[0,t]$ of  $(\bar L^*,D)$ is the  unique  solution of 
\begin{equation} \label{eq:defDdeviation} 
\bar L^*_s(u_1,\omega)=\Xi^*(u_1,s,B_\cdot(\omega),D_\cdot(u_1,\omega)) \,\text{ and }\, D_s(u_1,\omega)=\Psi(s,B_\cdot(\omega),\bar L^*_\cdot(u_1,\omega)),\, \forall s\in [0,t].
\end{equation}
Note that $D$ depends on $U_1$ through $\bar L^*$. The controlled processes associated to $\bar L^*$ and $D$ are denoted $X^*=X^{\bar L^*}$ and $Y^D$ respectively and the corresponding default times $\gamma_{X^*}$ and $\gamma_Y$.

The dependence of $D$ on $U_1$ is actually quite specific due to the definition of $\Xi^*$. Intuitively, if Player $2$ starts paying dividend strictly before Player 1, then the stopping time $T(D)=\inf\{t\ge 0: D_t>0\}$ at which that happens should be independent of $U_1$. In order to make this statement rigorous, let us define $\tilde D_t\coloneqq \Psi(t,B,0)$, i.e., the (non-randomised) control induced by the pure strategy $\Psi$ assuming that Player $1$ pays no dividend. 
Now we are going to show that 
\begin{equation}\label{eq:sigma_tilde_D}
 \{T(D)\geq \gamma^*_1\}=\{T(\tilde D)\geq \gamma^*_1\} \quad\text{and}\quad \{T(D)<\gamma^*_1\} \subseteq \{T(D)=T(\tilde D)\} .
\end{equation}
Note that $T(\tilde D)$ is an $(\cF_{t+})$-stopping time, which will be used below.  Note also that \eqref{eq:sigma_tilde_D} does not imply that $D$ and $\tilde D$ coincide at or after $T(D)$.

As in step 2, it must be $\bar L^*_t=0$ for $t<T(D) \wedge \gamma^*_1$. Therefore, we have $D_t=\tilde D_t=0$ for $t<T(D) \wedge \gamma^*_1$. It follows that  $\{T(D)\geq \gamma^*_1\} \subseteq \{T(\tilde D)\geq \gamma^*_1\}$.
On the event $\{T(D)<\gamma^*_1\}$, we have $D_t=\tilde D_t=0$ on $[0,T(D))$, and thus $T(\tilde D) \geq T(D)$. Assume by contradiction that 
there is $(\omega,u_1)$ such that $T(D(\omega))<\gamma^*_1(u_1,\omega)$ and $T(\tilde D(\omega)) > T(D(\omega))$. We omit the explicit dependence on $(\omega,u_1)$ for simplicity. Let $t \in (T(D), T(\tilde D)\wedge \gamma^*_1)$.  Then the pair $(0,\tilde D)$ is solution of the fixed point equation \eqref{eq:defDdeviation} for all $s\in [0,t]$ which, by uniqueness of the solution, implies $D=\tilde D$ on $[0,t]$. Hence, a contradiction. We conclude that $\{T(D)<\gamma^*_1\} \subseteq \{T(D)=T(\tilde D)\}$. 
This last inclusion together with $\{T(D)\geq \gamma^*_1\} \subset \{T(\tilde D)\geq \gamma^*_1\}$ imply $\{T(D)\geq \gamma^*_1\}=\{T(\tilde D)\geq \gamma^*_1\}$ concluding the proof of \eqref{eq:sigma_tilde_D}.

Player 2's payoff reads
\begin{equation}\label{eq:P2J}
\begin{aligned}
\cJ^2_{x,x}(\Psi,\Xi^*)&=\cJ^2_{x,x}(D,\bar L^*) = \bar\E_{x,x}\Big[ \int_{[0,{\gamma_{X^*} \wedge \gamma_Y}]} \e^{-rt} \ud D_t +1_{\{\gamma_{X^*} < \gamma_Y\}}\e^{-r \gamma_{X^*}} \hat v(Y^D_{\gamma_{X^*}}) \Big] \\
&= \bar\E_{x,x} \Big[ 1_{\{D_{\gamma^*_1-}>0\}}\int_{[0,\gamma_Y]} \e^{-rt} \ud D_t \\
&\qquad +1_{\{D_{\gamma^*_1-}=0\}}\Big( \int_{[0,{\gamma_{X^*} \wedge \gamma_Y}]} \e^{-rt} \ud D_t +1_{\{\gamma_X^* < \gamma_Y\}}\e^{-r \gamma_X^*} \hat v(Y^D_{\gamma_X^*}) \Big) \Big],
\end{aligned}
\end{equation}
where in the final equality we use that$\{D_{\gamma^*_1-}>0\}\subset\{\gamma_Y\le \gamma_{X^*}\}$ by definition of $\Xi^*$ and \eqref{eq:defDdeviation}.

We now make two claims which we will prove separately after the end of this proof, for the ease of exposition. The claims are: 
\begin{equation}\label{eq:claim1}
\begin{aligned}
&\bar\E_{x,x}\Big[1_{\{D_{\gamma^*_1-}=0\}}\Big( \int_{[0,{\gamma_{X^*} \wedge \gamma_Y}]} \e^{-rt}\ud D_t  +1_{\{\gamma_{X^*} < \gamma_Y\}}\e^{-r \gamma_{X^*}} \hat v(Y^D_{\gamma_{X^*}})  \Big)\Big]\\
&\le \bar\E_{x,x}\Big[1_{\{D_{\gamma^*_1-}=0\}}1_{\{\gamma^*_1\le \gamma_{X^*} \wedge \gamma_Y\}} \e^{-r\gamma^*_1}v_2(a_0,X^0_{\gamma^*_1})\Big].
\end{aligned}
\end{equation}
and 
\begin{equation}\label{eq:claim2}
\begin{aligned}
&\bar\E_{x,x} \Big[ 1_{\{D_{\gamma^*_1-}>0\}}\int_{[0,\gamma_Y]} \e^{-rt} \ud D_t\Big]\le \bar\E_{x,x}\Big[1_{\{D_{\gamma^*_1-}>0\}\cap\{\gamma_{Y}\ge T(D)\}}\e^{-rT(D)}v_0\big(X^0_{T(D)}\big)\Big].
\end{aligned}
\end{equation}

We substitute \eqref{eq:claim1} and \eqref{eq:claim2} into \eqref{eq:P2J}. Then we use $\{D_{\gamma^*_1-}>0\}=\{T(D)<\gamma^*_1\}$ and that $\{\gamma_Y\ge \rho\}\subset\{\gamma_0\ge \rho\}$ for any $(\cG_t)$-stopping time $\rho$, along with \eqref{eq:sigma_tilde_D}, to obtain
\begin{equation*}
\begin{aligned}
&\cJ^2_{x,x}(D,\bar L^*) \\
&\leq \bar\E_{x,x} \Big[ 1_{\{D_{\gamma^*_1-}>0\}} 1_{\{\gamma_Y \geq T(D)\}} \e^{-r T(D)}v_0\big(X^0_{T(D)}\big)+  1_{\{D_{\gamma^*_1-}=0\}} 1_{\{\gamma_{X^*}\wedge \gamma_Y \geq \gamma^*_1\}} \e^{-r \gamma^*_1}v_2\big(a_0,X^0_{\gamma^*_1}\big) \Big] \\
&\leq\bar\E_{x,x} \Big[ 1_{\{T(D)<\gamma^*_1\}} 1_{\{\gamma_0 \geq T(D)\}} \e^{-r T(D)}v_0\big(X^0_{T(D)}\big)+1_{\{T(D) \geq \gamma^*_1\}} 1_{\{\gamma_0 \geq \gamma^*_1\}} \e^{-r \gamma^*_1}v_2\big(a_0,X^0_{\gamma^*_1}\big) \Big] \\
& =\bar\E_{x,x} \Big[ 1_{\{T(\tilde D)<\gamma^*_1\}} 1_{\{\gamma_0 \geq T(\tilde D)\}} \e^{-r T(\tilde D)}v_0\big(X^0_{T(\tilde D)}\big)+1_{\{T(\tilde D) \geq \gamma^*_1\}} 1_{\{\gamma_0 \geq \gamma^*_1\}} \e^{-r \gamma^*_1}v_2\big(a_0,X^0_{\gamma^*_1}\big) \Big]\\
& \leq \sup_{\tau}\bar\E_{x,x} \Big[ 1_{\{\tau<\gamma^*_1\}} 1_{\{\gamma_0 \geq \tau\}} \e^{-r \tau}v_0(X^0_{\tau})  +1_{\{\tau \geq \gamma^*_1\}} 1_{\{\gamma_0 \geq \gamma^*_1\}} \e^{-r \gamma^*_1}v_2(a_0,X^0_{\gamma^*_1}) \Big]=V(x),
\end{aligned}
\end{equation*}
where the supremum ranges over all stopping times of the filtration $(\cF_{t+})$ and we notice that 
\begin{equation}\label{eq:silly}
1_{\{\gamma_0 = \tau\}} v_0(X^0_{\tau})=1_{\{\gamma_0 = \tau\}} v_0(0)=0.
\end{equation} 
Thus, we have established \eqref{eq:Vineq}.
\medskip

\noindent{\bf Step 5}. In this step we show that $\cJ^2_{x,x}(\bar D^*,\bar L^*)=V(x)$ so that $\Xi^*_2$ is Player 2's best response against Player 1's strategy $\Xi^*_1$ by \eqref{eq:Vineq} and Lemma \ref{lem:deviations}-(ii). Thanks to the symmetry of the set-up, that will conclude the proof of the theorem and show that $(\Xi^*_1,\Xi^*_2)$ is a Nash equilibrium.

Our observation \eqref{eq:silly} and an application of It\^o's formula yield for any stopping time $\tau$ of the filtration $(\cF_{t+})$ (see \eqref{eq:D1} for the notation)
\begin{equation*}
\begin{aligned}
&\E_{x,x}\Big[\e^{-r\tau}v_0(X^0_{\tau})\big(1-\Gamma^*_{\tau}\big)1_{\{\tau<\gamma_0\}}\!+\!\int_0^{\tau\wedge\gamma_0}\!\!\e^{-rt}v_2(a_0,X^0_t)\ud \Gamma^*_t\Big]\\
&=\E_{x,x}\Big[\e^{-r(\tau\wedge\gamma_0)}v_0(X^0_{\tau\wedge\gamma_0})\big(1-\Gamma^*_{\tau\wedge\gamma_0}\big)\!+\!\int_0^{\tau\wedge\gamma_0}\!\!\e^{-rt}v_2(a_0,X^0_t)\ud \Gamma^*_t\Big]\\
&=v_0(x)+\E_{x,x}\Big[\int_0^{\tau\wedge\gamma_0}\e^{-rs}(1-\Gamma^*_s)\big(\cA v_0\big)(X^0_s)\ud s\Big]\\
&\quad+\E_{x,x}\Big[\int_0^{\tau\wedge\gamma_0}\e^{-rs-\int_0^s\ell^*(X^0_u)\ud u}\big(v_2(a_0,X^0_s)-v_0(X^0_s)\big)\ell^*(X^0_s)\ud s\Big].
\end{aligned}
\end{equation*}
By definition of $\Gamma^*$ (see \eqref{eq:Gamma}) and \eqref{eq:Lw} we see that 
\begin{equation*}
\begin{aligned}
(1-\Gamma^*_s)\big(\cA v_0\big)(X^0_s)&=-\e^{-\int_0^s\ell^*(X^0_u)\ud u}\big[rv_0(X^0_s)-\mu_0\big]^+\\
&=-\e^{-\int_0^s\ell^*(X^0_u)\ud u}\big(v_2(a_0,X^0_s)-v_0(X^0_s)\big)\ell^*(X^0_s).
\end{aligned}
\end{equation*}
Therefore, for any $\tau$
\begin{equation}\label{eq:indif}
\E_{x,x}\Big[\e^{-r\tau}v_0(X^0_{\tau})\big(1-\Gamma^*_{\tau}\big)1_{\{\tau<\gamma_0\}}\!+\!\int_0^{\tau\wedge\gamma_0}\!\!\e^{-rt}v_2(a_0,X^0_t)\ud \Gamma^*_t\Big]=v_0(x),
\end{equation}
which yields $V(x)=v_0(x)$ for all $x\in[0,\infty)$.

Starting from \eqref{eq:equilibrium_payoff} and integrating out the randomisation device $U_2$ of Player 2, we have
\begin{equation*}
\begin{aligned}
 &\cJ^2_{x,x}(\bar D^*,\bar L^*)\\ 
 &=\int_0^1 \bar\E_{x,x}\big[1_{\{\gamma^*_2(u)<\gamma^*_1 \wedge \gamma_0\}}\e^{-r \gamma^*_2(u)}v_0(X^0_{\gamma^*_2(u)})+1_{\{\gamma^*_1<\gamma^*_2(u) \wedge \gamma_0\}}\e^{-r \gamma^*_1}v_2(a_0, X^0_{\gamma^*_1})\big]\ud u\\
 &= \int_0^1 \E_{x,x}\Big[\e^{-r\gamma^*_2(u)}v_0(X^0_{\gamma^*_2(u)})\big(1-\Gamma^*_{\gamma^*_2(u)}\big)1_{\{\gamma^*_2(u)<\gamma_0\}}\!+\!\int_0^{\gamma^*_2(u)\wedge\gamma_0}\!\!\e^{-rt}v_2(a_0,X^0_t)\ud \Gamma^*_t\Big]\ud u=v_0(x),
\end{aligned}
\end{equation*}
where the second equality is due to \eqref{eq:VG} and the final one is by \eqref{eq:indif}. 
Therefore, $\cJ^2_{x,x}(\bar D^*,\bar L^*)=V(x)=v_0(x)$, as needed.
\end{proof}
\medskip

It remains to prove the formulae in \eqref{eq:claim1} and \eqref{eq:claim2}. 

\begin{proof}[Proof of \eqref{eq:claim1} and \eqref{eq:claim2}.]
Let us start with the proof of \eqref{eq:claim1}.
First we recall that by construction $v_2(x,y)=v_2(a_0,y)$ for $x\ge a_0$. Second we recall that $X^0_{\gamma^*_1}\ge a_0=X^{\bar L^*}_{\gamma^*_1}$ on $\{D_{\gamma_1^*-}=0\}$ by definition of $\bar L^*$ and $\gamma^*_1$ (c.f.\ \eqref{eq:meas}). 
Finally we recall that $v_2(x,y)=u_2(x,y-x)$ (see \eqref{eq:u2}) and that $Z^{\bar L^*,D}_t=Z^D_t=Y^D_t-X^{\bar L^*}_t$. Then, on $\{D_{\gamma_1^*-}=0\}$, we have
$\e^{-r\gamma^*_1}v_2(a_0,Y^D_{\gamma^*_1})=\e^{-r\gamma^*_1}v_2(X^*_{\gamma^*_1},Y^D_{\gamma^*_1})=\e^{-r\gamma^*_1}u_2 (X^*_{\gamma^*_1},Z^D_{\gamma^*_1})$.

For any $(\cG_t)$-stopping time $\rho \leq \gamma_{X^*}\wedge \gamma_Z$, It\^o's formula yields (c.f.\ \eqref{eq:D1} for the notation)
\begin{equation}\label{eq:ito0}
\begin{aligned}
\e^{-r\gamma^*_1}u_2(X^*_{\gamma^*_1},Z^D_{\gamma^*_1})&=\bar\E_{x,x}\Big[\e^{-r(\gamma^*_1\vee\rho)}u_2(X^*_{\gamma^*_1\vee\rho},Z^D_{\gamma^*_1\vee\rho})-\int_{\gamma^*_1}^{\gamma^*_1\vee\rho}\e^{-rs}(\cA u_2)(X^*_s,Z^D_s)\ud s\Big|\cG_{\gamma^*_1}\Big]\\
&\quad-\bar\E_{x,x}\Big[\int_{\gamma^*_1}^{\gamma^*_1\vee\rho}\e^{-rs}\big(\partial_z u_2-\partial_x u_2\big)(X^*_s,Z^D_s)\ud \bar L^{*,c}_s\Big|\cG_{\gamma^*_1}\Big]\\
&\quad+\bar\E_{x,x}\Big[\int_{\gamma^*_1}^{\gamma^*_1\vee\rho}\e^{-rs}\partial_z u_2(X^*_s,Z^D_s)\ud D^{c}_s\Big|\cG_{\gamma^*_1}\Big]\\
&\quad-\bar\E_{x,x}\Big[\sum_{s\in(\gamma^*_1,\gamma^*_1\vee\rho]}\e^{-rs}\big(u_2(X^*_s,Z^D_s)-u_2(X^*_{s-},Z^D_{s-})\big)\Big|\cG_{\gamma^*_1}\Big],
\end{aligned}
\end{equation}
where we removed the stochastic integral, which is a $(\cG_t)$-(local)martingale (standard localisation arguments may be used if needed).

Now we recall from Step 5 in the proof of Proposition \ref{prop:SC} that $(\cA u_2)(X^*_s,Z^D_s)\le 0$ for a.e.\ $s\ge 0$, $\partial_z u_2(X^*_s,Z^D_s)\ge 1$ for all $s\ge 0$. Moreover, on the event $\{D_{\gamma^*_1-}=0\}$ we have that $\bar L^*_s=\Phi_{\gamma^*_1}(s,B)$ for $s\ge \gamma^*_1$. That implies $\bar L^*_s=\bar L^{*,c}_s$ and $\ud \bar L^*_s=1_{\{X^*_s=a_0\}}\ud \bar L^*_s$ for $s>\gamma^*_1$. Thus, for every $s>\gamma^*_1$ we have 
\[
(\partial_z u_2-\partial_x u_2)(X^*_s,Z^D_s)\ud \bar L^{*,c}_s=(\partial_z u_2-\partial_x u_2)(a_0,Z^D_s)\ud \bar L^{*,c}_s=0,
\]
and, using $\partial_z u_2\ge 1$ yields
$u_2(X^*_s,Z^D_s)-u_2(X^*_{s-},Z^D_{s-})\big)=u_2(X^*_s,Z^D_s)-u_2(X^*_{s},Z^D_{s-})\big)\le -\Delta D_s$.
Combining these observations with \eqref{eq:ito0} and rewriting $u_2$ in terms of $v_2$ leads us to 
\begin{equation}\label{eq:ito1}
\begin{aligned}
&1_{\{D_{\gamma^*_1-}=0\}}\e^{-r\gamma^*_1}v_2(a_0,Y^D_{\gamma^*_1})= 1_{\{D_{\gamma^*_1-}=0\}}\e^{-r\gamma^*_1}v_2(X^*_{\gamma^*_1},Y^D_{\gamma^*_1})\\
&\ge 1_{\{D_{\gamma^*_1-}=0\}}\bar\E_{x,x}\Big[\e^{-r(\gamma^*_1\vee\rho)}v_2(X^*_{\gamma^*_1\vee\rho},Y^D_{\gamma^*_1\vee\rho})+\int_{(\gamma^*_1,\gamma^*_1\vee\rho]}\e^{-rs}\ud D_s\Big|\cG_{\gamma^*_1}\Big].
\end{aligned}
\end{equation}

Choose $\rho=\gamma_{X^*}\wedge\gamma_Y$. On the event $\{D_{\gamma^*_1-}=0\}\cap\{\gamma^*_1<\gamma_{X^*}\wedge\gamma_Y\}$ we have 
$v_2(X^*_{\gamma_{X^*}\wedge\gamma_Y},Y^D_{\gamma_{X^*}\wedge\gamma_Y})=1_{\{\gamma_{X^*}<\gamma_Y\}}v_2(X^*_{\gamma_{X^*}\wedge\gamma_Y},Y^D_{\gamma_{X^*}\wedge\gamma_Y})=1_{\{\gamma_{X^*}<\gamma_Y\}}\hat v(Y^D_{\gamma_{X^*}})$,
where the first equality holds because $v_2(x,0)=0$ and the second one because $v_2(0,y)=\hat v(y)$. For $\bar\omega\in\{D_{\gamma^*_1-}=0\}\cap\{\gamma^*_1=\gamma_{X^*}\wedge\gamma_Y\}$ the process $D$ has no jump at $\gamma^*_1$ because of the admissibility condition \eqref{eq:def0}.
Hence it must be also $\gamma_{X^*}(\bar \omega) = \gamma_Y(\bar \omega)=\gamma_0(\bar \omega)$ and $X^*_{\gamma_0}(\bar \omega)=Y^D_{\gamma_0}(\bar \omega)=X^0_{\gamma_0}(\bar \omega)=0$. Thus, on the event $\{D_{\gamma^*_1-}=0\}\cap\{\gamma^*_1= \gamma_{X^*}\wedge\gamma_Y\}$ we have 
\begin{equation*}
v_2(X^*_{\gamma^*_1},Y^D_{\gamma^*_1})+\int_{(\gamma^*_1,\gamma^*_1]}\e^{-rs}\ud D_s=v_2(0,0)=0=\hat v(Y^D_{\gamma^*_1}),
\end{equation*} 
where for the integral we use that $(\gamma^*_1,\gamma^*_1]=\varnothing$.

Since $\{\gamma^*_1\le \gamma_{X^*}\wedge\gamma_Y\}\in\cG_{\gamma^*_1}$, combining the observation above with \eqref{eq:ito1}, we conclude that 
\begin{equation}\label{eq:ito2}
\begin{aligned}
&\bar\E_{x,x}\Big[1_{\{D_{\gamma^*_1-}=0\}\cap\{\gamma^*_1\le \gamma_{X^*}\wedge\gamma_Y\}}\e^{-r\gamma^*_1}v_2(a_0,Y^D_{\gamma^*_1})\Big]\\
&\ge \bar\E_{x,x}\Big[1_{\{D_{\gamma^*_1-}=0\}\cap\{\gamma^*_1\le \gamma_{X^*}\wedge\gamma_Y\}}\Big(\int_{(\gamma^*_1,\gamma_{X^*}\wedge\gamma_Y]}\e^{-rs}\ud D_s+1_{\{\gamma_{X^*}<\gamma_Y\}}\e^{-r\gamma_{X^*}}\hat v(Y^D_{\gamma_{X^*}})\Big)\Big].
\end{aligned}
\end{equation}
Notice that 
$v_2(a_0,Y^D_{\gamma^*_1})-v_2(a_0,Y^D_{\gamma^*_1-})\le -\Delta D_{\gamma^*_1}$,
because $\partial_y v_2=\partial_z u_2\ge 1$. Moreover, on the event $\{D_{\gamma^*_1-}=0\}\cap\{\gamma^*_1\le \gamma_{X^*}\wedge\gamma_Y\}$ 
\[
\int_{(\gamma^*_1,\gamma_{X^*}\wedge\gamma_Y]}\e^{-rs}\ud D_s+\e^{-r\gamma^*_1}\Delta D_{\gamma^*_1}=\int_{[0,\gamma_{X^*}\wedge\gamma_Y]}\e^{-rs}\ud D_s.
\]
Then, adding on both sides of \eqref{eq:ito2} the quantity $\e^{-r\gamma^*_1}\Delta D_{\gamma^*_1}$
we obtain
\begin{equation}\label{eq:ito3}
\begin{aligned}
&\bar\E_{x,x}\Big[1_{\{D_{\gamma^*_1-}=0\}\cap\{\gamma^*_1\le \gamma_{X^*}\wedge\gamma_Y\}}\e^{-r\gamma^*_1}v_2(a_0,Y^D_{\gamma^*_1-})\Big]\\
&\ge  \bar\E_{x,x}\Big[1_{\{D_{\gamma^*_1-}=0\}\cap\{\gamma^*_1\le \gamma_{X^*}\wedge\gamma_Y\}}\Big(\int_{[0,\gamma_{X^*}\wedge\gamma_Y]}\e^{-rs}\ud D_s+1_{\{\gamma_{X^*}<\gamma_Y\}}\e^{-r\gamma_{X^*}}\hat v(Y^D_{\gamma_{X^*}})\Big)\Big].
\end{aligned}
\end{equation}
Finally, we notice that
\begin{equation*}
\begin{aligned}
&1_{\{D_{\gamma^*_1-}=0\}\cap\{\gamma^*_1> \gamma_{X^*}\wedge\gamma_Y\}}\Big(\int_{[0,\gamma_{X^*}\wedge\gamma_Y]}\e^{-rs}\ud D_s+1_{\{\gamma_{X^*}<\gamma_Y\}}\e^{-r\gamma_{X^*}}\hat v(Y^D_{\gamma_{X^*}})\Big)\\
&=1_{\{D_{\gamma^*_1-}=0\}\cap\{\gamma^*_1> \gamma_{X^*}\wedge\gamma_Y\}}1_{\{\gamma_{X^*}<\gamma_Y\}}\e^{-r\gamma_{X^*}}\hat v(X^0_{\gamma_{X^*}})=0,
\end{aligned}
\end{equation*}
where the final equality uses that for every $\bar \omega\in\{D_{\gamma^*_1-}=0\}\cap\{\gamma^*_1> \gamma_{X^*}\wedge\gamma_Y\}$ we have $D_t(\bar \omega)=\bar L^*_t(\bar \omega)=0$ for all $t\in[0,\gamma_{X^*}(\bar \omega)\wedge\gamma_Y(\bar \omega)]$, hence implying that $\gamma_{X^*}(\bar \omega)=\gamma_Y(\bar \omega)=\gamma_0(\bar \omega)$\footnote{Notice that $X^*_t(\bar \omega)=Y^D_t(\bar \omega)=X^0_t(\bar \omega)$ for $t\in[0,\gamma_{X^*}(\bar \omega)\wedge\gamma_Y(\bar \omega)]$. So even if we consider the event $\{\gamma_{X^*}\le \gamma_Y\}$ instead of $\{\gamma_{X^*}< \gamma_Y\}$ we obtain $1_{\{\gamma_{X^*}\le \gamma_Y\}}\e^{-r\gamma_{X^*}}\hat v(X^0_{\gamma_{X^*}})=1_{\{\gamma_{X^*}\le \gamma_Y\}}\e^{-r\gamma_{X^*}}\hat v(0)=0$.}. 
Combining with \eqref{eq:ito3} we arrive at  
\begin{equation*}
\begin{aligned}
&\bar\E_{x,x}\Big[1_{\{D_{\gamma^*_1-}=0\}\cap\{\gamma^*_1\le \gamma_{X^*}\wedge\gamma_Y\}}\e^{-r\gamma^*_1}v_2(a_0,Y^D_{\gamma^*_1-})\Big]\\
&\ge  \bar\E_{x,x}\Big[1_{\{D_{\gamma^*_1-}=0\}}\Big(\int_{[0,\gamma_{X^*}\wedge\gamma_Y]}\e^{-rs}\ud D_s+1_{\{\gamma_{X^*}<\gamma_Y\}}\e^{-r\gamma_{X^*}}\hat v(Y^D_{\gamma_{X^*}})\Big)\Big].
\end{aligned}
\end{equation*}
The expression in \eqref{eq:claim1} is finally obtained upon noticing that on $\{D_{\gamma^*_1-}=0\}$ we have $Y^D_{\gamma^*_1-}=X^0_{\gamma^*_1}$.

Now we prove \eqref{eq:claim2}. Recall $T(D)=\inf\{t\ge 0:D_t>0\}$. For any $(\cG_t)$-stopping time $\rho \leq \gamma_Y$, It\^o's formula yields
\begin{equation}\label{eq:ito4}
\begin{aligned}
&\e^{-rT(D)} v_0(Y^D_{T(D)})\\
&= \bar\E_{x,x}\Big[\e^{-r(\rho\vee T(D))}v_0\big(Y^D_{\rho\vee T(D)}\big)- \!\int_{T(D)}^{\rho\vee T(D)}\!\!\e^{-rs} \big(\cA v_0\big)(Y^D_s) \ud s\! +\!\int_{T(D)}^{\rho\vee T(D)}\!\! \e^{-rs}\partial_y v_0(Y^D_s)\ud D^c_s\\
&\quad - \sum_{s\in(T(D), T(D)\vee \rho]}\e^{-rs}\big(v_0(Y^D_s)-v_0(Y^D_{s-})\big)\Big|\cG_{T(D)}\Big], 
\end{aligned}
\end{equation}
where we removed the stochastic integral (using standard localisation if needed).
From \eqref{eq:fbpdiv} we know that for all $s\ge 0$, $\big(\cA v_0\big)(Y^D_s) \leq 0$, $\partial_y v_0(Y^D_s)\geq 1$ and $v_0(Y^D_s)-v_0(Y^D_{s-}) \leq \int_{Y^D_{s-}}^{Y^D_s} \partial_y v_0(u)du \leq -(D_s-D_{s-})$. 
Combining these with \eqref{eq:ito4} and taking $\rho=\gamma_Y$, we obtain
\begin{equation*}
\begin{aligned}
1_{\{D_{\gamma^*_1-}>0\}\cap\{\gamma_Y>T(D)\}}\e^{-rT(D)} v_0(Y^D_{T(D)})\ge 1_{\{D_{\gamma^*_1-}>0\}\cap\{\gamma_Y>T(D)\}}\bar\E_{x,x}\Big[\int_{(T(D),\gamma_Y]}\e^{-rs}\ud D_s\Big|\cG_{T(D)}\Big].
\end{aligned}
\end{equation*}

Since $v_0(Y^D_{T(D)})-v_0(Y^D_{T(D)-})\le -\Delta D_{T(D)}$, then adding $\e^{-rT(D)}\Delta D_{T(D)}$ on both sides in the expression above and recalling that $Y^D_t(\bar \omega)=X^0_t(\bar \omega)$ for $t\in[0,T(D)(\bar \omega))$, yields
\begin{equation}\label{eq:ito5}
\begin{aligned}
1_{\{D_{\gamma^*_1-}>0\}\cap\{\gamma_Y>T(D)\}}\e^{-rT(D)} v_0(X^0_{T(D)})&\ge 1_{\{D_{\gamma^*_1-}>0\}\cap\{\gamma_Y>T(D)\}}\bar\E_{x,x}\Big[\int_{[T(D),\gamma_Y]}\e^{-rs}\ud D_s\Big|\cG_{T(D)}\Big]\\
&= 1_{\{D_{\gamma^*_1-}>0\}\cap\{\gamma_Y>T(D)\}}\bar\E_{x,x}\Big[\int_{[0,\gamma_Y]}\e^{-rs}\ud D_s\Big|\cG_{T(D)}\Big].
\end{aligned}
\end{equation}

Now we notice that on $\{\gamma_Y= T(D)\}$ we have $Y^D_{T(D)}=Y^D_{\gamma_Y}=0$ and $Y^D_{T(D)-}=X^0_{T(D)}$. It follows that $\Delta D_{T(D)}=X^0_{T(D)}$ and $v_0(Y^D_{T(D)})-v_0(Y^D_{T(D)-})=v_0(0)-v_0(X^0_{T(D)})=-v_0(X^0_{T(D)})$. 
Since $v_0'\geq 1$ on $[0,\infty)$, then
\[
-v_0(X^0_{T(D)})=v_0(Y^D_{T(D)})-v_0(Y^D_{T(D)-})=-\int_{0}^{\Delta D_{T(D)}}v'_0(Y^D_{T(D)-}-y)\ud y\le -\Delta D_{T(D)}.
\] 
We deduce that  on the event $\{D_{\gamma^*_1-}>0\}\cap\{\gamma_Y= T(D)\}$
\[
\int_{[0,\gamma_Y]}\e^{-rs}\ud D_s=e^{-r T(D)}\Delta D_{T(D)} \leq e^{-r T(D)}v_0(X^0_{T(D)}).
\]
Notice also that $\int_{[0,\gamma_Y]}\e^{-rs}\ud D_s=0$ on the event $\{D_{\gamma^*_1-}>0\}\cap\{\gamma_Y< T(D)\}$.

Taking expectation in \eqref{eq:ito5} and using these observations we can conclude
\begin{equation*}
\bar\E_{x,x}\Big[ 1_{\{ D_{\gamma^*_1-}>0\}}\int_{[0, \gamma_Y]}\e^{-rs}\ud D_s\Big] \leq \bar\E_{x,x} \Big[1_{\{ D_{\gamma^*_1-}>0\}\cap\{\gamma_Y \geq T(D)\}} \e^{-r T(D)}v_0\big(X^0_{T(D)}\big)\Big],
\end{equation*}
as claimed in \eqref{eq:claim2}.
\end{proof}

\section{Some concluding remarks.}\label{sec:conclusions}

The model we studied in this paper is certainly a very stylized formulation of a complex economic problem. The assumption that both firms' capital evolves as an aBm with the same drift and diffusion coefficients is perhaps restrictive. From an economic perspective we interpret that assumption as saying that the two firms are equally efficient, they produce the same good and they operate on the same market. On that market, customers perceive the products manufactured by the two firms as completely equivalent and so they are equally likely to buy either of the two. That determines identical profit streams for the two firms. The same type of one-dimensional model with one exogenous shock variable has already been studied for example in exit decision problems for a duopoly (see \cite{Lambrecht}, \cite{Murto}) or more generally  in two-player continuous-time nonzero-sum stopping games modeling a war of attrition with symmetric information (see \cite{DGM}).

Despite the simplifications in the model, the mathematical complexity is significant. Yet, we are able to obtain an equilibrium with {\em explicit} expressions for players' payoffs and strategies. The crucial observation in our solution method is that when we look at the problem in the $(X,Z)$-coordinates, the process $Z$ is purely controlled with no exogenous component in its dynamics (i.e., no drift or diffusion). Similar structures are well-known in single-agent, singular control problems (e.g., \cite{FP} and \cite{MZ}). When translated into free boundary problems (e.g., in Proposition \ref{prop:SC}) this feature allows us to treat the $z$-variable (almost) as a parameter. Indeed, in \eqref{eq:fbp}, the differential operator associated to the dynamics $(X,Z)$ only involves derivatives in the $x$-variable, while $z$-derivatives only appear in some boundary conditions. That enables the explicit solution of a family of ordinary differential equations (ODEs), parametrized by $z$, via the general theory of fundamental/particular solutions. 

If instead the drift coefficients in the dynamics $(X,Y)$ were different (say $\mu^X_0\neq\mu^Y_0$), then the dynamics of $Z$ would acquire a drift $\mu^Z_0\coloneqq\mu^Y_0-\mu^X_0$. Then, the differential operator in \eqref{eq:fbp} would include a term $\mu^Z_0\partial_z$ and the free boundary problem would be in the form a true partial differential equation (PDE). Similarly, if the diffusion coefficients in the dynamics $(X,Y)$ were different, or if the two processes were driven by two (possibly correlated) Brownian motions, then the dynamics of $Z$ would acquire a diffusive term. Again, at the level of the free boundary problem we would have a PDE with a differential operator including second order partial derivatives in the $x$- and $z$-coordinates. Explicit solutions of free boundary problems with PDEs are fundamentally out of reach. Even if we could prove existence of a solution to the PDE analogue of \eqref{eq:fbp}, such solution's structure would be abstract and extremely complex. Then, finding a best response map for Player 1 (i.e., the analogue of Lemma \ref{lem:Phi}) appears a phenomenally challenging task. 

The discussion in the paragraph above shows that completely different methods would be needed to solve the problem at hand when drift and/or diffusion coefficients in the dynamics $(X,Y)$ are not the same (and/or two Brownian motions drive the dynamics). An approach with variational methods would require the solution of a complicated system of free boundary problems with Neumann-type boundary conditions at the (unknown) free boundaries and Dirichlet conditions at the default boundaries $\{0\}\times(0,\infty)$ and $(0,\infty)\times\{0\}$. We are not aware of any instance of such problems being studied in the literature. Alternatively, it would be tempting to study the game with methods from Backward Stochastic Differential Equations. As far as we know, a study of BSDEs associated to games of singular controls with absorbing states (defaults) is not available.

\section{Appendix.}

In this short appendix we recall a simple useful lemma (see, e.g., \cite[Lem.\ 4.4]{PK22}).
\begin{lemma}\label{lem:BV}
Let $(\nu_t)_{t\ge 0}$ be a c\`adl\`ag process of bounded variation and let $(M_t)_{t\ge 0}$ be a continuous semimartingale. Assume there is a positive, (locally) integrable process $(m_t)_{t\ge 0}$ such that 
$\langle M\rangle_t=\int_0^t m_s\ud s$ 
and $m_t\ge \eps$ for all $t\ge 0$, for some $\eps>0$.
Then, 
$\E\left[\int_0^S 1_{\{M_s=\nu_s\}}\ud s\right]=0$ for any $S>0$.
\end{lemma}

\begin{proof}
Set $N:=M-\nu$ and let $h_\delta(z):=1_{(-\delta,\delta)}(z)$. By the occupation time formula (\cite[Thm.\ IV.45.1]{RW})
$\int_0^S h_\delta(N_s)\ud \langle N \rangle_s=\int_\R h_\delta(z)\ell^{z}_S\ud z=\int_{-\delta}^{\delta}\ell^z_S\ud z$, $\P-a.s.$
where $(\ell^z_t)_{t\ge 0}$ is the local time at $z\in\R$ of the process $N$. The left-hand side of the expression above is finite and therefore $\ell^z_S<\infty$, $\P$-a.s., for a.e.\ $z\in(-\delta,\delta)$. Letting $\delta\downarrow 0$, using the dominated convergence theorem on both sides of the expression, we obtain
\begin{equation*}
0=\int_0^S 1_{\{N_s=0\}}\ud \langle N \rangle_s=\int_0^S 1_{\{M_s=\nu_s\}}\ud \langle M \rangle_s=\int_0^S 1_{\{M_s=\nu_s\}}m_s\ud s\ge \eps\int_0^S 1_{\{M_s=\nu_s\}}\ud s,\quad \P-a.s.,
\end{equation*}
where for the second equality we recall that $\nu$ is of bounded variation. This yields the lemma.
\end{proof}

\section{Notations.}\label{sec:notation}

We summarize here the main notations for the readers' convenience.
\begin{itemize}
\item $\Omega = C_0([0,\infty))$: continuous functions $\varphi:[0,\infty)\to \R$ with $\varphi(0)=0$; 
\item $\P$ is the Wiener measure on the Borel $\sigma$-algebra of $\Omega$;
\item $D^+_0([0,\infty))$: right-continuous, non-decreasing functions $\zeta:[0,\infty)\to \R$ with $\zeta(0-)=0$;
\item ${\mathbb W}_t(\varphi,\zeta):=(\varphi(t),\zeta(t))$ for $t\in[0,\infty)$: coordinate mapping on the canonical space $C_0([0,\infty))\times D_0^+([0,\infty))$. Its raw filtration is denoted $(\cF^{\mathbb W}_t)_{t\ge 0}$.
\item $(B_t)_{t\ge 0}$: standard, $1$-dimensional Brownian motion (notice that $B_t(\omega)=\omega(t)$ for $\omega\in\Omega$); 
\item $\cF_t=\sigma(B_s,s\le t)$ is the {\em raw} Brownian filtration;
\item $\hat v(x)$, $\hat a$, $\hat \xi_t$: value function, optimal boundary and optimal control of dividend problem with drift $\hat\mu$;
\item $v_0(x)$, $a_0$, $\xi^0_t$: value function, optimal boundary and optimal control of dividend problem with drift $\mu_0$; 
\item $\Sigma_R(x)$ (resp.\ $\Sigma(x)$): randomised (resp.\ pure) strategies with state dynamics starting from $x$ at time zero;
\item $\cD_R(x)$ (resp.\ $\cD(x)$): randomised (resp.\ pure) controls with state dynamics starting from $x$ at time zero;
\item $\Theta_R(x,y)$: control-inducing pairs with state dynamics starting from $(x,y)$ at time zero;
\item $b(x)$: optimal boundary for the firm with larger initial endowment; $\alpha=b(a_0)$;
\item $H=\{(x,z)\in[0,a_0]\times\R| z\ge -x\}$, $H_\le=\{(x,z)\in H|z\le 0\}$, $H_{[0,\alpha]}=\{(x,z)\in H|z\in[0,\alpha]\}$, $H_{[\alpha,b]}=\{(x,z)\in H|\alpha\le z\le b(x)\}$;
\item $\beta_2<0<\beta_1$: solutions of $\sigma^2\beta^2+2\mu_0\beta-2r=0$;
\item $(\bar \Omega, \bar \cF, \bar \P):=(\Omega\times [0,1]\times [0,1], \cF\otimes \cB([0,1])\otimes \cB([0,1]), \P\otimes \lambda \otimes \lambda)$;
\item For $\zeta\in D^+_0([0,\infty))$ we set $T(\zeta)=\inf \{ t \geq 0 : \zeta_t >0\}$.
\end{itemize} 

\medskip

\noindent{\bf Acknowledgements}: T.\ De Angelis was partially supported by PRIN2022 (project ID: BEMMLZ) {\em Stochastic control and games and the role of information}. Part of the research was conducted while S.\ Villeneuve was visiting Collegio Carlo Alberto in Torino, under a fellowship granted by LTI@UniTO. F.\ Gensbittel acknowledges financial support from ANR (Programmes d'Investissements d'Avenir CHESS ANR-17-EURE-0010 and ANITI ANR-19-PI3A-0004).


\begin{thebibliography}{99}

\bibitem{Aumann}
\textsc{Aumann, R.J.} (1964). Mixed and Behavior Strategies in Infinite Extensive Games.  {\em Advances in Game Theory}, Annals of Mathematics Study, Vol. 52, ed. by M. Dresher, L.S. Shapley, and A.W. Tucker. Princeton: Princeton University Press, 627--650.

\bibitem{Avanzi}
\textsc{Avanzi, B.} (2009). Strategies for dividend distribution: A review. {\em North Am.\ Actuarial J.}, {\bf 13} (2), 217-251.

\bibitem{BDeA}
\textsc{Bandini, E., De Angelis, T., Ferrari, G., Gozzi, F.} (2022). Optimal dividend payout under stochastic discounting. {\em Math.\ Finance}, {\bf 32} (2), 627-677.

\bibitem{BP}
\textsc{Back, K., Paulsen, D.} (2009). Open loop equilibria and perfect competition in option exercise games. {\em Rev.\ Financ.\ Stud.} {\bf 22}, 4531–4552. 

\bibitem{BaCh} \textsc{Bather, J., Chernoff H.} (1967). Sequential decisions in the control of a spaceship. {\it Fifth Berkeley Sympos. Math. Statist. Probab.} {\bf 3}, 181-207.

\bibitem{Be} \textsc{Benes, V. E., Shepp, L. A., Witsenhausen, H. S.} (1980). Some solvable stochastic control problems. {\it Stochastics}, {\bf4}(1), 39-83.

\bibitem{Cai}
\textsc{Cai, C., De Angelis, T.} (2023). A change of variable formula with applications to multi-dimensional optimal stopping problems. {\em Stoch.\ Process.\ Appl.}, {\bf 164}, 33-61. 

\bibitem{DGM}
\textsc{D\'{e}camps, J.-P., Gensbittel F., Mariotti T.} (2022). The War of Attrition under Uncertainty: Theory and Robust Testable Implications. {\it TSE Working Papers 22-1374}, Toulouse School of Economics (TSE), revised Jun 2024.

%\bibitem{DeAM23}
%\textsc{De Angelis, T., Milazzo, A.} (2022). Dynamic programming principle for classical and singular control with discretionary stopping. {\em arXiv}: 2111.09608. Extended version of {\em Appl.\ Math.\ Optim.} (2023), {\bf 88} (7).

\bibitem{DeAFe2} \textsc{De Angelis, T., Ferrari, G.} (2017). Stochastic Nonzero-Sum Games: A new Connection between Singular Control and Optimal Stopping. 
{\it Adv. Appl. Probab.} {\bf 50}(2), 347-372. 

\bibitem{DeFinetti} \textsc{De Finetti, B.} (1957). Su un'impostazione alternativa della teoria colletiva del rischio. {\em Trans. 15th Int. Congress of Actuaries}, {\bf 2}, 433-443.

\bibitem{DRV} \textsc{Dammann, F., Rodosthenous N. and Villeneuve S.} (2023). A Stochastic Non-zero sum game of controlling the debt-to-GDP ratio, TSE working paper, April 2023.

\bibitem{EL} \textsc{Ekstr\"om, E., Lindensj\"o, K.} (2023). De Finetti’s Control Problem with Competition. {\em Appl. Math. Optim.}, {\bf 87} (16). 

\bibitem{FP}
\textsc{Federico, S., Pham, H.} (2014). Characterization of the optimal boundaries in reversible investment problems. {\em SIAM J.\ Control Optim.}, {\bf 52} (4), 2180-2223.

\bibitem{Gr}
\textsc{Grenadier, S.R.} (2002). Option exercise games: an application to the equilibrium investment strategies of firms. {\em Rev.\ Financ.\ Stud.} {\bf 15}, 691–721. 

\bibitem{HWW}
\textsc{Hendricks, K., Weiss A., Wilson C.} (1988). The War of Attrition in Continuous Time with Complete Information. {\em Int. Econ. Rev.}, 29 (4), 663-680.

\bibitem{JS} \textsc{Jeanblanc-Picqu\'e, M., Shiryaev, A.N.} (1995). Optimization of the flow of dividends, {\em Russian Mathematics Surveys}, {\bf 50}, 257-277.

\bibitem{Ka} \textsc{Karatzas, I.} (1981). The monotone follower problem in stochastic decision theory. {\it Appl. Math. Optim.}, {\bf 7} (1), 175-189.

\bibitem{KS} \textsc{Karatzas, I., Shreve, S.E.} (1988). {\em Brownian motion and stochastic calculus}, Springer-Verlag, New York.

\bibitem{Kobayashi} \textsc{Kobayashi, B.} (2010). {\em The Law and Economics of Predatory Pricing}, in
Keith N. Hylton (ed.), Antitrust Law and Economics, Edward Elgar Publishing.

\bibitem{Kw} \textsc{Kwon, H.D., Zhang, H.} (2015). Game of singular stochastic control and strategic exit. {\it Math. Oper. Res.} {\bf 40}, 869-887.

\bibitem{Lambrecht}
\textsc{Lambrecht, B.} (2001). The Impact of Debt Financing on Entry and Exit in a Duopoly. {\it Rev.\ Financial Stud.} {\bf 14} (2), 765-804.

\bibitem{MZ} \textsc{Merhi, A., Zervos, M.} (2007). A model for reversible investment capacity expansion. {\em SIAM J.\ Control Optim.}, {\bf 46} (3), 839-876.

\bibitem{Murto}
\textsc{Murto, P.} (2004). Exit in Duopoly under Uncertainty. {\it The RAND Journal of Economics}, {\bf 35} (1), 111-127. 

\bibitem{Neyman}
\textsc{Neyman, A.} (2017). Continuous-time stochastic games. {\it Games Econom.\ Behav.},
{\bf 104}, 92-130.

\bibitem{PK22}
\textsc{Kwon, H.D., Palczewski, J.} (2024). Exit game with private information. {\em To appear in Math.\ Oper.\ Res.} ({\em arXiv}: 2210.01610).

\bibitem{PTZ} \textsc{Possama\"i, D., Touzi, N., Zhang, J.} (2020). Zero-sum path-dependent stochastic differential games in weak formulation. {\it Ann.\ Appl.\ Probab.}, {\bf 30} (3), 1415-1457.

\bibitem{RS} \textsc{Radner, R., Shepp, L.} (1996), Risk versus profit potential: A model for corporate strategy,
{\em J. Econ. Dyn. Control}, {\bf 20}, 1373-1393.

\bibitem{RW}
\textsc{Rogers, L.C.G., Williams, D.} {\em Diffusions, Markov processes and martingales}, 2nd Edition (Volume 2). Cambridge University Press, Cambridge, 2000.

\bibitem{Schmidli} \textsc{Schmidli, H}. {\em Stochastic Control in Insurance}. Springer-Verlag, London, 2008.

\bibitem{Steg-a}
\textsc{Steg, J.H.} (2012). Irreversible investment in oligopoly. {\em Finance Stoch.}, {\bf 16} (2), 207-224.

\bibitem{Steg}
\textsc{Steg, J.-H.} (2015). Symmetric Equilibria in Stochastic Timing Games. \em{Center for Mathematical Economics Working Paper} No. 543, Universit\"at Bielefeld.

\end{thebibliography}
\end{document}